\documentclass[reqno]{amsart}

\usepackage{amsmath,amssymb,amsthm,amsfonts}
\usepackage{color}
\usepackage{hyperref}
\usepackage{a4wide}

\numberwithin{equation}{section}

\newtheorem{theorem}{Theorem}[section]
\newtheorem{lemma}[theorem]{Lemma}
\newtheorem{prop}[theorem] {Proposition}
\newtheorem{cor}[theorem]  {Corollary}

\theoremstyle{definition}
\newtheorem{assumption}{Assumption}

\theoremstyle{remark}
\newtheorem*{remark}{Remark}

\newcommand{\N}{\mathbb{N}}
\newcommand{\R}{\mathbb{R}}
\newcommand{\Z}{\mathbb{Z}}

\newcommand{\Q}{\mathbb{Q}}
\renewcommand{\P}{\mathbb{P}}
\newcommand{\E}{\mathbb{E}}
\newcommand{\dd}{\mathrm{d}} 
\newcommand{\eps}{\varepsilon}
\newcommand{\la}{\langle}
\newcommand{\ra}{\rangle}

\newcommand{\vect}[1]{\boldsymbol{#1}}

\DeclareMathOperator{\var}{var}  

\newcommand{\be}{\begin{equation}}
\newcommand{\ee}{\end{equation}}
\newcommand{\ba}{\begin{equation} \begin{aligned}}
\newcommand{\ea}{\end{aligned}\end{equation}}
\newcommand{\bes}{\begin{equation*}}
\newcommand{\ees}{\end{equation*}}

\newcommand{\ssup}[1] {{\scriptscriptstyle{({#1}})}}
\def\1{{\mathchoice {1\mskip-4mu\mathrm l}      
{1\mskip-4mu\mathrm l}
{1\mskip-4.5mu\mathrm l} {1\mskip-5mu\mathrm l}}}
\newcommand{\e}   {{\operatorname e }}
\newcommand{\heap}[2]{\genfrac{}{}{0pt}{}{#1}{#2}}

\begin{document}

\title[Boundary layers for a chain of atoms at low temperature]{Surface energy and boundary layers for a chain of atoms at low temperature}
\author{Sabine Jansen}
\address{Mathematisches Institut, Ludwig-Maximilians-Universit{\"a}t, Theresienstr. 39, 80333 M{\"u}nchen, Germany}
\email{jansen@math.lmu.de}
\author{Wolfgang K{\"o}nig}
\address{Weierstrass Institute Berlin, Mohrenstr. 39, 10117 Berlin and Technische Universit{\"a}t Berlin, Str. des 17. Juni 136, 10623 Berlin, Germany}
\email{koenig@wias-berlin.de}
\author{Bernd Schmidt}
\address{Institut f{\"u}r Mathematik, Universit{\"a}t Augsburg, Universit{\"a}tsstr. 14, 86159 Augsburg, Germany}
\email{bernd.schmidt@math.uni-augsburg.de}
\author{Florian Theil}
\address{Mathematics Institute, University of Warwick, Coventry, CV4 7AL, UK}
\email{F.Theil@warwick.ac.uk}

\date{April 12, 2019}
\maketitle
\begin{abstract}We analyze the surface energy and boundary layers for a chain of atoms at low temperature for an interaction potential of Lennard-Jones type. The pressure (stress) is assumed small but positive and bounded away from zero, while the temperature $\beta^{-1}$ goes to zero. Our main results are: (1)  As $\beta \to \infty$ at fixed positive pressure $p>0$, the Gibbs measures $\mu_\beta$ and $\nu_\beta$ for infinite chains and semi-infinite chains satisfy path large deviations principles. The rate functions are bulk and surface energy functionals $\overline{\mathcal{E}}_{\mathrm{bulk}}$ and $\overline{\mathcal{E}}_\mathrm{surf}$. The minimizer of the surface functional corresponds to zero temperature boundary layers. (2) The surface correction to the Gibbs free energy converges to the zero temperature surface energy, characterized with the help of the minimum of $\overline{\mathcal{E}}_\mathrm{surf}$. (3)  The bulk Gibbs measure and Gibbs free energy can be approximated by their Gaussian counterparts. (4) Bounds on the decay of correlations are provided, some of them uniform in $\beta$.  \\


\end{abstract}

\tableofcontents

\section{Introduction} 

The purpose of the present article is to analyze the low-temperature behavior for a one-dimensional chain of atoms that interact via a Lennard-Jones type potential. The model is atomistic and in terms of the Gibbs measures of classical statistical mechanics. Two limiting procedures are at  play: the zero-temperature limit, for which the inverse temperature $\beta$ goes to infinity,  and the thermodynamic limit, where the number of particles $N$ and the system size go to infinity. The order of the limits matters. When the zero-temperature limit is taken before the $N\to \infty$ limit, the analysis of Gibbs measures is replaced by energy minimization, leading to variational models of non-linear elasticity. 
We perform instead the zero-temperature limit after the thermodynamic limit.
The zero-temperature limit for infinite systems is far from trivial, see~\cite{vanenter-ruszel07,chazottes-gambaudo-ugalde11,chazottes-hochman10} and the discussion in~\cite{bellissard-radin-shlosman10}. 

For the one-dimensional Lennard-Jones interaction, it is known that energy minimizers (ground states) converge to a periodic lattice~\cite{gardner-radin79} (``crystallization''). For one-dimensional systems with pair potentials that decay faster than $1/r^2$ it is well-known that, in contrast, at positive temperature,  no matter how small, there is no crystallization~\cite{blanc-lewin15}. Nevertheless, some quantities can be approximated well by their zero-temperature counterpart. For the bulk free energy this is to be expected, for other quantities such as surface corrections this is already more subtle. For the decay of correlations, it is a priori not even clear what the zero-temperature counterpart should be; we propose a natural candidate, see Eqs~\eqref{eq:effective-int} and~\eqref{eq:effective-zero}. 

At zero temperature, surface corrections and boundary layers have been studied, for example, in order to better understand variational models of fracture, see e.g.~\cite{braides-cicalese07, scardia-schloem-zanini11} and the references therein. 
Fracture might be expected for elongated chains, forced to stretch beyond their preferred length. At small positive temperature, large interparticle distances correspond to low pressure (stress) $p=p_\beta \to 0$. We address this regime in a subsequent work and focus here on the elastic regime of positive pressure $p>0$, though the case of small pressure $p_\beta \to 0$ is discussed in some comments. 

Our main results come in four parts. They are listed in Sections~\ref{sec:zerores}--\ref{sec:corres} and proven in Sections~\ref{sec:energy}--\ref{sec:brascamp}. 
At zero temperature, we extend the result on bulk periodicity from~\cite{gardner-radin79} to a more general class of potentials and positive pressure, see Theorem~\ref{thm:periodic}. We prove the existence of bounded surface corrections, and characterize them with the help of an energy functional $\mathcal{E}_\mathrm{surf}$ for semi-infinite chains (Theorem~\ref{thm:surface}). 

At positive temperature, we prove large deviations principles for the Gibbs measures $\mu_\beta$ and $\nu_\beta$ on $\R_+^\Z$ and $\R_+^\N$ (product topology) as $\beta \to \infty$ at fixed $p>0$ (Theorem~\ref{thm:ldp}).  The speed is $\beta$ and the respective rate functions are energy functionals $\overline{\mathcal{E}}_\mathrm{bulk}$ and $\overline{\mathcal{E}}_\mathrm{surf} - \min \overline{\mathcal{E}}_\mathrm{surf}$ whose minimizers are, respectively, the periodic bulk ground state and the zero-temperature boundary layer. The convergence of positive-temperature surface corrections to their zero-temperature counterpart is addressed in Theorem~\ref{thm:geld}. These results are intimately related to path large deviations for Markov processes and Hamilton-Jacobi-Bellman equations~\cite{feng-kurtz-book}, semi-classical analysis~\cite{helffer-book}, and a more direct approach to low-temperature expansions~\cite{shapeev-luskin17}. We remark that our results are valid for long range interactions which in particular are not assumed to have superlinear growth at infinity. 
The large deviations principle is complemented by a result on Gaussian approximations for the bulk Gibbs measure and the Gibbs free energy, valid for finite interaction range $m$ (Theorems~\ref{thm:gaussian2} and~\ref{thm:gaussian3}). 

Finally we study the temperature-dependence of correlations and informally discuss how correlations connect with effective interactions of defects and the decay of boundary layers. Theorem~\ref{thm:corrgen} provides a priori estimates that hold for all $\beta,p>0$. In Theorem~\ref{thm:corr-finitem} we show that for finite $m$ and small positive pressure $p$, the decay of correlations is exponential with a rate of decay that stays bounded as $\beta\to \infty$---the associated Markov chain has a spectral gap bounded away from zero. This uniform estimate is proven with perturbation theory for the transfer operator. For infinite $m$, we provide instead a uniform estimate for restricted Gibbs measures (Proposition~\ref{prop:restrict-alg}),  which follows from the convexity of the energy (in a neighborhood of the periodic gound state) and techniques from the realm of Brascamp-Lieb inequalities~\cite{helffer-book}. 
At vanishing pressure $p_\beta \to 0$ or fixed high pressure $p>0$, the spectral gap might become exponentially small because of fracture or metastable wells~\cite{bovier-denholla} in non-convex energy landscapes. 

Bringing statistical mechanics into atomistic models of crystals and elasticity has a rich tradition~\cite{born-huang-book,weiner-book,bavaud-choquard-fontaine86,penrose02}. Modern developments include: the study of gradient Gibbs measures~\cite{funaki-spohn97} with sophisticated tools such as renormalization groups and cluster expansions~\cite{adams-kotecky-mueller16}, random walk representations~\cite{brydges-frohlich-spencer82}, and Witten Laplacians~\cite{helffer-book}; scaling limits and gradient Young-Gibbs measures~\cite{presutti-book,kotecky-luckhaus14,runa16}; the 
extension of approximation schemes, e.g., the quasi-continuum method, to positive temperature \cite{blanc-lebris-legoll-patz10,tadmor-miller-book}.
In addition, there have been some inroads into the open problem of proving crystallization in the form of orientational order for two-dimensional models~\cite{aumann15,heydenreich-merkl-rolles14}.

To the best of our knowledge, all of the aforementioned mathematical literature, notably on Gibbs gradient measures, is limited to potentials with a superlinear growth at infinity. This is in stark contrast with the decay to zero typically imposed in statistical mechanics of point particles~\cite{ruelle-book69}.  We work with potentials $v(r)\to 0$,  an additional linear term $pr$ enters because we work in the constant pressure ensemble, which is the most convenient ensemble for one-dimensional systems~\cite[Section 5.6.6]{ruelle-book69}. As a consequence, the by now classical combination of Bakry-{\'E}mery estimates and Holley-Stroock perturbation principle, see~\cite{menz14} and the references therein, becomes potentially more delicate. We use instead  estimates on energy penalties, some aspects of which might generalize to higher-dimensional models. 

Another aspect that might generalize to higher dimension concerns the large deviations principle. 
The existence of a large deviations principle for the Gibbs measure as $\beta \to \infty$, proven using a exponential tightness and fixed point equation for the measure, amounts to the construction of an infinite volume energy functional that vanishes on ground states only. In higher dimension, the role of the fixed point equation is taken by  DLR-conditions named after Dobrushin, Lanford, Ruelle~\cite{georgii-book} and the proof of a large deviations principle  reduces to the investigation of a higher-dimensional analogue of a Bellman equation. The theory of the latter, for non-unique ground states, might mirror possible intricacies of the zero-temperature limit of Gibbs measure described in~\cite{vanenter-ruszel07}.

Finally we remark that the results of this work allow for a detailed analysis of typical atomic configurations at low temperature and low density. In \cite{jkst19} we will in particular prove that, when the density is strictly smaller than the density of the ground state lattice, a system with $N$ particles fills space by alternating approximately crystalline domains (``clusters'') with empty domains (``cracks''). The number of domains is of the order of $N \exp(-\beta e_{\rm surf}/2)$ with $e_{\rm surf}$ the surface energy from Theorem~\ref{thm:surface} below.

\section{Main results}\label{sec:results}

\subsection{Zero temperature} \label{sec:zerores}

Let $v:(0,\infty) \to \R$ be a pair potential, $m\in \N\cup \{\infty\}$ a truncation parameter and $p\geq 0$ the pressure. At zero temperature we allow for $p=0$, at positive temperature we impose $p>0$. The Gibbs energy at zero temperature and pressure $p$ for a system of $N$ particles with 
 positions $x_1<\ldots<x_N$ and interparticle spacings $z_j = x_{j+1}- x_j$, $j=1,\ldots,N-1$, is 
\bes
	\mathcal{E}_N(z_1,\ldots,z_{N-1}) = \sum_{\heap{1\leq i <j\leq N}{|i-j|\leq m}} v(z_i+\cdots + z_{j-1}) + p \sum_{j=1}^{N-1} z_j. 
\ees
The parameter $m$ restricts the range of the interaction: $m=2$ corresponds to a next-nearest neighbor interaction. This section deals with the minimization problem
\bes
	E_N = \inf_{z_1,\ldots,z_{N-1}>0}\,  \mathcal{E}_N(z_1,\ldots,z_{N-1})
\ees
in the limit $N\to \infty$. 
Throughout we assume that the following assumption holds.

\begin{assumption} \label{assu:v}
The pair potential $v : (0, \infty) \to \R \cup \{+ \infty\}$ is equal to $+ \infty$ on $(0, r_{\rm hc}]$ for some $r_{\rm hc} \ge 0$ and a $C^2$ function on $(r_{\rm hc}, \infty)$. There exist $r_{\rm hc} < z_{\min} < z_{\max} < 2 z_{\min}$ and $\alpha_1, \alpha_2 > 0$, $s > 2$ such that the following holds. 
\begin{itemize}
\item[(i)] Shape of $v$: $z_{\max}$ is the unique minimizer of $v$ and satisfies $v(z_{\max})<0$.  $v$ is decreasing on $(0,z_{\max})$ and  increasing and non-positive on $(z_{\max},\infty)$. 

\item[(ii)] Growth of $v$: $v(z) \ge - \alpha_1 z^{-s}$ for all $z > 0$ and $v(z) + v(z_{\max}) - 2 \alpha_1 \sum_{n=2}^\infty (n z)^{-s} > 0$ for all $z < z_{\min}$.  

\item[(iii)] Shape of $v''$: $v''$ is decreasing on $[z_{\min}, z_{\max}]$ and increasing and non-positive on $[2 z_{\min}, \infty)$.

\item[(iv)] Growth of $v''$: $v''(z) \ge -\alpha_2 z^{-s-2}$ for all $z > r_{\rm hc}$ and $v''(z_{\max}) + \sum_{n=2}^\infty n^2 v''(n z_{\min}) > 0$. 
\end{itemize}
\end{assumption} 

\noindent The assumption is satisfied, for example, by the Lennard-Jones potential $v(r) = r^{-12} - r^{-6}$. As we will see, parts (i) and (ii) of the assumption guarantee that energy minimizers at $p=0$ have interparticle spacings $z_j$ in $(z_{\min},z_{\max})$, parts (iii) and (iv) ensure that $\mathcal{E}_N$ is uniformly strictly convex in $(z_{\min},z_{\max})^{N-1}$; moreover the Hessian $\mathrm{D}^2 \mathcal{E}_N$ is diagonally dominant with positive diagonal entries and negative off-diagonal entries.

\begin{assumption}\label{assu:p} 
	The pressure $p$ satisfies $0\leq p<p^*$ with $p^*:= \frac{|v(z_{\max})|}{z_{\max}}$. 
\end{assumption}

At positive temperature we shall assume in addition that $p>0$, $r_{hc}>0$, and for some results we need $\lim_{r\searrow r_\mathrm{hc}} v(r) = \infty$. The next theorem is the adaptation of a similar result by Gardner and Radin \cite{gardner-radin79}. It is proven in Section~\ref{sec:bulk}. 

\begin{theorem}[Bulk properties] \label{thm:periodic}
	Let $m\in \N\cup \{\infty\}$ and $p\in [0,p^*)$ as in Assumption~\ref{assu:p}. 
	\begin{enumerate}
		\item [(a)] For every $N\geq 2$, the map $\mathcal{E}_N:\R_+^{N-1}\to \R$ has a unique minimizer $(z_1^\ssup{N},\ldots,z_{N-1}^\ssup{N})$. The mimizer has all its spacings $z_j$  in $ [z_{\min},z_{\max}]$.  
		\item [(b)] As $j,N\to \infty$ along $N-j\to \infty$, we have $z_j^\ssup{N}\to a$ where $a \in (z_{\min}, z_{\max}]$ is the unique minimizer of $\R_+\ni r\mapsto p r+ \sum_{k=1}^m v(kr)$.  
		\item [(c)] The limit $e_0 = \lim_{N\to \infty} (E_N/N)  < 0$ exists and is given by 
		$$ e_0= p a+ \sum_{k=1}^m v(ka) = \min_{r>0} \Bigl( p r +  \sum_{k=1}^m v(kr)\Bigr).$$
	\end{enumerate} 
\end{theorem}

Let $\mathcal{D}_0 \subset  (r_{\rm hc}, \infty)^{\N}$ be the space of sequences $(z_j)_{j\in \N}$ with none or at most finitely many elements different from $a$. 
Define 
\begin{align}
	h(z_1,\ldots,z_{m}) &= pz_1 + \sum_{k=1}^m v(z_1+\cdots + z_k) \label{eq:h-def} \\
	\mathcal{E}_{\rm surf}\bigl( (z_j)_{j\in \N}\bigr)& =\sum_{j=1}^\infty \bigl( h(z_j,\ldots,z_{j+m-1}) - e_0 \bigr), \quad  (z_j)_{j\in \N}\in \mathcal{D}_0. \notag
\end{align}
When $m=\infty$, $h((z_j)_{j\in \N})$ is a function of the whole sequence. $\mathcal{E}_\mathrm{surf}$ is the Gibbs energy of a semi-infinite chain, with additive constant chosen in such a way that at spacings $z_j\equiv a$ the Gibbs energy is zero; $h(z_1,z_2,\ldots)$ represents the interaction of the left-most particle with everybody else.
Let $\mathcal{D} = \{ (z_j)_{j\in \N}\in  (r_{\rm hc}, \infty)^{\N} \mid \sum_{j=1}^\infty(z_j-a)^2 <\infty\}$ be the space of square summable strains. 

\begin{theorem} [Surface energy] \label{thm:surface}
	Let $m\in \N\cup \{\infty\}$ and $p\in [0,p^*)$ as in Assumption~\ref{assu:p}. 
	Equip $\mathcal{D}$ with the $\ell^2$-metric. Then 
	\begin{enumerate}
		\item [(a)] $\mathcal{E}_\mathrm{surf}$ extends to a continuous functional on $\mathcal{D}$. 
		\item [(b)] On $\mathcal{D} \cap [z_{\min},z_{\max}]^\N$ it is strictly convex. 
		\item [(c)] $\mathcal{E}_\mathrm{surf}$ has a unique minimizer. The minimizer lies in $\mathcal{D}\cap [z_{\min},z_{\max}]^\N$. 
		\item [(d)] The limit $e_\mathrm{surf} =\lim_{N\to \infty}(E_N - N e_0)$ exists and is given by
			$$ e_\mathrm{surf} = 2 \min_\mathcal{D} \mathcal{E}_\mathrm{surf} - pa - \sum_{k=1}^m k v(ka).$$
	\end{enumerate}
\end{theorem} 
\noindent The theorem is proven in Section~\ref{sec:surface}. 
Note that $-pa-\sum_{k=1}^\infty k v(ka)$ is the surface energy for a clamped chain with all spacings equal to $a$ and $\mathcal{E}_\mathrm{surf}$ encodes the effect of boundary layers. $\mathcal{E}_\mathrm{surf}$ is  multiplied by $2$ because finite chains have two ends. 
We note that $\min \mathcal{E}_\mathrm{surf}$ is exactly the boundary layer energy introduced by Braides and Cicalese \cite{braides-cicalese07}; Braides and Cicalese dealt with the special case $m=2$ of next-nearest neighbor interactions but more general potentials. For finite $m\geq 2$, see ~\cite[Theorem 4.2]{schaeffner-schloemerkemper2018}.

For later purpose we also define a bulk functional 
\bes 
\begin{aligned}
	\mathcal{E}_\mathrm{bulk} \bigl( (z_j)_{j\in \Z}\bigr) & = \sum_{j=-\infty}^\infty \bigl( h(z_j,\ldots,z_{j+m-1}) - e_0\bigr) \\
	& = \sum_{j=-\infty}^\infty \sum_{k=1}^m \bigl( v(z_j+\cdots + z_{j+k-1}) - v(ka) + \delta_{1k} p (z_j - a) \bigr).
\end{aligned} 
\ees
It is defined, a priori, on the space $\mathcal{D}_0^+$ of positive bi-infinite sequences  $(z_j)_{j\in \Z} \in (r_{\rm hc}, \infty)^{\Z}$ that have at most finitely many elements $z_j \neq a$.  Denoting the space of square summable strains $\mathcal{D}^+ = \{(z_j)_{j\in \Z} \in (r_{\rm hc}, \infty)^{\Z} \mid \sum_{j\in \Z}(z_j-a)^2<\infty\}$, an analysis similar to the one for the surface functional yields the following result.

\begin{prop} [Limiting bulk properties] \label{prop:lim-bulk} 
	Let $m\in \N\cup \{\infty\}$ and $p \in [0,p^*)$ as in Assumption~\ref{assu:p}. 
	Equip $\mathcal{D^+}$ with the $\ell^2$-metric. Then 
	\begin{enumerate}
		\item [(a)] $\mathcal{E}_\mathrm{bulk}$ extends to a continuous functional on $\mathcal{D}^+$. 
		\item [(b)] On $\mathcal{D}^+ \cap [z_{\min},z_{\max}]^\N$ it is strictly convex. 
		\item [(c)] The unique minimizer of $\mathcal{E}_\mathrm{bulk}$ is the constant sequence $(\ldots, a, a, \ldots)$. The minimum value is $\mathcal{E}_\mathrm{bulk}(\ldots, a, a, \ldots) = 0$. 
		\item [(d)] For every $(z_j)_{j \in \Z} \in \mathcal{D}^+$ one has 
			\begin{align*} 
				\mathcal{E}_\mathrm{bulk}((z_j)_{j \in \Z}) 
				&= \mathcal{E}_\mathrm{surf}(z_1, z_2, \ldots) + \mathcal{E}_\mathrm{surf}(z_0, z_{-1}, \ldots) \\ 
				&\qquad + \mathcal{W}( \cdots z_{-1} z_0 \mid z_1 z_2 \ldots ), 
			\end{align*}
			where $\mathcal{W}( \cdots z_{-1} z_0 \mid z_1 z_2 \ldots ) := \sum_{\heap{j \leq 0, k \geq 1}{|k-j|\leq m-1}} v(z_j+\cdots + z_k)$ is the total interaction between the left and right half-infinite chain. 
	\end{enumerate}
\end{prop}

\subsection{Small positive temperature} 
Next we analyze infinite volume Gibbs measures on $\R_+^\N$ and $\R_+^\Z$ in the limit $\beta \to \infty$. We focus on fixed positive $p\in (0, |v(z_{\max})|/z_\mathrm{\max})$ but comment on vanishing $p=p_\beta\to 0$ at the end of the section. Let $\Q_N^\ssup{\beta}$ be the probability measure on $\R_+^{N-1}$ defined by 
\bes
	\Q_N^\ssup{\beta} (A) = \frac{1}{Q_N(\beta)} \int_A \e^{-\beta \mathcal{E}_N(z_1,\ldots,z_{N-1})} \dd z_1\cdots \dd z_{N-1} 
\ees
where
\bes
	Q_N(\beta) = \int_{\R_+^{N-1}} \e^{-\beta \mathcal{E}_N(z_1,\ldots,z_{N-1})} \dd z_1\cdots \dd z_{N-1}. 
\ees
Standard arguments  (see Section~\ref{sec:gibbs}) show there is a  uniquely defined probability measure $\nu_\beta$ on the product space $\R_+^\N$ such that  for every $k\in \N$, every bounded continuous test function $f \in C_b(\R_+^k)$, 
\be \label{eq:nutherm}
	\lim_{N\to \infty} \int_{\R_+^{N-1}} f(z_1,\ldots,z_k) \dd \Q_N^\ssup{\beta}(z_1,\ldots,z_{N-1})  = \int_{\R_+^{\N}} f(z_1,\ldots,z_k) \dd \nu_\beta( (z_j)_{j\geq 1}).
\ee
Similarly, there is a uniquely defined probabilty measure $\mu_\beta$ on $\R_+^{\Z}$ such that 
for all local test functions $f$ as above, and all sequences $i_N$ with $i_N\to \infty$ and $N-i_N\to \infty$, 
\be \label{eq:mutherm}
	\lim_{N\to \infty} \int_{\R_+^{N-1}} f(z_{i_N+1},\ldots,z_{i_N+k}) \dd \Q_N^\ssup{\beta}(z_1,\ldots,z_{N-1})  = \int_{\R_+^{\Z}} f(z_1,\ldots,z_k) \dd \mu_\beta( (z_j)_{j\geq 1}).
\ee 
Moreover the measure $\mu_\beta$ is shift-invariant and mixing. The measure $\mu_\beta$ describes the bulk behavior of a semi-infinite chain, the measure $\nu_\beta$ is the equilibrium measure for a semi-infinite chain and encodes the probability distribution of boundary layers. 

Our first result is a large deviations principle for the equilibrium measure $\nu_\beta$ as $\beta \to \infty$. The rate function is a suitable extension of $\mathcal{E}_\mathrm{surf}$: define 
$\overline{\mathcal{E}}_\mathrm{surf}:\ \R_+^\N \to \R \cup \{\infty \}$ by 
\be\label{eq:Esurfextended}
	\overline{\mathcal{E}}_\mathrm{surf}\bigl( (z_j)_{j\in \N}\bigr) = \begin{cases}
			 \mathcal{E}_\mathrm{surf}\bigl( (z_j)_{j\in \N}\bigr), &\quad  (z_j)_{j\in \N} \in \mathcal{D},\\
			\infty, &\quad \text{else}. 
		\end{cases} 
\ee
In the same way $\mathcal{E}_\mathrm{bulk}$ extends to a map  $\overline{\mathcal{E}}_\mathrm{bulk}$ from $\R_+^\Z$ to $\R\cup\{\infty\}$. Both $\R_+^\N$ and $\R_+^\Z$ are equipped with the product topology. 
 
\begin{theorem} \label{thm:ldp}
	Fix $p\in (0,p^*)$ and $m\in \N\cup \{\infty\}$. Assume that $r_\mathrm{hc}>0$ and $\lim_{r\searrow r_\mathrm{hc}} v(r) = \infty$. Then 
	as $\beta\to \infty$, the equilibrium measures $(\nu_\beta)_{\beta>0}$ and $(\mu_\beta)_{\beta>0}$ satisfy large deviations principles with speed $\beta$ and respective rate functions $\overline{\mathcal{E}}_{\mathrm{surf}} - \min \mathcal{E}_\mathrm{surf}$ and $\overline{\mathcal{E}}_\mathrm{bulk}$. The rate functions are \emph{good}, i.e., lower semi-continuous with compact level sets.
\end{theorem} 
\noindent The theorem is proven in Section~\ref{sec:ldp}. The large deviations principle for $\nu_\beta$ says that for every closed set $A\subset \R_+^\N$ and every open set $O\subset \R_+^\N$ (product topology)
\be\label{eq:ldp}
\begin{aligned}
	\limsup_{\beta\to \infty} \frac{1}{\beta}\log \nu_\beta(A) & \leq - \inf_{(z_j)\in A }\Bigl( \overline{\mathcal{E}}_\mathrm{surf} \bigl((z_j)\bigr)- \min_{\R_+^\N} \mathcal{E}_\mathrm{surf} \Bigr) \\
	\liminf_{\beta\to \infty} \frac{1}{\beta}\log \nu_\beta(O) & \geq - \inf_{(z_j)\in O }\Bigl( \overline{\mathcal{E}}_\mathrm{surf} \bigl((z_j)\bigr)- \min_{\R_+^\N} \mathcal{E}_\mathrm{surf} \Bigr).
\end{aligned} 
\ee
It is essential that we work in the product topology. Indeed we shall later see that $\nu_\beta$ is mixing, therefore for every $\eps>0$, the measure $\nu_\beta$ gives full mass $1$  to sequences $(z_j)_{j\in \N}$ that have infinitely many spacings $|z_j-a|>\eps$. Thus for every ball $O= \{ (z_j) \in \R_+^\N \mid \sum_{j=1}^\infty (z_j-a)^2<\delta\}$,  we have $\nu_\beta(O) = 0$ hence $\beta^{-1}\log\nu_\beta(O) = -\infty$, to be contrasted with the lower bound in Eq.~\eqref{eq:ldp}.

Another consequence concerns the evaluation of the Gibbs energies of localized defects: suppose that because of some impurity, the energy is not $\mathcal{E}_N$ but $\mathcal{E}_N + \mathcal{V}$, where $\mathcal{V}$ is, say, continuous in the product topology, localized in the bulk,  and bounded from below. Then by Varadhan's lemma~\cite{dembo-zeitouni}, as $\beta\to \infty$, the effective Gibbs energy converges to the zero temperature energy of the defect,
\bes
	-\frac{1}{\beta}\log \mu_\beta\bigl( \e^{-\beta \mathcal{V}}\bigr) \to \inf_\mathcal{D} (\mathcal{E}_\mathrm{bulk}  + \mathcal{V})\quad (\beta\to \infty).
\ees
Surface energies occur as  a specific type of defect, when $\mathcal{V}$ cancels all interactions between two half-infinite chains (see Proposition~\ref{prop:thermolim}(a)), which leads to the following theorem. 
Define
\be\label{eq:gdef}
	g(\beta) = -  \lim_{N\to \infty} \frac{1}{\beta N }\log Q_N(\beta),\quad 
	g_\mathrm{surf}(\beta) =  \lim_{N\to \infty} \Bigl( -\frac{1}{\beta}\log Q_N(\beta) - N g(\beta) \Bigr), 
\ee
the Gibbs free energy $g(\beta)$ per particle in the bulk and the surface correction $g_\mathrm{surf}(\beta)$. 

\begin{theorem} \label{thm:geld} 
	Fix $p\in (0,p^*)$ and $m\in \N\cup \{\infty\}$. 
	The limits~\eqref{eq:gdef} exist. If in addition $r_\mathrm{hc}>0$ and $\lim_{r\searrow r_\mathrm{hc}} v(r) = \infty$, then the bulk and surface Gibbs energy  approach their zero-temperature counterparts when $\beta \to \infty$: 
	\bes
		\lim_{\beta\to \infty} g(\beta)= e_0,\quad \lim_{\beta\to \infty} g_\mathrm{surf}(\beta) =e_\mathrm{surf}. 
	\ees
\end{theorem} 
\noindent This proves that the thermodynamic limit and the zero temperature limit can be exchanged, which is non-trivial (and in fact, fails when the pressure goes to zero too fast, see below). 

One last consequence of Theorem~\ref{thm:ldp} concerns the distribution of spacings and the pressure-density (or stress-strain) relation. The Gibbs free energy and our partition functions correspond to an ensemble where the overall length of the system is not fixed, but instead may fluctuate with a law that depends on the pressure---high pressures $p$ favor compressed states. In the thermodynamic limit $N\to \infty$, though, the average spacing between particles becomes a well-defined quantity, given by 
\be\label{eq:avspa}
  \ell (\beta) = \int_{\R_+^\Z} z_0 \dd \mu_\beta ((z_j)_{j\in \Z}). 
\ee
By the contraction principle~\cite[Theorem 4.2.1]{dembo-zeitouni}, the distribution of $z_0$ under $\mu_\beta$ satisfies a large deviations principle with good rate function $w(z) = \inf \{\overline{\mathcal{E}}_\mathrm{bulk}((z_j)_{j\in \Z})\mid (z_j)_{j\in \Z}\in \R_+^\Z,\, z_0 = z\}$. The unique minimizer of $w(z)$ is the ground state spacing $a$. Lemma~\ref{lem:ti} implies that the distribution of spacings has exponential tails 
$$\mu_\beta\bigl( \{(z_j)_{j\in \Z}\mid z_0\geq r\}\bigr) \leq C\exp(-\beta p r)$$ 
for some $\beta$-independent constant $C$. 

\begin{cor} \label{cor:spacings}
		Under the assumptions of Theorem~\ref{thm:geld}, we have
		$$\lim_{\beta\to \infty} \ell(\beta) = a = \mathrm{argmin} \bigl(p r +\sum_{k=1}^{m} v(kr)\bigr).$$ 
\end{cor}

\noindent In particular, for large $\beta$, we have $\ell(\beta)< a_0$ where $a_0$ is the minimizer of the zero-stress Cauchy-Born energy density $\sum_k v(kr)$. Conversely, spacings $\ell(\beta)>a_0$ (elongated chains) imply vanishing pressure $p=p_\beta\to0$. This is clearly apparent for  nearest neighbor interactions ($m=1$, Takahashi nearest neighbor gas~\cite{takahashi42,lieb-mattis-chap1}), for which 
\be \label{eq:taka}
	 g(\beta) = - \frac{1}{\beta}\log \Bigl( \int_0^\infty \e^{-\beta[v(r)+ p_\beta r]} \dd r\Bigr),\quad \ell(\beta) = \frac{\int_0^\infty r \exp(-\beta[v(r)+ p_\beta r]) \dd r}{\int_0^\infty \exp(-\beta[v(r)+ p_\beta r]) \dd r}.\\
\ee

\emph{Comments on vanishing pressure.} We add a superscript to indicate that zero-temperature quantities are evaluated at $p=0$.  When $p=p_\beta\to 0$ slower than any exponential, it is still true that $g(\beta) \to e_0^0$. When $\beta p_\beta = \exp(-\beta \nu)$ with $\nu>0$, one can show with~\cite{jansen-koenig-metzger15,jansen12} that 
\be \label{eq:soltogas}
	\lim_{\beta\to \infty} g(\beta) = \min (e_0^0, - \nu). 
\ee
At pressures vanishing faster than $\exp(- \beta |e_0^0|)$, the most likely configurations have very large spacings (dilute gas phase, $\ell(\beta)\to \infty$) and the previous results no longer apply. For $\liminf\frac{1}{\beta}\log (\beta p_\beta)>e_0^0$, we expect that large deviations principles with  rate functions $\overline{\mathcal{E}}_\mathrm{bulk}^0$ and $\overline{\mathcal{E}}_\mathrm{surf}^0 - \min \overline{\mathcal{E}}_\mathrm{surf}^0$ still hold (in fact our proofs still  show \emph{weak} large deviations principles). However rate functions have non-compact level sets and exponential tightness is lost. Moreover large spacings may contribute to the average~\eqref{eq:avspa} and Corollary~\ref{cor:spacings} need no longer be true, thus allowing for spacings $\ell(\beta)\to \ell>a_0$. \\

\subsection{Gaussian approximation} 

Here we complement the large deviations result by a Gaussian approximation. This section deals with finite $m$ and the bulk measure $\mu_\beta$ only. 
Remember $d=m-1$. 
We will see that the Hessian of $\mathcal E_{\mathrm{bulk}}$ at $(\ldots,a,a,\ldots)$ is associated with a positive-definite, bounded operator $\mathcal H$ in $\ell^2 (\Z)$. It is represented by a doubly-infinite matrix $(\mathcal H_{ij})_{i,j\in \Z}$ that is diagonally dominant. Write $(\mathcal H^{-1})_{ij}$ for the matrix elements of the inverse operator and let $\mu ^{\mathrm{Gauss}}$ be the uniquely defined measure on $\R^\Z$, equipped with the product topology and its associated Borel $\sigma$-algebra, such that 
$$
		\int_{\R^\Z} s_i s_j \dd \mu^\mathrm{Gauss}\bigl( (s_k)_{k\in \Z}\bigr) = (\mathcal H^{-1})_{ij}
$$
for all $i,j\in \Z$, and every finite-dimensional marginal of $\mu^\mathrm{Gauss}$ is a multi-dimensional Gaussian distribution. Equivalently, $\mu^\mathrm{Gauss}$ is the distribution of a Gaussian process $(N_j)_{j\in \Z}$ with mean zero and covariance $\E[N_i N_j] = (\mathcal H^{-1})_{ij}$. More concrete expressions for the probability density functions of $nd$-dimensional marginals of $\mu^\mathrm{Gauss}$ are provided in Proposition~\ref{prop:gaussmarginals} below. 

In the following we identify the measure $\mu_\beta$ on $\R_+^\Z$ with the measure $\1_{\R_+^\Z} \mu_\beta$ on $\R^\Z$. We exclude the trivial case $m=1$. 

\begin{theorem} \label{thm:gaussian2}
	Assume $2 \leq m<\infty$, $p\in (0,p^*)$, and $r_\mathrm{hc}>0$.
	Then for every $n\in \N$, the $n$-dimensional marginals of $\mu_\beta$ and $\mu^\mathrm{Gauss}$ have probability density functions $\rho_n^{(\beta)}$ and $\rho_n^\mathrm{Gauss}$, and 
	$$
		\lim_{\beta \to \infty}\int_{\R^n} \Bigl| \beta^{-n/2} \rho_n^{(\beta)} \bigl(a+ \beta^{-1/2}s_1,\ldots, a+ \beta^{-1/2}s_n \bigr) - \rho_n^\mathrm{Gauss}(s_1,\ldots,s_n )\Bigr| \dd s_1\ldots \dd s_n = 0. 
	$$
\end{theorem} 

\noindent It follows that the distribution of the spacings, suitably rescaled, converges locally to the Gaussian measure $\mu^\mathrm{Gauss}$: for every bounded function $f:\R^Z\to \R$ that depends on finitely many spacings $z_j$ only (bounded cylinder functions), we have 
$$
	\lim_{\beta \to \infty}\int_{\R^\Z} f\bigl( \sqrt{\beta}(z_j-a)_{j\in \Z}\bigr) \dd \mu_\beta\bigl((z_j)_{j\in \Z}\bigr) = \int_{\R^\Z} f \dd \mu^\mathrm{Gauss}.
$$
For example, in the limit $\beta\to \infty$, the distribution of a single spacing $z_j$ is approximately normal, with mean $a$ and variance $\beta^{-1} (\mathcal H^{-1})_{ii}$. We expect that Theorem~\ref{thm:gaussian2} stays true for $m=\infty$ but a proof or disproof is beyond the scope of this article. 

The next theorem says that the Gibbs free energy is close to the Gibbs free energy of the approximate Gaussian model.

\begin{theorem} \label{thm:gaussian3}
	Assume $2 \leq m<\infty$, $p\in (0,p^*)$, and $r_{\mathrm{hc}}>0$. The Gibbs free energy satisfies, as $\beta \to \infty$,  
	$$
		g(\beta) = e_0 - \frac1\beta\log \sqrt{\frac{2\pi}{\beta (\det C)^{1/d}}} +o(\beta^{-1})
	$$
	where $d=m-1$ and $C$ is a $d\times d$ positive-definite matrix. 
\end{theorem} 
\noindent The matrix $C$ is introduced in Eq.~\eqref{eq:Cdef}, see also Lemma~\ref{lem:qq}, it is a function of the Hessian of the energy. \\

\begin{remark}[Gaussian approximation and semi-classical expansions] If $v$ is smooth and $p>0$ is fixed, the Gibbs energy should admit an asymptotic expansion of  form 
$$ g(\beta) = e_0 - \frac{1}{\beta}\log \sqrt{\frac{2\pi}{\beta c}} + \sum_{j=1}^n a_j \beta^{-j/2} + O(\beta^{-(n+1)/2}) \quad (\beta \to \infty)$$
to arbitrarily high order $n$, for some $c>0$ and coefficients $a_j\in \R$. The first correction comes from a Gaussian approximation of the partition function (\emph{harmonic crystal}), see Section~\ref{sec:gaussian}, with the constant $c$ capturing the asymptotic behavior of the determinant of the Hessian around the energy minimum. Higher order corrections correspond to anharmonic effects. A similar expansion holds for $g_\mathrm{surf}(\beta)$.  Rigorous results for finite $m$ are derived with semi-classical analysis~\cite{helffer-book,schach-moller2001,bach-schachmoller2003} which build on  the analogy with the $\hbar \to 0$ limit from quantum mechanics. For $m=2$ and potentials with superlinear growth at infinity, independent results are given
 in~\cite{shapeev-luskin17}.
\end{remark}

\subsection{Decay of correlations} \label{sec:corres}
Suppose that two defects change the energy functional from $\mathcal{E}_\mathrm{bulk}$ to $\mathcal{E}_\mathrm{bulk} + \mathcal{V}_0 + \mathcal{V}_k$, where we assume for simplicity that $\mathcal{V}_0$ and $\mathcal{V}_k$ depend on $z_0$ and $z_k$ alone. 
For large $k$, we may expect that the Gibbs energies are approximately additive, i.e., 
\be \label{eq:effective-int}
	\mathcal{I}^\ssup{\beta}_\mathrm{eff} (k)= -\frac{1}{\beta} \log \mu_\beta( \e^{-\beta (\mathcal{V}_0+\mathcal{V}_k)}) + \frac{1}{\beta}\log \mu_\beta( \e^{-\beta \mathcal{V}_0}) + \frac{1}{\beta} \log \mu_\beta( \e^{-\beta \mathcal{V}_k})
\ee
should be small when the defects are far apart. $\mathcal{I}_\mathrm{eff}^\ssup{\beta}(k)$ represents an effective interaction between the defects. 
In the study of systems with many defects it is important to understand how fast the effective interaction decreases at large distances. Some intuition is gained from the zero-temperature counterpart
\be \label{eq:effective-zero}
	\mathcal{I}^\ssup{\infty}_\mathrm{eff}(k) = \inf (\mathcal{E}_\mathrm{bulk} + \mathcal{V}_0 + \mathcal{V}_k) - 	\inf (\mathcal{E}_\mathrm{bulk} + \mathcal{V}_0) - 	\inf (\mathcal{E}_\mathrm{bulk} + \mathcal{V}_k), 
\ee
however in general the limits $\beta,k\to \infty$ cannot be interchanged and a full study of~\eqref{eq:effective-int} for large $k$ requires techniques beyond variational calculus. 

A closely related problem is about the localization of changes induced by  a  defect: at zero temperature, if $(z_j)_{j\in \Z}$ is a minimizer of $\mathcal{E}_\mathrm{bulk} + \mathcal{V}_0$, how fast does $z_k$ converge to the ground state spacing $a$ as $k\to \pm \infty$? On a similar note, how fast does $z_k\to a$ for a minimizer of the surface energy $\mathcal{E}_\mathrm{surf}$ (decay of boundary layers)? At positive temperature, the question is about the speed of convergence, for test functions $f:\R_+^k \to \R$, in 
$$\frac{\mu_\beta(\e^{-\beta \mathcal{V}_0} f_{i})}{\mu_\beta(\e^{-\beta \mathcal{V}_0})}
 \to \mu_\beta(f),\quad \nu_\beta(f_i)\to \mu_\beta(f)$$
as $i\to \infty$. Here $f_i((z_j)_{j\in \Z}):= f(z_i,\ldots,z_{i+k-1})$, so that $f_{n+i} = f_i\circ \tau^n$ when $\tau$ denotes the left shift on $\R_+^{\Z}$. These questions naturally lead to the investigation of the decay of correlations. We start with a general result which holds for all $\beta,p>0$. 

\begin{theorem}  \label{thm:corrgen}
	Assume $m\in \N\cup\{\infty\}$ and $p>0$. There exist $c,C>0$ such that for all $\beta,p>0$, $k\in \N$, and bounded $f,g:\R_+^k \to \R$,
	\begin{equation*}
		\bigl|\mu_\beta (f_0 g_n) - \mu_\beta(f_0) \mu_\beta(g_n)\bigr| \leq   \min_{\heap{q\in \N:}{1 \leq q \leq n/k}} \Bigl( (1- \e^{-c\beta})^q + \e^{c\beta} (\e^{C \beta (q/n)^{s-2}} -1) \Bigr) ||f||_\infty ||g||_\infty.
	\end{equation*}	
	When $m$ is finite and $k=m-1$, we have the stronger bound 
	\bes	
		\bigl|\mu_\beta (f_0 g_n) - \mu_\beta(f_0) \mu_\beta(g_n)\bigr| \leq  (1-\e^{-c\beta})^{n/k} ||f||_\infty ||g||_\infty. 
	\ees
\end{theorem} 

\noindent The theorem is proven in Section~\ref{sec:ergodicity}. 
When $m$ is finite, it implies exponential decay of correlations as $n\to \infty$, however the rate $ -\log(1- \e^{-c\beta})$ can be exponentially small for large $\beta$. When $m$ is infinite, Theorem~\ref{thm:corrgen}
implies algebraic decay of correlations: for $q=\lfloor n^{\eps}\rfloor$ and sufficiently large $n$, $(1- \e^{-c\beta})^q$ is negligible compared to $\beta(q/n)^{s-2}$ and we find that as $n\to \infty$ 
\be\label{eq:apalg}
	\bigl|\mu_\beta (f_0 g_n) - \mu_\beta(f_0) \mu_\beta(g_n)\bigr| \leq (1+ O(1)) \frac{C \beta \exp(c\beta)}{n^{(s-2)(1-\eps)}}. 
\ee
Better bounds are available for restricted Gibbs measures. Let $\tilde \mu_\beta ^\ssup{N}$ be the measure $\Q_N^\ssup{\beta}$ conditioned on $[z_\mathrm{min},z_\mathrm{max}]^{N-1}$ and $\tilde \mu_\beta$ the probability measure on $[z_\mathrm{min},z_\mathrm{max}]^\Z$ obtained from the thermodynamic limit of $\tilde \mu_\beta^\ssup{N}$. 

\begin{prop} \label{prop:restrict-alg}
	Let $m \in \N\cup \{\infty\}$. There exists $c>0$ such that for all $\beta,p>0$, smooth $f,g:\R_+\to \R$, and $i\neq j$, 
	\bes
			\Bigl|\tilde \mu_\beta (f_i g_j) 
 - \tilde \mu_\beta (f_i) 	\tilde \mu_\beta(g_j) \Bigr| \leq 
  	  \frac{c}{\beta |i-j|^{s}}  \Bigl( \tilde \mu_\beta({f'_i}^2)  \tilde \mu_\beta({g'_j}^2) \Bigr)^{1/2}.
	\ees
\end{prop} 

\begin{remark}
	When $m$ is finite, the uniform algebraic decay for the restricted Gibbs measure is replaced with uniform exponential decay  $\exp (- \gamma |j-i|)$ with $\beta$-independent $\gamma>0$. 
\end{remark} 

\noindent The proposition is proven in Section~\ref{sec:brascamp}. It follows from the uniform convexity of the energy (Lemma~\ref{lem:hessian}) and known results from the realm of Brascamp-Lieb, Poincar{\'e} and Log-Sobolev inequalities. Proposition~\ref{prop:restrict-alg} differs from the  estimate~\eqref{eq:apalg} in two ways: there is no exponentially large prefactor $\exp(c\beta)$, and the rate of algebraic decay is $1/n^s$ instead of $1/n^{s-2}$. Exponentially large prefactors are absent because the energy landscape has no local minimum. The improved algebraic decay $1/n^s$ arises, roughly, because the Gibbs measure is comparable to a Gaussian measure whose covariance is the inverse of the energy's Hessian near the minimum, and instead of the tails of $v(r)$, it is the tails of $v''(r)$ that count. 

We suspect that for large $\beta$ and small pressure, these improvements should carry over to the full Gibbs measure $\mu_\beta$, but we have proofs for 
interactions involving finitely many neighbors only. 

\begin{theorem} \label{thm:corr-finitem}
	Assume $2\leq m<\infty$, $p\in (0,p^*)$, and $r_\mathrm{hc}>0$. There exists $\gamma>0$ such that for all sufficiently large $\beta$, suitable $C(\beta)$, all $n\in \N$, and all 
	$f,g:\R_+^d \to \R$, we have 
	$$
		\bigl|\mu_\beta(f_{0} g_{n}) - \mu_\beta (f_{0}) \mu_\beta(g_{n})\bigr| \leq C(\beta) \e^{-\gamma n} ||f_0||_\infty \, ||g_n||_\infty. 
	$$
	If $m=2$, we can  pick $C(\beta) =1$. 
\end{theorem}

\noindent The theorem is proven in Section~\ref{sec:gaussian} with perturbation theory for compact integral operators  in $L^2(\R^d)$. When $m=2$, the relevant operators are self-adjoint and spectral norms and operator norms coincide, leading to improved statements. 
We conclude with a few comments. \\

\emph{Lagrangian vs. Eulerian point of view.}  The theorems above formulate decay of correlations in terms of labelled spacings, which in the language of continuum mechanics is a Lagrangian viewpoint. On the other hand, in statistical mechanics of point particles it is more common to deal with unlabelled particles (Eulerian viewpoint) and correlations are between portions of space rather than labelled interparticle distances. The difference between the two approaches becomes quite clear for nearest neighbor interactions ($m=1$, see Eq.~\eqref{eq:taka}), for which the spacings are i.i.d. with probability density $q_\beta(r)$ proportional to $\exp( - \beta [v(r)+ p_\beta r])$. Because of the independence of spacings, correlations in terms of spacings vanish, $\mu_\beta (f_0 g_n) - \mu_\beta(f_0) \mu_\beta(g_n) =0$. On the other hand, the two-point function $\rho_2(0,x)$\footnote{Intuitively, $\rho_2(0,x)$ represents the probability for having one particle at $0$ and one particle at~$x$. Rigorously, $\rho_2(x_1,x_2) = \rho_2(0,x_2-x_1)$ and for every $A$, $\int_A \rho_2(x_1,x_2) \dd x_1 \dd x_2$ is the average number $\la N_A(N_A-1)\ra$ of ordered particle pairs in $A$.} studied in statistical mechanics of particles is a sum over the number of particles contained in $(0,x]$,
$$ \rho_2(0,x) =\frac{1}{\ell(\beta)} \sum_{k=1}^\infty q_\beta^{*k}(x)
	= \frac{q_\beta(x)}{\ell(\beta)} + \int_0^\infty q_\beta(y-x) \rho_2(0,y)\dd y
$$
with $q_\beta^{*k}$ the $n$-fold convolution of $q_\beta$ with itself. It is a well-known fact from renewal theory~\cite[Chapter XI]{feller-vol2} that 
$$ \rho_2(0,x) - \frac{1}{\ell(\beta)^2} \to 0 \quad (x\to \infty), $$
but in general the difference is non-zero finite for $x$---in fact changing $q_\beta$ the convergence as $x\to \infty$ can be arbitrarily slow, even though correlations of labelled interparticle spacings vanish identically. One should keep this difference in mind when browsing the literature. 
\\

\emph{Path-large deviations, non-linear semi-groups, Bellman equation.} 
For $m=2$, we may view $\mu_\beta$ as the law of a stationary Markov chain with state space $\R_+$ and transition kernel $P_\beta$ defined in Eq.~\eqref{eq:pkernel}. Theorem~\ref{thm:ldp} is a path-large deviations result for the Markov chain. Path large deviations are often investigated with the help of non-linear semi-groups and Hamilton-Jacobi-Bellman equations~\cite{feng-kurtz-book}. In our context, a natural non-linear semi-group is
$$ V_\beta^nf:= - \frac{1}{\beta}\log \Bigl(P_\beta^n \e^{-\beta f} \Bigr)
$$ 
and for sufficiently smooth $f$ we have a convergence of the form 
$$\lim_{\beta\to \infty}V_\beta f (x) =- u(x)+  \inf_{y \in \R_+} \bigl( p x + v(x)   + v(x+y) - e_0 + u(y)+ f(y)\bigr)$$
where $u$ solves 
\bes 
 u(x) = \inf_{y\in \R_+} \bigl(p x + v(x)   + v(x+y) - e_0 + u(y)\bigr).
\ees 
Similar equations, motivated by quantum mechanics and geometric optics, appear in semi-classical analysis~\cite[Eq.~(5.4.4)]{helffer-book}.  Proposition~\ref{prop:bellman-ersatz} below provides an infinite-$m$ ersatz and is instrumental in the proof of Theorem~\ref{thm:ldp}. 
\\

\emph{Vanishing pressure.} When $\beta p = \beta p_\beta\to 0$ faster than $\exp(- \beta |e_0^0|)$ (see~\eqref{eq:soltogas}), the Gibbs measure should no longer be comparable to a Gaussian. Instead, it should be close to the ideal gas measure, for which spacings are i.i.d. exponentially distributed with parameter $\beta p_\beta$, and we may again expect uniform exponential decay of correlations (for finite $m$). When $\beta p_\beta\to 0$ at a speed comparable to $\exp(-\beta |e_0^0|)$, we should instead expect an exponentially small spectral gap: the Markov chain has two metastable wells, one corresponding to the optimal spacing $a$ and another well at infinity. The exponentially small spectral gap is associated with the fracture of the chain of atoms, in the spirit of ``fracture as a phase transition''~\cite{truskinovsky96}. \\


\section{Energy estimates} \label{sec:energy}

In this section we analyze the variational problems arising at zero temperature. Throughout the section we assume that $p\in [0,p^*)$ as in Assumption~\ref{assu:p}.

\subsection{Bulk periodicity} \label{sec:bulk}

\begin{lemma}\label{lem:est-on-min}
 Every minimizer of  $\mathcal{E}_N:\R_+^{N-1} \to \R$ lies in $[z_{\min},z_{\max}]^{N-1}$.
\end{lemma}

\begin{proof}
 Let $z_1,\ldots,z_{N-1}>0$. If $z_j > z_{\max}$ for some $j$,
  define a new configuration by shrinking $z_j$ to $z_{\max}$, leaving all other spacings unchanged:
  $z'_i = z_i$ for $i \neq j$ and $z'_j= z_{\max}$. Since $z_{\rm max}$ is a strict minimizer of $v$ and $r \mapsto v(r)$ increases on $[z_{\max}, \infty)$, shrinking the bonds  decreases $\mathcal{E}_N$ strictly and the original configuration could not have been a minimizer. 

  If some interparticle spacing is smaller than $z_{\min}$, we remove a particle and reattach it to one end of the chain as follows. Assume $b:= \min (z_1,\ldots,z_{N-1})< z_{\min}$ and  let $j\in \{1,\ldots,N-1\}$ with $z_j=b$. Let $x_1=0$ and $x_i = z_1+\cdots + z_{i-1}$, $i=2,\ldots,N$ be associated particle positions. Thus $x_{j+1}-x_j =z_j = b$ and $x_{i+1}- x_i \geq b$ for all $i$. The interaction of $x_j$ with all other particles is 
\[ v(b) + \sum_{i=1}^{\min\{m-1,N-j-1\}} v(z_{j} + \ldots + z_{j+i}) + \sum_{i=1}^{\min\{m,j-1\}} v(z_{j-1} + \ldots + z_{j-i}). \]
For finite $m$ we note that, if $v(z_{j-i} + \ldots + z_{j-i+m}) > 0$ for an $i \in \{ 1, \ldots, \min\{m,j-1\} \}$, then $v(z_{j-i} + \ldots + z_{j-i+m}) < v(z_{j-i} + \ldots + z_{j-1})$ by Assumption~\ref{assu:v}(i). Removing the particle $x_j$ thus leads to a configuration of $N$ atoms whose energy has decreased by at least 
 \be \label{eq:dwb}
    \Delta_1 = v(b) + v(z_{\max}) - 2 \alpha_1 \sum_{n=2}^m (n b)^{-s} 
             \geq v(b) + v(z_{\max}) - 2 \alpha_1 \sum_{n=2}^\infty (n b)^{-s} > 0.
  \ee
  The last inequality holds because of Assumption~\ref{assu:v}(ii) and $b<z_{\rm min}$. We define a new configuration by attaching the removed particle to either end of the chain at a distance $r = z_{\rm max}$. 
  Since $v(z_{\max}) + p z_{\max} < 0$ by Assumption \ref{assu:p}, this decreases $\mathcal{E}_N$ further, so overall the new configuration has strictly smaller energy, and the original sequence of spacings cannot be a minimizer of  $\mathcal{E}_N$.
\end{proof}

At zero pressure, it is a well-known fact that the $N$-particle energy is subadditive, $E_{N+M}\leq E_N + E_M$. Indeed placing two $N$,$M$-particle minimizers side by side with large mutual distance, because of $v(r)\to 0$ at $r\to \infty$, yields  an $N+M$-particle configuration with energy $\leq E_N+E_M$. Positive pressure penalizes large  mutual distances between two consecutive blocks, so the construction has to be modified. 

\begin{lemma} \label{lem:subadditive}
	Let $m\in \N\cup \{\infty\}$ and $p\in [0,p^*)$. Then  
	 $E_{N+M-1}\leq E_{N} + E_{M}$ for all $N,M\in \N$, and 
	the limit $e_0= \lim E_N/N$ exists and satisfies 
	 $E_N \geq (N -1)e_0$ for all $N\in \N$. 
\end{lemma}

\begin{proof}
	Let $z\in (r_\mathrm{hc},\infty)^{N-1}$ and $w\in (r_\mathrm{hc},\infty)^{M-1}$ be minimizers of $\mathcal{E}_N$ and $\mathcal{E}_{M}$ respectively. 
	Define $y\in (r_\mathrm{hc},\infty)^{M+N-2}$ by concatenating $z$ and $w$. By Lemma~\ref{lem:est-on-min}, all spacings are in $[z_{\min}, z_{\max}]$. Therefore interactions that involve bonds from both blocks are for spacings $\geq 2 z_{\min} >z_{\max}$, hence negative, and 
		$$ E_{N+M-1} \leq \mathcal{E}_{N-1}(y) \leq E_N + E_M.$$ 
	As a consquence, $a_n:= E_{n+1}$ is subadditive. By Fekete's subadditive lemma, the limit $e_0= \lim a_n/n = \lim E_{n}/n$ exists and is equal to the infimum of $a_n/n$, hence 
	$ E_{N} \geq (N-1) e_0$. Notice that $e_0>-\infty$
	since 
	\[ E_n 
	   \ge (n-1) \Big( v(z_{\max}) + \sum_{j=2}^{\infty} v(j z_{\min}) \Big) 
		 \ge (n-1) \Big( v(z_{\max}) + \alpha_1 z_{\min}^{-s} \sum_{j=2}^{\infty} j^{-s} \Big). \] 
	(In the terminology of statistical mechanics, the energy is \emph{stable}~\cite[Chapter 3.2]{ruelle-book69}.)
\end{proof}

The next lemma in particular shows that $\mathcal{E}_N$ is uniformly convex on $[z_{\min},z_{\max}]^{N-1}$. For later purposes, we state and prove this on a slightly larger set. 

\begin{lemma}\label{lem:hessian} 
	There are constants $\eps,\eta,C>0$ such that for all $m, N, N_1, N_2 \in \N$ with $N_1 < N_2 \le N$, and $z=(z_1,\ldots,z_{N-1}) \in [z_{\min},\infty]^{N-1}$ with $z_j \le z_{\max}+\eps$ for $N_1 \le j \le N_2-1$, the Hessian of $\mathcal{E}_N$ at $z$ satisfies 
	\bes
		\eta \sum_{j=N_1}^{N_2-1} \zeta_j^2 
		\leq \sum_{i,j=N_1}^{N_2-1} \zeta_i \zeta_j \partial_i \partial_j \mathcal{E}_N(z) 
		\leq C \sum_{j=1_1}^{N_2-1} \zeta_j^2
	\ees
	for all $\zeta \in \R^{N-1}$. Moreover,  the submatrix $(\partial_i\partial_j \mathcal{E}_N(z))_{N_1 \le i.j \le N_2-1}$ of the Hessian has strictly positive diagonal entries $\partial_i^2 \mathcal{E}_N(z) >0$ and 
	non-positive off-diagonal entries $\partial_i\partial_j \mathcal{E}_N(z) \leq 0$. In particular, this matrix is monotone.
\end{lemma} 
 \noindent Note that the Hessian is independent of the pressure $p$. 
\begin{proof}
Let $\mathcal{L}$ be the collection of discrete intervals $\{i,\ldots,j-1\}\subset\{1,\ldots,N-1\}$  of length $j-i \leq m$. Then for all $i,j$
\bes 
	\partial_i \partial_j \mathcal{E}_N(z) 
	= \sum_{L\in \mathcal{L}:\, \{i,j\}\subset L} v''\Big(\sum_{j\in L} z_j\Big).  
\ees
For $i\neq j$ and $i,j\in L$ we have $\sum_{j\in L} z_j \geq 2 z_{\min}$ hence $v''(\sum_L z_j) \le 0$; it follows that the off-diagonal entries of the Hessian are non-positive. Next we show that the row-sums are bounded from below by some constant $\eta>0$ if $N_1 \le i \le N_2-1$. 
\begin{align*}
  \sum_{j = 1}^N \partial_i \partial_j \mathcal{E}_N(z)
	  &= \partial_i^2 \mathcal{E}_N(z)+ \sum_{j: j \neq i} \partial_j \partial_i \mathcal{E}_N(z) \\
	  &= v''(z_i) + \sum_{L\ni i,\#L\geq 2} v'' \Big( \sum_{j\in L} z_j \Big) + \sum_{j: j\neq i} 
	\sum_{L\supset \{i,j\}} v'' \Big( \sum_{j\in L} z_j \Big) \\
	  &\geq v''(z_i)
     + \sum_{n = 2}^m v''(n z_{\min})
       \sum_{L\ni i,\#L=n}  \Big(1 + \sum_{j\in L, j\neq i} 1\Big)  \\
	&\geq v''(z_i) - v''(z_{\max}) + v''(z_{\max}) + \sum_{n=2}^\infty n^2 v''(n z_{\min}) =\eta.
\end{align*}
Assumption~\ref{assu:v} guarantees that $\eta>0$ for $\eps > 0$ sufficiently small. Thus row sums are positive, off-diagonal matrix elements non-positive, and consequently diagonal elements positive. Moreover, with $C = 2 \max \{ v''(r) \mid r \in [z_{\min}, z_{\max}+\eps] \}$ the diagonal elements are bounded from above by $\frac{C}{2}$. The proof of the lemma is then completed with the help of standard arguments, for example every eigenvalue of  $(\partial_i\partial_j \mathcal{E}_N(z))_{N_1 \le i.j \le N_2-1}$ lies in a Gershgorin circle with center $\partial_i^2 \mathcal{E}_N$ and radius $\sum_{j\neq i} |\partial_i\partial_j \mathcal{E}_N|$.  In particular, $(\partial_i\partial_j \mathcal{E}_N(z))_{N_1 \le i,j \le N_2-1}$ is an M-matrix and thus monotone.
\end{proof}

\begin{proof} [Proof of Theorem~\ref{thm:periodic}]
	(a) By Lemma~\ref{lem:est-on-min} minimizers lie in the compact set $[z_{\min}, z_{\max}]^{N-1}$. On that set the Hessian of $\mathcal{E}_N$ is positive definite because of 
Lemma~\ref{lem:hessian}, so $\mathcal{E}_N$ is strictly convex and the minimzer is unique. 

 (b) The convergence $z_j^\ssup{N}\to a$ as $j,N\to \infty$ along $N-j\to \infty$, where $a \in [z_{\min}, z_{\max}]$ is the unique minimizer of $\R_+\ni r\mapsto p r+ \sum_{k=1}^m v(kr)$, with the help of Lemma \ref{lem:hessian} is a straightforward adaptation of the corresponding proof in \cite{gardner-radin79} and will be omitted. By Assumption \ref{assu:v}(ii) we even have $a > z_{\min}$.  We remark that the proof in \cite{gardner-radin79} also shows that $\max\{ z_j^\ssup{N+1},\ z_{j+1}^\ssup{N+1} \} \le z_j^\ssup{N}$ for $j = 1, \ldots, N-1$. This in turn implies that the convergence is in fact uniform away from a boundary layer of vanishing volume fraction. 

	(c) This observation in combination with Lemma \ref{lem:subadditive} yields (c). Note that $e_0 < 0$ since $e_0 \le pz_{\max} + \sum_{k=1}^{\infty} v(k z_{\max}) \le p z_{\max} + v(z_{\max}) < 0$ by Assumptions \ref{assu:v} and \ref{assu:p}. 
	
\end{proof}
\noindent Notice that also $a < z_{\max}$ except for the exceptional cases in which only nearest neighbors interact, i.e.\ $m = 1$ or $v(z) = 0$ for $z \ge 2 z_{\max}$, and the pressure vanishes. 

\subsection{Surface energy} \label{sec:surface}

\begin{prop}\label{prop:surface}
	Let $m\in \N\cup \{\infty\}$ and $p\geq 0$. 
	Then 
	\bes
	\lim_{N\to \infty}(E_N- Ne_0)=   e_{\rm surf}  = 2 \inf_{\mathcal{D}_0} \mathcal{E}_\mathrm{surf} - p a -  \sum_{k=1}^m k v(ka).
	\ees
\end{prop}

\begin{proof}
	For simplicity we write down the proof for $m=\infty$; the proof when $m\in \N$ is completely analogous. 
	Fix $k\geq 2$ and $\eps>0$. Let $n_1,n_2\in \N$ with $n_2\geq k$ and $N= n_1+n_2+1$. Let $z=(z_{-n_1},\ldots, z_{n_2-1})\in [z_{\min},z_{\max}]^{n_1+n_2}$ be the spacings of the $N$-particle ground state, labelled by $j=-n_1,\ldots,n_2-1$ rather than $1,\ldots,N-1$. Choosing $n_1$ and $n_2$ large enough we may assume $\sum_{j=0}^{k-1}|z_j-a|^2\leq \eps$. Since the Hessian has matrix norm uniformly bounded from above (Lemma~\ref{lem:hessian}), changing the spacings $z_0,\ldots,z_{k-1}$ to $a$ increases the energy by $C \eps$ at most thus
	\bes
	  E_N \geq \mathcal{E}_N(z_{-n_1},\ldots,z_{-1},a,\ldots,a, z_k,\ldots,z_{n_2-1}) - C\eps.
	\ees
	We decompose the energy of the modified configuration as $A_N+ B_N+ C_N+ D_N$ where 
	\bes
	\begin{aligned} 
		A_N & = \mathcal{E}_{n_1+1} (z_{-n_1},\ldots,z_{-1}) + \mathcal{W}(z_{-n_1},\ldots,z_{-1};a,\ldots,a),\\
		B_N & = \mathcal{E}_{k+1} (a,\ldots,a) \\
		C_N & = \mathcal{W}(a,\ldots,a; z_k,\ldots,z_{n_2-1}) 
	    + \mathcal{E}_{n_2-k+1}(z_k,\ldots,z_{n_2-1})\\ 
		D_N &= \sum_{i=-n_1}^{-1}\sum_{j=k}^{n_2} v(z_i+\cdots + z_{-1} + ka + z_k+\cdots + z_j)
	\end{aligned}
	\ees
	where $\mathcal{W}$ gathers interactions that involve bonds from two consecutive blocks. The term $D_N$  represents the interactions between the left and right blocks. It satisfies 
      \bes
	0 \geq D_N \geq   \sum_{n=k}^\infty (n-k) v(n z_{\rm min}) \geq - \alpha_1 \sum_{n=k}^\infty  \frac{n-k}{(n z_{\rm min})^s} 	
    \geq - \frac{\alpha_1}{z_{\rm min}^s}\, \sum_{n=k}^\infty \frac{1}{n^{s-1}}
	\ees
	which goes to zero as $k\to \infty$. Next we subtract $Ne_0$ from $\mathcal{E}_N$  and distribute it as
	 $Ne_0= n_1 e_0 + (k+1)e_0 +(n_2-k) e_0$ over the first three sums.
	The middle block contributes
	\bes
	\begin{aligned}
	     B_N -  (k+1) e_0 
	&= \sum_{n=1}^{k} (k-n+1) v(na) + kpa - (k+1)pa - (k+1) \sum_{n=1}^\infty v(na)  \\
	&= - pa - \sum_{n=1}^{k} n v(na) - (k+1) \sum_{n=k+1}^\infty v(na) \to - \sum_{n=1}^\infty n v(na) 
	\end{aligned}
	\ees
	as $k \to \infty$. For the first block, we notice that
	\bes
	   A_N  - n_1 e_0 \geq \mathcal{E}_\mathrm{surf} (z_{-n_1},\ldots,z_{-1},a,a,\ldots) \geq \inf_{\mathcal{D}_0}\mathcal{E}_\mathrm{surf} .
	\ees
	Indeed the only missing piece are negative interactions between the left block and the right tail of a semi-infinite chain. The contribution of the right block $C_N$ is estimated in a similar way. We combine the estimates and let first $n_1,n_2 \to \infty$, then $k\to \infty$, and finally $\eps\to 0$ and find
	\bes
	    \liminf_{N\to \infty} \bigl(E_N - N e_0\bigr) \geq 2 \inf_{\mathcal{D}_0}  \mathcal{E}_\mathrm{surf} - pa - \sum_{n=1}^\infty n v(na).
	\ees
	For the upper bound, we take approximate minimizers of $\mathcal{E}_\mathrm{surf}$ and glue them together to an $N$-particle configuration by assigning them to the left and right boundaries, with spacings $a$ in between. This yields an $N$-particle configuration with energy $\mathcal{E}_N(z) - Ne_0 \leq 2 \inf_{\mathcal{D}_0} \mathcal{E}_\mathrm{surf}-  \sum_{n=1}^\infty n v(na) + O(\eps)$, and the required upper bound follows.
\end{proof}

Next we extend $\mathcal{E}_\mathrm{surf}$ to the space $\mathcal{D} \subset (r_{\rm hc}, \infty)^{\N}$ of sequences with $\sum_{j=1}^\infty(z_j-a)^2<\infty$.

\begin{lemma} \label{lem:extension}
	Let $m\in \N\cup \{\infty\}$. Let $\beta_j = \sum_{k=j+1}^m (k-j) v'(ka)$, $j=1,\ldots,m-1$. 
	Then for all $(z_j)_{j\in\N} \in \mathcal{D}_0$, we have 
    \be\label{eq:extension}
      \mathcal{E}_\mathrm{surf}((z_j)_{j\in \N}) = -  \sum_{j=1}^{m-1} \beta_j (z_j - a)
      + \sum_{j=1}^\infty \sum_{k=1}^m 
	\Bigl[ v\Bigl(\sum_{i=j}^{j+k-1} z_i\Bigr) - v(ka) - v'(ka) \sum_{i=j}^{j+k-1} (z_i-a) \Bigr].
  \ee	
The right-hand side is absolutely convergent for all $(z_j)_{j\in\N} \in \mathcal{D}$. 
\end{lemma}

\begin{proof}
 Let $\gamma_j = z_j - a$. 
  Using $e_0 = \sum_{k=1}^m v(ka)$, we have
  \bes
    \mathcal{E}_\mathrm{surf} ((z_j)_{j\in \N} )  =  \sum_{j=1}^\infty \Bigl[ p(z_j-a) + \sum_{k=1}^m \bigl( v(ka+\gamma_j+ \cdots + \gamma_{j+k-1})
	  - v(ka)\bigr) \Bigr].
  \ees
The equilibrium condition $p + \sum_{k=1}^m k v'(ka)=0$ yields 
  \begin{align*}
      & \sum_{j=1}^\infty \sum_{k=1}^m v'(ka) (\gamma_j+ \cdots + \gamma_{j+k-1}) \\ 
      & \qquad = \sum_{i=1}^\infty \gamma_i \sum_{k=1}^m v'(ka) \#\{j \geq 1\mid j \leq i \leq j+k-1\} \\
      & \qquad = \sum_{i=1}^\infty \gamma_i \sum_{k=1}^m v'(ka) \min(i,k) \\
      & \qquad = - \sum_{i=1}^{m-1} \gamma_i  \sum_{k=i+1}^{m}(k-i) v'(ka)  
        = - \sum_{i=1}^{m-1} \beta_i \gamma_i - \sum_{i=1}^{\infty} p \gamma_i 
  \end{align*}
  and the alternate expression for $\mathcal{E}_\mathrm{surf}$ follows. Next consider $(\gamma_j) \in \ell^2(\N)$ with $\gamma_j > r_{\rm hc} -a$ for all $j \in \N$. 
 Under Assumption~\ref{assu:v} the derivatives behave as $v''(r) = O(r^{-s-2})$ and $v'(r) = O(r^{-s-1})$ as $r\to \infty$ with $s>2$. It follows that $\eps_j:= \sum_{k=1}^\infty k v'(ka)$ decays like $\int_{ja}^\infty  r \times r^{-s-1} \dd r = O(j^{-s+1})$ so that $\sum_{j=1}^\infty \eps_j ^2 <\infty$. The Cauchy-Schwarz inequality then shows 
\bes
	\sum_{j=1}^{m-1} \bigl|\beta_j \gamma_j \bigr| \leq c \Bigl(\sum_{j=1}^\infty \gamma_j^2\Bigr)^{1/2}
\ees
for some suitable $m$-independent constant $c$. In particular, when $m=\infty$ the sum $\sum_j \beta_j \gamma_j$ is absolutely convergent. In order to show that the double sum over $k$ and $j$ in Eq.~\eqref{eq:extension} is absolutely convergent, we  proceed with estimates analogous to Lemma~\ref{lem:hessian}. 
 Assume first that all spacings $z_j = \gamma_j + a$ are larger than  $z_{\rm min}$. Set $\sup_{r\geq z_{\min}} |v''(r)|=c_1$ and note that, by Assumption~\ref{assu:v}(iii) for all $k\geq 2$, $\sup_{r\geq kz_{\min}} |v''(r)| \leq |v''(k z_{\min})|$. Hence
\begin{align*}
  & 2 \sum_{j=1}^\infty \sum_{k=1}^m \bigl| v(ka+\gamma_j+ \cdots + \gamma_{j+k-1}) - v(ka) - v'(ka) (\gamma_j+ \cdots + \gamma_{j+k-1}) \bigr| \\
  & \qquad \leq  c_1 \sum_{j=1}^\infty  \gamma_j^2 + \sum_{j=1}^\infty \sum_{k=2}^m |v''(k z_{\min})|
    \, ( \gamma_j + \cdots + \gamma_{j+k-1}\bigr)^2 \\
  & \qquad \leq c_1 \sum_{j=1}^\infty  \gamma_j^2 + \sum_{j=1}^\infty \sum_{k=2}^m k |v''(kz_{\min})|
      \, ( \gamma_j^2 + \cdots + \gamma_{j+k-1}^2\bigr) \\
  & \qquad  \leq  \Bigl(c_1 + \sum_{k=1}^m k^2 |v''(kz_{\min})| \Bigr) \sum_{j=1}^\infty \gamma_j^2.
\end{align*}
More generally, if $(\gamma_j)\in \ell^2(\N) \cap (r_{\rm hc}-a, \infty)^{\N}$, then $\gamma_j \to 0$ and because of { $a>z_{\rm \min}$},
there is an $i \in \N$ such that $z_j \geq z_{\min}$ for all $j\geq i$. Let $\eps = \min\{ |z_j| \mid j=1,\ldots,i\}$. Summands with $j \geq i$ can be estimated as before. For $j\leq i$ and $k \geq i+2$, we proceed as before as well, except that we  replace $v''(k z_{\min})$ by $v''((k-i)z_{\min} + i \eps)$. This leaves a finite sum over $j \leq i, k \leq i+2$ and overall, the sum is absolutely convergent.
\end{proof}

\begin{lemma} \label{lem:Esurfcont}
	The map $\mathcal{D} \to \R$, $(z_j) \mapsto \mathcal{E}_\mathrm{surf}\bigl( (z_j)_{j\in \N}\bigr)$ defined by \eqref{eq:extension} is continuous. 
\end{lemma} 

\begin{proof}
Let $z, z^\ssup{1}, z^\ssup{2}, \ldots$ be sequences in $\mathcal{D}$ such that $z^\ssup{n} - z \to 0$ in $\ell^2(\N)$. As $\lim_{i\to\infty} \sum_{j \ge i} (\gamma^\ssup{n}_j)^2 = 0$ uniformly in $n$, the estimates above show that for every $\eps>0$, we can find $i \in \N$ such that the sum over $\{(j,k) \mid j \geq i \text{ or } k \geq i\}$ contributes to $\mathcal{E}_\mathrm{surf}(\gamma^\ssup{n})$ and $\mathcal{E}_\mathrm{surf}(\gamma)$ an amount bounded by $\eps$. In the remaining finite sum the continuity of $v(r)$ allows us to pass to the limit.  The proof is easily concluded with an $\eps/3$ argument.
\end{proof} 

\begin{lemma} \label{lem:coercive}
  The restriction of $\mathcal{E}_\mathrm{surf}$ to $\mathcal{D}\cap [z_{\min},z_{\max}+\eps]^\N$ is strictly convex and satisfies 
  \bes
     \mathcal{E}_\mathrm{surf}\bigl( (z_j)_{j\in \N}\bigr) \geq c_1 \sum_{j=1}^\infty (z_j-a)^2 -c_2
  \ees
  for suitable $m$-independent constants $\eps,c_1,c_2>0$. 
\end{lemma}

\begin{proof}
	The proof of the convexity is similar to Lemma~\ref{lem:hessian} and therefore omitted. For the coercivity, consider first $m=\infty$. Let $\gamma_j = z_j -a$,  $\gamma_j^\ssup{n} = \gamma_j \1_{\{j \leq n\}}$ the truncated strain, and $z_j^\ssup{n} = a+ \gamma_j^\ssup{n}$.
   Then
  \begin{align*}
    \mathcal{E}_\mathrm{surf}(z^\ssup{n}) &= \sum_{j=1}^n \bigl(h(z_j^\ssup{n},z_{j+1}^\ssup{n},\ldots)-e_0) \\
	& = \mathcal{E}_{n+1}(z_1,\ldots,z_n) - n e_0 + \sum_{j=1}^n \sum_{k=1}^\infty v(z_j+\cdots + z_n + k a)
  \end{align*}
  thus
  \bes
   \mathcal{E}_{n+1}(z_1,\ldots,z_n) - n e_0\leq \mathcal{E}_\mathrm{surf} (z^\ssup{n}) + C
  \ees
  where $C = - \sum_{k,\ell=1}^\infty v(\ell z_{\min} + k a) <\infty$.
  Next we cut and paste $(z_1,\ldots,z_n)$ into the middle of a large ground state chain: let $k_1,k_2\in \N$ with $k_2 \geq n+1$, $N=k_2+k_1+1$  and $(z_{-k_1+1}^\ssup{N},\ldots,z_{k_2}^\ssup{N})$ the spacings of the $N$-particle ground state.  Let $z'=(z_{-k_1+1}^\ssup{N},\ldots,z_0^\ssup{N}, z_1,\ldots,z_n, z_{n+1}^\ssup{N},\ldots,z_{k_2}^\ssup{N})$. A Taylor expansion of $\mathcal{E}_N$ around the minmizer $z^\ssup{N}$ together with Lemma~\ref{lem:hessian} and Theorem~\ref{thm:periodic} yields 
  \be \label{eq:coerc-lb}
      \mathcal{E}_N(z') - \mathcal{E}_N(z^\ssup{N}) \geq \frac{\eta}{2} \sum_{j=1}^n (z_j - z_j^\ssup{N})^2
      \to \frac{\eta}{2} \sum_{j=1}^n (z_j-a)^2 \quad (k_1,k_2\to \infty).
  \ee
  On the other hand, let $C_1 = \sum_{\ell=2}^\infty \ell |v(\ell z_{\min})|$ be a bound for interactions between blocks and remember $E_k \geq k e_0$ by Lemma \ref{lem:subadditive} and $e_0 \le 0$. Then 
  \begin{align*}
      \mathcal{E}_N(z') - \mathcal{E}_N(z^\ssup{N}) 
	& \leq 2 C_1 + \mathcal{E}_{k_1+1}(z_{-k_1+1}^\ssup{N},\ldots,z_0^\ssup{N})  + \mathcal{E}_{n+1}(z_1,\ldots,z_n) \\
	&\qquad \qquad 
+ \mathcal{E}_{k_2-n+1}(z_{n+1}^\ssup{N},\ldots,z_{k_2}^\ssup{N}) - E_N \\
	& \leq 4 C_1 + \mathcal{E}_{n+1}(z_1,\ldots,z_n) - \mathcal{E}_{n+1}(z_1^\ssup{N},\ldots,z_n^\ssup{N}) \\
	& \leq 4 C_1 + \mathcal{E}_{n+1}(z_1,\ldots,z_n) - (n+1) e_0  \\
	& \leq 4 C_1 - e_0 + C + \mathcal{E}_\mathrm{surf} (z^\ssup{n}) = C_2 + \mathcal{E}_\mathrm{surf} (z^\ssup{n}).
  \end{align*}
  We combine with Eq.~\eqref{eq:coerc-lb} and let first $k_1,k_2\to \infty$, then $n\to \infty$, and conclude that $\frac{\eta}{2}\sum_{j=1}^\infty \gamma_j^2 \leq \mathcal{E}_\mathrm{surf}(z) + C_2$ with the help of Lemma \ref{lem:Esurfcont}. This proves the coercivity in the case $m=\infty$. The proof for finite $m$ is similar. 
\end{proof}

\begin{lemma}\label{lem:Esurf_min}
	The surface energy $\mathcal{E}_\mathrm{surf}$ has a unique minimizer in $\mathcal{D}$. The minimizer is in $\mathcal{D}\cap [z_{\min}, z_{\max}]^\N$. 
\end{lemma}

\begin{proof}
	We proceed as in Section~\ref{sec:bulk}. 
	Let $(z_j)_{j\in \N}\in\mathcal{D}$. If one of the $z_j$'s is larger than $z_{\max}$, we can define a new configuration by shrinking this spacing to $z_{\max}$, leaving all other configurations unchanged. This decreases $\mathcal{E}_\mathrm{surf}$. If one of the $z_j$'s is smaller than $z_{\min}$, let $b$ be the smallest among them, and $j \in \N$ with $b=z_j$. Then we can define a new configuration by removing a participating particle and possibly shrinking a bond,  i.e., $(z_1,z_2,\ldots) \mapsto (z_1,z_2,\ldots,z_{j-1}, \min(z_{j} + z_{j+1}, z_{\max}), z_{j+2}, \ldots)$. Since $e_0 \le 0$, just as in Lemma~\ref{lem:est-on-min}, we see that this decreases the energy. Repeating these steps if necessary, the initial configuration is mapped to a new one that has strictly lower energy and all spacings in $[z_{\min},z_{\max}]$.

  The existence of a minimizer now follows from the coercivity proven in Lemma~\ref{lem:coercive}, the compactness of $[z_{\min},z_{\max}]^\N\cap \mathcal{D}$ with respect to the weak $\ell^2$-convergence (shifted by $(a, a, \ldots)$) and the weak lower semicontinuity of $\mathcal{E}_\mathrm{surf}$ on that set due to Lemmas~\ref{lem:Esurfcont} and \ref{lem:coercive}. The minimizer is unique because of the strict convexity from  Lemma~\ref{lem:coercive}.
\end{proof}

\begin{proof}[Proof of Theorem \ref{thm:surface}.]
Clear from Lemmas \ref{lem:Esurfcont}, \ref{lem:coercive}, \ref{lem:Esurf_min} and Proposition \ref{prop:surface}. 
\end{proof}

\begin{proof}[Proof of Proposition \ref{prop:lim-bulk}.]
In complete analogy to Lemma \ref{lem:extension} we obtain 
    \be\label{eq:extensionbulk}
      \mathcal{E}_\mathrm{bulk}((z_j)_{j\in \Z}) = 
      \sum_{j=-\infty}^\infty \sum_{k=1}^m 
	\Bigl[ v\Bigl(\sum_{i=j}^{j+k-1} z_i\Bigr) - v(ka) - v'(ka) \sum_{i=j}^{j+k-1} (z_i-a) \Bigr].
  \ee	
for all $(z_j)_{j\in\Z} \in \mathcal{D}^+_0$, and as in Lemma \ref{lem:Esurfcont}, we see that \eqref{eq:extensionbulk} defines a continuous map $\mathcal{D}^+ \to \R$. The proof of strict convexity, even on $[z_{\min}, z_{\max} + \eps]^{\Z} \cap \mathcal{D}^+$ for some $\eps > 0$, is again similar to Lemma~\ref{lem:hessian}. As in Lemma \ref{lem:Esurf_min} we have that $\mathcal{E}_\mathrm{bulk}$ has a unique minimizer in $\mathcal{D}$, which lies in $\mathcal{D}\cap [z_{\min}, z_{\max}]^\N$. Since $a \in (z_{\min}, z_{\max}]$ and $\partial_i \mathcal{E}_\mathrm{bulk}((z_j)_{j\in \Z}) = 0$ for every $i \in \Z$ by \eqref{eq:extensionbulk}, the minimizer of $\mathcal{E}_\mathrm{bulk}$ is $(\ldots, a, a, \ldots)$. Clearly, $\mathcal{E}_\mathrm{bulk}(\ldots, a, a, \ldots) = 0$. Finally, the formula connecting $\mathcal{E}_\mathrm{bulk}$ and $\mathcal{E}_\mathrm{surf}$ is clear on $\mathcal{D}^+_0$ and follows on $\mathcal{D}^+$ by approximation. 
\end{proof}


\subsection{A fixed point equation}

In the following we assume that $v$ has a hard core: 
\begin{assumption}\label{assu:hcvtoinfty} 
$r_{\rm hc} > 0$ and $v(r) \to \infty$ as $r \searrow r_{\rm hc}$. 
\end{assumption}

\noindent We extend $h$, defined by \eqref{eq:h-def} on $(r_{\rm hc}, \infty)^{\N}$, to $\R_+^{\N}$ by setting 
	\be\label{eq:h-def-ext} 
	h(z) = \infty \mbox{ if } z_j \le r_{\rm hc} \mbox{ for some } j. 
	\ee
Our main aim in this subsection is to obtain the following characterisation of $\overline{\mathcal{E}}_\mathrm{surf}$, cf.\ \eqref{eq:Esurfextended}. 
\begin{prop}\label{prop:bellman-ersatz}
	Let $I= \overline{\mathcal{E}}_\mathrm{surf} - \min \mathcal{E}_\mathrm{surf}$. Then $I$ is the unique lower semi-continuous solution (product topology) of the equation 
	\be\label{eq:bellmaennchen}
		I(z_1,z_2,\ldots) = h(z_1,z_2,\ldots) - e_0 + I(z_2,z_3,\ldots) 
	\ee
	such that $\min I =0$ and $I = \infty$ if $z_j\leq r_\mathrm{hc}$ for one of the $z_j$'s.
\end{prop}
Note that, by induction, \eqref{eq:bellmaennchen} is equivalent to 
\be\label{eq:bellmann-ersatz-induktion} 
  I(z) 
  = \sum_{j=1}^k \big( h(z_j, z_{j+1}, \ldots ) - e_0 \big) + I( z_{k+1}, z_{k+2}, \ldots ) 
\ee 
for all $k \in \N$ and $z = (z_j)_{j\in\N} \in \R_+^{\N}$. (Observe that $h(z) > - \infty$ for all $z \in \R_+^{\N}$ by the decay assumption on $v$ and $r_{\rm hc} > 0$.) 

We begin with a technical auxiliary result. 
\begin{lemma}\label{lem:sum-bound}
  If $z_1, z_2, \ldots >0$ and $\bar{c} < \infty$ are such that
  $$
      \sup_{k \in \N} \sum_{i = 1}^k \big( h(z_{i}, \ldots, z_{m+i-1}) - e_0 \big) \le \bar{c},
  $$
then $z = (z_j)_{j\in\N} \in \mathcal{D}$. Moreover, any $z \in \mathcal{D}$ satisfies 
$$ 
   \lim_{k\to \infty} \sum_{j= 1}^k \big( h(z_{j}, \ldots, z_{j+m-1}) - e_0 \big) = \mathcal{E}_\mathrm{surf}(z). 
$$  
\end{lemma}

\begin{proof}
Let $\eps_0< \min(a-z_{\min}, z_{\max}- a)$. 
 The partial sum  $\sum_{j=1}^k h(z_j,\ldots,z_{j+m-1})$ is equal to the energy $\mathcal{E}_{k+1}(z_1,\ldots,z_{k})$ plus an interaction 
$$ \sum_{j = 1}^k \sum_{i = k+1}^{m+j-1} v( z_j + \ldots + z_i ), $$ 
(the inner sum being $0$ if $m+j-1<k+1$) which is bounded from below by 
\begin{align*}
  - \alpha_1 \sum_{j = 1}^k \sum_{i = k+1}^\infty \big( (i - j + 1) r_\mathrm{hc} \big)^{-s} 
  &\ge - C \sum_{j = 1}^k (k - j + 1)^{-s+1} \\ 
  &\ge - C \sum_{i = 1}^{\infty} i^{-s+1} 
   =: - C_1 > - \infty. 
\end{align*}
By adding $n_1$ and $n_2$ spacings $a$ to the left and right respectively, we may view $z$ as a block of spacings in an $N$-particle configuration where $N=n_1+n_2 + k +1$. Let $\hat z = (a,\ldots,a,z_1,\ldots,z_k,a,\ldots,a)$. The new configuration satisfies
\begin{align*}
  \mathcal{E}_N(\hat z)
  &\leq \mathcal{E}_{k+1}(z_1,\ldots,z_k) + 2 C_1
       + \mathcal{E}_{n_1+1}(a,\ldots,a) + \mathcal{E}_{n_2+1}(a,\ldots,a) \\ 
  &\leq C + Ne_0
\end{align*}
for some suitable constant $C$ that depends on $r_\mathrm{hc}$, $\bar c$ and $v$ only. Let $z^\ssup{N}$ be the $N$-particle ground state with spacings labelled by $j=-n_1+1,\ldots,k+n_2$ rather than $1,\ldots,N-1$. Since $\mathcal{E}_N(z^\ssup{N}) = E_N\geq N e_0$ by Lemma \ref{lem:subadditive} and $e_0 \le 0$, we get 
\bes
	 \mathcal{E}_N(\hat z) - \mathcal{E}_N(z^\ssup{N}) \leq C.
\ees
Suppose that all spacings $z_j$ are in $[z_{\min},z_{\max}]$. We use a Taylor approximation around the minimizer $z^\ssup{N}$, apply Lemma~\ref{lem:hessian} and Theorem~\ref{thm:periodic}, and obtain. 
\be \label{eq:cbound}
	C\geq \frac{\eta}{2} \sum_{j=1}^k (z_j- z_j^\ssup{N})^2 \to \frac{\eta}{2} \sum_{j=1}^k (z_j- a)^2\qquad (n_1,n_2\to \infty).
\ee
Letting $k\to \infty$ we obtain an upper bound for the $\ell^2$-norm of $(z_j-a)_{j\in \N}$. 
If there are $z_j$ with $z_j < z_{\min}$ or $z_j > z_{\max}$, we modify the configuration $z_1, \ldots, z_{k}$ without increasing its energy as in the proof of Lemma~\ref{lem:est-on-min} to obtain $z'_1, \ldots, z'_{k}$. When we shrink bonds $z_j >z_{\max}$ to $z'_j = z_{\max}$, leaving all other spacings unchanged, both $z'_j$ and $z_j$ are strictly larger than $\eps_0$ so 
the truncated $\ell^2$-norm $\sum_{j=1}^k \min \bigl( (z_j - a)^2, \eps_0^2)$ 
is unaffected. 

On the other hand suppose $z_i = \min (z_j) <z_{\min}$. Then we remove the particle $x_i$, reattach it a distance $z_{\max}$ to the left of the $k$-particle block. This effects
the change
\bes
	(z_{i-1}-a)^2 + (z_{i}-a)^2 \to  (z_{\max} - a)^2  + ((z_{i-1}+z_{i})- a)^2 
\ees
on the $\ell^2$-norm. Both $|z_{i}- a|$ and $|z_{\max}-a|$ are larger than $\eps_0$, moreover
\bes
 \min ( (z_{i-1}+z_{i}- a)^2,\eps_0^2 ) - \min ( (z_{i-1} - a)^2,\eps_0^2 ) \leq \eps_0^2.
\ees
So the truncated $\ell^2$-norm increases by at most $\eps_0^2$. Let $n$ be the number of times this step has to be performed. Iterating we arrive at a configuration $z''_1,\ldots,z''_k\in [z_{\min},z_{\max}]$ with 
\bes
	\sum_{j=1}^k \min (\eps_0^2,(z''_j-a)^2 ) \leq n\eps_0^2 + \sum_{j=1}^k \min ((z_j-a)^2,\eps_0^2)
\ees
and $\mathcal{E}_{k+1}(z'') \leq \mathcal{E}_{k+1}(z) - n \delta$
for some $\delta>0$, cf.\ \eqref{eq:dwb}. Making $\eps_0$ smaller if necessary we may assume $\eps_0^2<\delta$. We combine with Eq.~\eqref{eq:cbound} for $\hat z''$ and $C''= C-n \delta$ and obtain 
\bes
	\sum_{j=1}^k \min ((z_j-a)^2,\eps_0^2) \leq C- n \delta  +n\eps_0^2 \leq C.
\ees
We let $k\to \infty$ and find that the truncated $\ell^2$-norm of $(z_j)_{j\in \N}$ is finite. It follows in particular that there are only finitely many spacings $|z_j - a|\geq \eps_0$, and $(z_j-a)_{j\in \N}$ is square summable. This establishes the first assertion. 

In order to show the convergence of the partial sums to $\mathcal{E}_\mathrm{surf}$, first observe that $\mathcal{E}_\mathrm{surf}$ satisfies \eqref{eq:bellmann-ersatz-induktion} for $I = \mathcal{E}_\mathrm{surf}$. This is clear for $z \in \mathcal{D}_0$ and follows for general $z \in \mathcal{D}$ by continuity. If $z \in \mathcal{D}$, the sequence of shifts $((z_j)_{j \ge k})_{k \in \N}$ converges to $(\ldots, a, a, \ldots)$ strongly and thus 
\begin{align*}
  \sum_{j=1}^k \big( h(z_j, z_{j+1}, \ldots ) - e_0 \big) 
  &= \mathcal{E}_\mathrm{surf}(z) - \mathcal{E}_\mathrm{surf}( z_{k+1}, z_{k+2}, \ldots ) \\ 
  &\to \mathcal{E}_\mathrm{surf}(z) - \mathcal{E}_\mathrm{surf}(\ldots, a, a, \ldots) 
  = \mathcal{E}_\mathrm{surf}(z). 
\end{align*}
as $k \to \infty$. 
\end{proof}

\noindent We have actually proven the following: for sufficiently small $\eps_0>0$, suitable $c_1,c_2>0$, and all $(z_j)_{j\in \N} \in \R_+^\N$, 
\be \label{eq:truncoerc}
	\overline{\mathcal{E}}_\mathrm{surf}\bigl( (z_j)\bigr) \geq c_1 \sum_{j=1}^\infty \min ((z_j-a)^2,\eps_0^2) - c_2.
\ee

\begin{proof}[Proof of Proposition \ref{prop:bellman-ersatz}.]
Let $I= \overline{\mathcal{E}}_\mathrm{surf} - \min \mathcal{E}_\mathrm{surf}$. Observe that $I$ satisfies \eqref{eq:bellmaennchen}. This is clear for $z \in \mathcal{D}_0$ and for $z \notin \mathcal{D}$. For the remaining $z$ it follows from Lemma \ref{lem:Esurfcont}. We now show that $I$ is lower semi-continuous with respect to pointwise convergence. Without loss we suppose that $z^\ssup{n} \in \mathcal{D}$ converges to $z \in [r_{\rm hc}, \infty)^{\N}$ pointwise with $I(z^\ssup{n}) \le \bar{c} < \infty$ for some constant $\bar{c} > 0$. Passing to a subsequence (not relabelled) we may furthermore assume that $\liminf_{n \to \infty} I(z^\ssup{n}) = \lim_{n \to \infty} I(z^\ssup{n})$. Fix an $\eps > 0$ such that the estimate in Lemma \ref{lem:coercive} is satisfied. By \eqref{eq:truncoerc} 
\[ \max_{n \in \N} \# \{ j \mid z^\ssup{n}_j \notin [z_{\min}, z_{\max}+\eps] \} \le C \]
for some uniform constant $C > 0$ since $z_{\min} < a \le z_{\max}$. For given $N \in \N$ we denote by $j_n$ the first index $j \ge N$, if existent, with $z^\ssup{n}_{j} \notin [z_{\min}, z_{\max}+\eps]$. Passing to a further subsequence (not relabelled) and choosing $N$ sufficiently large we may achieve that either such indices do not exist or that $j_n \to \infty$ as $n \to \infty$. In both cases we get that $z_j \in [z_{\min}, z_{\max}+\eps]$ for $j \ge N$. In particular, $z_j > r_{\rm hc}$ for $j \ge N$. 

In the second case we define new configurations $\tilde{z}^\ssup{n}$ by applying the procedure detailed in the proof of Lemma \ref{lem:Esurf_min} to the tails $(z^\ssup{n}_j)_{j \ge N}$ shrinking the bonds $z^\ssup{n}_j > z_{\max}+\eps$, $j \ge N$, and deleting particles $x^\ssup{n}_{j+1}$ if $z^\ssup{n}_j < z_{\min}$, $j \ge N$, so that 
\[ {\mathcal{E}}_\mathrm{surf}((\tilde{z}^\ssup{n}_j)_{j \ge N}) 
   \le  {\mathcal{E}}_\mathrm{surf}((z^\ssup{n}_j)_{j \ge N}). \] 
In the first case we simply set $\tilde{z}^\ssup{n} = z^\ssup{n}$. Since $j_n \to \infty$ in the second case, we still have $\tilde{z}^\ssup{n} \to z$ pointwise. 

By \eqref{eq:bellmann-ersatz-induktion} with $k = N-1$ we have 
\[ I(z^\ssup{n}) 
  \ge \sum_{j=1}^{N-1} \big( h(z^\ssup{n}_j, z^\ssup{n}_{j+1}, \ldots ) - e_0 \big) + I( \tilde{z}^\ssup{n}_{N}, \tilde{z}^\ssup{n}_{N+1}, \ldots ). \]
From the decay properties of $v$ and $z^\ssup{n}_j \ge r_{\rm hc} > 0$ it is easy to see that, for any $j \in \N$, $h(z^\ssup{n}_j, z^\ssup{n}_{j+1}, \ldots)$ converges to $h(z_j, z_{j+1}, \ldots)$. Since $I(z^\ssup{n}) \le \bar{c}$ and $I \ge 0$, from Assumption \ref{assu:hcvtoinfty} we also get $z_j > r_{\rm hc}$ for $j = 1, \ldots, N-1$. So 
\[ \sum_{j=1}^{N-1} \big( h(z^\ssup{n}_j, z^\ssup{n}_{j+1}, \ldots ) - e_0 \big) 
   \to \sum_{j=1}^{N-1} \big( h(z_j, z_{j+1}, \ldots ) - e_0 \big).  
\] 
In particular, $I( (\tilde{z}^\ssup{n}_{j})_{j \ge N}) \le C$ and so Lemma \ref{lem:coercive} implies that $z \in \mathcal{D}$ and $\tilde{z}^\ssup{n} - z \rightharpoonup 0$ in $\ell^2$ by coercivity and hence that 
\[ \liminf_{n \to \infty} I( (\tilde{z}^\ssup{n}_{j})_{j \ge N}) 
   \ge  I( (z_{j})_{j \ge N}) \] 
by convexity. Summarizing we obtain 
\[ \liminf_{n \to \infty} I(z^\ssup{n}) 
  \ge \sum_{j=1}^{N-1} \big( h(z_j, z_{j+1}, \ldots ) - e_0 \big) + I( z_{N}, z_{N+1}, \ldots ) 
  = I(z). \]

Suppose, conversely, that a lower semi-continuous $I : \R_+^{\N} \to \R \cup \{+\infty\}$ satisfies \eqref{eq:bellmaennchen} with $\min I = 0$ and $I(z) = \infty$ if $z_j \le r_\mathrm{hc}$ for some $j$. We first note that, 
since $I \ge 0$, for any $z$ with $I(z) < \infty$ one has 
$$ \sup_{k \in \N} \sum_{j=1}^k \big( h(z_j, z_{j+1}, \ldots ) - e_0 \big) < \infty $$
by \eqref{eq:bellmann-ersatz-induktion} and so $z \in \mathcal{D}$ by Lemma \ref {lem:sum-bound}. It thus suffices to show that 
\be\label{eq:bellmann-ersatz-NTS}  
  I(z) = \mathcal{E}_\mathrm{surf}(z) + I(a, a, \ldots) 
\ee 
for all $z \in \mathcal{D}$. 

If $z \in \mathcal{D}$, then $\mathcal{E}_\mathrm{surf}(z)$ is indeed finite by Lemma~\ref{lem:extension}. We have $\lim_{k\to \infty} \sum_{j= 1}^k \big( h(z_{j}, \ldots, z_{j+m-1}) - e_0 \big) = \mathcal{E}_\mathrm{surf}(z)$ by Lemma \ref{lem:sum-bound}. Since the sequence of shifts $( (z_j)_{j \ge k} )_{k \in \N}$ converges to $(a, a, \ldots)$ pointwise as $k \to \infty$, taking the $\liminf$ in \eqref{eq:bellmann-ersatz-induktion} yields 
\bes 
  I(z) 
  = \lim_{k \to \infty} \sum_{j=1}^k \big( h(z_j, z_{j+1}, \ldots ) - e_0 \big) 
       + \liminf_{k \to \infty} I( z_{k+1}, z_{k+2}, \ldots ) 
   \ge \mathcal{E}_\mathrm{surf}(z) + I(a, a, \ldots). 
\ees
Note that, as $I \not\equiv \infty$, this inequality also shows that $I(a, a, \ldots) < \infty$. 

For the reverse inequality, by choosing $k$ large enough in \eqref{eq:bellmann-ersatz-induktion} we first see that \eqref{eq:bellmann-ersatz-NTS} holds true for all $z \in \mathcal{D}_0$. We denote by $z^{(N)}$ the truncation with $z^{(N)}_j = z_j$ for $j \le N$ and $z^{(N)}_j = a$ for $j \ge N+1$. Since $z^{(N)} \to z$ pointwise and $z^{(N)} - z \to 0$ in $\ell^2$ as $N \to \infty$, lower semi-continuity of $I$ and strong continuity of $\mathcal{E}_\mathrm{surf}$ (see Lemma \ref{lem:Esurfcont}) give 
$$ I(z) 
   \le \liminf_{N \to \infty} I(z^{(N)}) 
   = \liminf_{N \to \infty} \mathcal{E}_\mathrm{surf}(z^{(N)}) + I(a, a, \ldots) 
   = \mathcal{E}_\mathrm{surf}(z) + I(a, a, \ldots), $$
where we have used that $z^{(N)} \in \mathcal {D}_0$ for all $N$. 
\end{proof}

We now restrict to the case $m < \infty$. Let $d = m-1$.
By \eqref{eq:bellmann-ersatz-induktion} with $k = d$ we have
\begin{align}\label{eq:bellmann-ersatz-k-gleich-m}
\begin{split}
  \mathcal{E}_{\rm surf}((z_j)_{j \in \N})
  &= \sum_{j = 1}^d \big( h(z_j, \ldots, z_{j+d}) - e_0 \big) + \mathcal{E}_{\rm surf}(z_{d+1}, z_{d+2}, \ldots) \\ 
  &= \mathcal{E}_{d+1}(z_1, \ldots, z_d) - d e_0 + W(z_1, \ldots, z_d; z_{d+1}, \ldots, z_{2d}) \\ 
  &\ \ \    + \mathcal{E}_{\rm surf}((z_j)_{j \ge d+1}), 
\end{split}
\end{align}
for any $(z_j)_{j \in \N} \in \mathcal{D}$, where 
$$
	W(z_1, \ldots, z_{d}; z_{d+1}, \ldots, z_{2d}) 
   = \sum_{1 \le i \le d < j \le 2d \atop j - i \le d} v(z_i + \ldots + z_j).
$$   
Taking the infimum over $(z_j)_{j \in \N} \in \mathcal{D}_0$, with fixed $z_1, \ldots, z_d$ and setting
\begin{align*}
  u(x) 
  &= \inf \big\{ \mathcal{E}_{\rm surf}((z_j)_{j \in \N}) \mid (z_j)_{j \in \N} \in \mathcal{D}_0,~ (z_1, \ldots, z_d) = x \big\} \\ 
  &= \inf \big\{ \mathcal{E}_{\rm surf}((z_j)_{j \in \N}) \mid (z_j)_{j \in \N} \in \mathcal{D},~ (z_1, \ldots, z_d) = x \big\} 
\end{align*}
(recall Lemma \ref{lem:Esurfcont}) leads to
$$ u(x) = \inf_{y \in \R^d_+} \big( \mathcal{E}_{d+1}(x) + W(x; y) - d e_0 + u(y) \big). $$

In Chapter \ref{sec:gaussian} we will need the following estimate.

\begin{lemma} \label{lem:unibell}
Set $A_{\eps} = [z_{\min}, z_{\max} + \eps]^d$ and $B_{\eps} = \R_+^d \setminus A_{\eps}$. Then, for any $\eps > 0$ there exists a $\delta > 0$ such that
$$ \inf_{y \in B_{\eps}} \big( \mathcal{E}_{d+1}(x) + W(x; y) - d e_0 + u(y) \big) \ge u(x) + \delta $$
for all $x \in A_{\eps}$.
\end{lemma}

\begin{proof}
Suppose $(z_j)_{j \in \N} \in \mathcal{D}_0$ is such that $(z_1, \ldots, z_d) \in A_{\eps}$, in particular, $z_j \ge z_{\min}$ for $j = 1, \ldots, d$. If $(z_{d+1}, z_{d+2}, \ldots) \notin [z_{\min}, z_{\max}+\eps]^{\N}$ we construct a new configuration $(z'_j)_{j \in \N} \in \mathcal{D}_0$ without changing the first $d$ spacings similarly as in the proofs of Lemma \ref{lem:est-on-min} and \ref{lem:Esurf_min}. 

If $z_i > z_{\max} + \eps$, we define $(z'_j)_{j \in \N}$ by setting $z'_j = z_j$ for $j \ne i$ and $z'_i = z_{\max}$. Then
\begin{align}\label{eq:u-too-long}
  \mathcal{E}_{\rm surf}((z'_j)_{j \in \N}) 
  \le \mathcal{E}_{\rm surf}((z_j)_{j \in \N}) + v(z_{\max}) - v(z_{\max} + \eps). 
\end{align}
Now assume $b = \min\{z_{d+1}, z_{d+2}, \ldots\} < z_{\min}$. We choose an $i \ge d+1$ with $z_i = b$ and define $(z'_j)_{j \in \N}$ by setting $z'_j = z_j$ for $j < i$, $z'_i = \min\{z_i+z_{i+1}, z_{\max}\}$ and $z'_j = z_{j+1}$ for $j > i$. As in Lemmas \ref{lem:est-on-min} and \ref{lem:Esurf_min} (in particular using that $e_0 \le 0$), we see that 
\begin{align}\label{eq:u-too-short}
\begin{split}
  \mathcal{E}_{\rm surf}((z'_j)_{j \in \N}) 
  &\le \mathcal{E}_{\rm surf}((z_j)_{j \in \N}) - \Big( v(b) + v(z_{\max}) - 2 \alpha_1 \sum_{n=2}^m (nb)^{-s} \Big) \\ 
  &\le \mathcal{E}_{\rm surf}((z_j)_{j \in \N}) - 2 \alpha_1 \sum_{n=m+1}^{\infty} (nz_{\min})^{-s}. 
\end{split}
\end{align}

The estimates \eqref{eq:u-too-long} and \eqref{eq:u-too-short} show that, for any $(z_j)_{j \in \N} \in \mathcal{D}_0$ with $(z_1, \ldots, z_d) \in A_{\eps}$ and $(z_{d+1}, \ldots, z_{2d}) \in B_{\eps}$ there is a $(z'_j)_{j \in \N} \in \mathcal{D}_0$ with $(z'_1, \ldots, z'_d) = (z_1, \ldots, z_d)$ such that 
\begin{align*}
  \mathcal{E}_{\rm surf}((z'_j)_{j \in \N}) 
  &\le \mathcal{E}_{\rm surf}((z_j)_{j \in \N}) - \delta, 
\end{align*}
where $\delta = \min \big\{ v(z_{\max} + \eps) - v(z_{\max}), \ 2 \alpha_1 \sum_{n=m+1}^{\infty} (nz_{\min})^{-s} \big\} > 0$. Using \eqref{eq:bellmann-ersatz-k-gleich-m} we arrive at 
\begin{align*}
  u(z_1, \ldots, z_d) + \delta 
  \le \mathcal{E}_{d+1}(z_1, \ldots, z_d) - d e_0 + W(z_1, \ldots, z_d; z_{d+1}, \ldots, z_{2d})  
  + \mathcal{E}_{\rm surf}((z_j)_{j \ge d+1}).   
\end{align*}
The claim now follows by taking the infimum over $(z_j)_{j \in \N}$ with fixed $(z_1, \ldots, z_d)$ conditioned on $(z_{d+1}, \ldots, z_{2d}) \in B_{\eps}$. 
\end{proof}

\noindent A simpler proof gives the following estimate that  will also be needed in Chapter \ref{sec:gaussian}. 
\begin{lemma}\label{lem:minEbulk-quantest}
For any $\eps > 0$ there exists a $\delta > 0$ such that $\mathcal{E}_\mathrm{bulk}(z) \ge \delta$ for all $z \in \mathcal{D}^+ \setminus [z_{\min}, z_{\max} + \eps]^{\Z}$.
\end{lemma}

\begin{proof}
By continuity we may assume that $z = (z_j)_{j \in \Z} \in \mathcal{D}^+_0 \setminus [z_{\min}, z_{\max} + \eps]^{\Z}$. If $z_i > z_{\max} + \eps$, we define $z' = (z'_j)_{j \in \Z}$ by setting $z'_j = z_j$ for $j \ne i$ and $z'_i = z_{\max}$. Then
\begin{align*}
  0 \le \mathcal{E}_{\rm bulk}(z') 
  \le \mathcal{E}_{\rm bulk}(z) + v(z_{\max}) - v(z_{\max} + \eps). 
\end{align*}
If $b = \min\{z_j : j \in \Z\} < z_{\min}$. We choose the smallest $i$ with $z_i = b$ and define $z = (z'_j)_{j \in \N}$ by setting $z'_j = z_j$ for $j < i$, $z'_i = \min\{z_i+z_{i+1}, z_{\max}\}$ and $z'_j = z_{j+1}$ for $j > i$. As in \eqref{eq:u-too-short} we get  
\begin{align*}
  0 \le \mathcal{E}_{\rm bulk}(z') 
  \le \mathcal{E}_{\rm bulk}(z) - 2 \alpha_1 \sum_{n=m+1}^{\infty} (nz_{\min})^{-s}. 
\end{align*}
This concludes the proof. 
\end{proof}


\section{Gibbs measures for the infinite and semi-infinite chains} \label{sec:gibbs} 

Here we prove the existence of $\nu_\beta$, $\mu_\beta$, $g(\beta)$, $g_\mathrm{surf}(\beta)$  and check that $\mu_\beta$ is shift-invariant and mixing, hence ergodic; the results and methods are fairly standard.
 In addition, we provide an a priori estimate on the decay of correlations with explicit analysis of the $\beta$-dependence (Theorem~\ref{thm:ergodicity}) which to the best of our knowledge is new. 
The results from this section need only very little on the pair potential: we only use that $v$ has a hard core and that $v(r)= O(1/r^{s})$, for large $r$, with $s>2$. The technical assumption of a hard core  frees us from superstability estimates~\cite{presutti-lebowitz76,ruelle76}. The decay of the potential ensures that the infinite volume Gibbs measure is unique, see e.g.~\cite[Chapter 8.3]{georgii-book} and~\cite{papangelou84a,papangelou84b,klein85}. 

We follow the classical treatment of one-dimensional systems with transfer operators. 
 For compactly supported pair potentials with a hard core (or, in our case, when $m$ is chosen finite), the transfer operators are integral operators in $L^2(\R_+^{m-1},\dd x)$~\cite[Chapter 5.6]{ruelle-book69}, see Section~\ref{sec:gaussian}. For long-range interactions, the transfer operator (also known as \emph{Ruelle operator} or \emph{Ruelle-Perron-Frobenius operator}) acts instead from the left on functions of infinitely many variables, and from the right on  measures~\cite{ruelle68,gallavotti-miraclesole70,ruelle-book78}. 
The formalism of transfer operators keeps being developed in the context of dynamical systems and ergodic theory~\cite{baladi-book,baladi00}.
 
 For the decay of correlations, we adapt~\cite{pollicott00} to the present context of continuous unbounded spins and carefully track the $\beta$-dependence in the bounds. 
 In Section~\ref{sec:ldp}, transfer operators will also help us investigate the large deviations behavior of the Gibbs measures; notably the eigenvalue equation from Lemma~\ref{lem:eigenmeas} translates into a fixed point equation for the rate function (see Lemma~\ref{lem:fpe}).

The results of this section hold for all $m\in \N\cup \{\infty\}$ and $\beta,p>0$; the additional condition $p<p^*$ is not needed. 

\subsection{Transfer operator} \label{sec:transferoperator-infinite}

For $j \in \Z$ and $z_j,z_{j+1},\ldots > 0$ we abbreviate 
$h_j = h(z_j,z_{j +1},\ldots)$, cf.\ \eqref{eq:h-def} and \eqref{eq:h-def-ext}. 
The transfer operator acts on functions as 
\bes
	\mathcal{L}_\beta f(z_1,z_2,\ldots) = \int_0^\infty \dd z_0\, \e^{-\beta h_0 } f(z_0,z_1,\ldots).
\ees
The dual action on measures is defined by $(\mathcal{L}_\beta^*\nu)(f)= \nu(\mathcal{L}_\beta f)$ and is given by
\bes 
  {\mathcal L}_\beta ^*  \nu (\dd z_1 \dd z_2...) = e^{-\beta h_1}\, \dd z_1  \nu (\dd z_2 \dd z_3 ...).
\ees

\begin{lemma} \label{lem:eigenmeas}
	There exist $\lambda_0(\beta)>0$ and a probability measure $\nu_\beta$ on $\R_+^\N$ such that 
	\bes 
		 \mathcal{L}_\beta^* \nu_\beta = \lambda_0(\beta) \nu_\beta.
	\ees	
	Moreover $\nu_\beta((r_\mathrm{hc},\infty)^\N) =1$ and the pair $(\nu_\beta,\lambda_0(\beta))$ is unique. 
\end{lemma} 

\noindent We will show in Proposition \ref{prop:thermolim} that $\nu_\beta$ is the measure satisfying \eqref{eq:nutherm}. The non-compactness of $(r_\mathrm{hc},\infty)^\N$ forms an obstacle to the application of a Schauder-Tychonoff fixed point theorem for the map $\nu\mapsto \mathcal{L}_\beta^* \nu /\nu(\mathcal{L}_\beta \mathbf{1})$, see e.g.~\cite[Proposition~2]{ruelle68}. It might be possible to remove the obstacle using tightness estimates, but we prefer to follow a different route and exploit the known uniqueness of infinite volume Gibbs measures~\cite[Chapter 8.3]{georgii-book} instead. 

\begin{proof}
Let $\nu$ be a probability measure on $\R_+^\N$, $\lambda:= \nu(\mathcal{L}_\beta \mathbf{1})$, and $\tilde\nu:= \frac{1}{\lambda}\mathcal{L}_\beta^* \nu$. 	We show that if $\nu$ is a Gibbs measure, then $\tilde \nu$ is a Gibbs measure as well.  Let us first introduce the kernels needed to formulate that $\nu$ is a Gibbs measure. By~\cite[Theorem 1.33]{georgii-book} it is enough to look at one-point kernels. 
	Pick $k\in \N$. For $z'_k>0$ and  $z= (z_j)_{j\in \N}\in \R_+^\N$, let 
	$$	H_k(z'_k\mid z) = p z'_k + \sum_{J\subset \N,\, J\ni k} v\Bigl( z'_k + \sum_{j\in J\setminus\{k\}} z_j\Bigr) $$
	where sum runs over discrete intervals $J=\{i,\ldots, \ell-1\}\subset \N$.
	Further define the kernel 
	$$ \gamma_k\bigl(z,A\bigr) = \frac{1}{N_k(z)} 
		\int_{0}^\infty \1_A\bigl( \ldots,z_{k-1},z'_k, z_{k+1},\ldots) \e^{- \beta H_k(z'_k \mid z)} \dd z'_k
	$$
	where $A\subset \R_+^\N$ and $N_k(z)  = \int_{0}^\infty  \e^{- \beta H_k(z'_k \mid z)} \dd z'_k$.
	The kernel acts on functions and measures in the usual way, in particular $(\gamma_k \1_A)(z) = \gamma_k(z,A)$. 
	Notice that $\gamma_k^2 f = \gamma_k f$ for all $f$. 
 Indeed $\gamma_k f$ yields a function where $z_k$-dependence has been integrated out, and integrating it against the probability measure  $\gamma_k(z,\cdot)$ does not change its value. Replacing $\N$ with $\N_0$, we define in a completely analogous fashion conditional energies $H_k^0$ and kernels $\gamma_k^0\bigl( (z_j)_{j\in \N_0}, B\bigr)$. 
	
	Suppose that $\nu$ is a Gibbs measure, i.e., $\nu \gamma_k = \nu$ for all $k\in \N$. Let $f:\R_+^{\N_0}\to \R_+$ be a measurable test function. Treat $\tilde \nu = \lambda^{-1} \mathcal{L}_\beta^*\nu$ as a measure on $\R_+^{\N_0}$.	We check that $\tilde \nu(\gamma_k^0 f) = \tilde \nu(f)$ for all $k\in \N_0$. For $k\in \N$, this property is inherited from the Gibbsianness of $\nu$: we have 
	$$
		\tilde \nu(f) = \frac{1}{\lambda}\int_0^\infty \nu\Bigl( f(z_0,\cdot)  \e^{-\beta h(z_0,\cdot)}\Bigr) \dd z_0 
			 = \frac{1}{\lambda}\int_0^\infty \nu \gamma_k \Bigl( f(z_0,\cdot)  \e^{-\beta h(z_0,\cdot)}\Bigr) \dd z_0.
	$$
	Set $\tilde f:= \gamma_k^0 f$. Note $\tilde f = (\gamma_k^0) \tilde f$.  Therefore 
	\begin{align*}
		\gamma_k \Bigl( f(z_0,\cdot)  \e^{-\beta h(z_0,\cdot)}\Bigr) (z) 
		& = 	( \gamma_k^0 f)(z_0,z)	\times (\gamma_k 	 \e^{-\beta h(z_0,\cdot)})(z)\\
		& = \gamma_k \Bigl( \tilde f(z_0,\cdot)  \e^{-\beta h(z_0,\cdot)}\Bigr) (z)
	\end{align*}
	hence $\tilde \nu(f) = \tilde \nu(\tilde f) = \tilde \nu(\gamma_k^0 f)$.
	For $k=0$, the required property follows from the definition of $\tilde \nu$. Notice $H_0^0 =h_0$ and 
		$$(\gamma_0^0 f)\bigl( (z_j)_{j\in \N_0}\bigr) = \frac{\int_0^\infty f(z'_0, z_1,z_2,\ldots) \e^{-\beta h(z'_0,z_1,\ldots)} \dd z'_0}{\int_0^\infty \e^{-\beta h(z'_0,z_1,\ldots)} \dd z'_0}.$$
	Let $\tilde f= \gamma_0^0 f$. Then 
	\begin{align*}
	   \tilde \nu(f) &=\frac{1}{\lambda} \nu \Bigl( \int_0^\infty f(z_0,\cdot)  \e^{-\beta h(z_0,\cdot)}\dd z_0 \Bigr)  
	    = \frac{1}{\lambda} \nu \Bigl( \int_0^\infty  \tilde f(z_0,\cdot)  \e^{-\beta h(z_0,\cdot)}\dd z_0 \Bigr)\\
	   & = \tilde \nu(\tilde f) = \tilde \nu(\gamma_0^0 f). 
	\end{align*}
	The previous identities hold for all non-negative test functions $f$, consequently $\tilde \nu \gamma_k^0 = \tilde \nu$ for all $k\in \N_0$ and $\tilde \nu$ is a Gibbs measure as well. 

	By~\cite[Theorem 8.39]{georgii-book}, the Gibbs measure $\nu$ exists and is unique.
 Treating $\nu$ and $\tilde \nu$ both as measures on $\R_+^\N$, we must therefore have $\nu = \tilde \nu$, i.e., the unique Gibbs measure is an eigenmeasure of $\mathcal{L}_\beta^*$ and in particular, there exists an eigenmeasure. 
	Conversely, let $\nu = \frac{1}{\lambda} \mathcal{L}_\beta^*\nu$ be an eigenmeasure. Arguments similar to the investigation of $\tilde \nu$ given above, based on the iterated fixed point equation $\nu = \frac{1}{\lambda^k} {\mathcal{L}_\beta^*}^k\nu$, show that $\nu \gamma_j = \nu$ for all $j=1,\ldots,k$ and all $k$, hence for all $j$. Every eigenmeasure is a Gibbs measure. Since the latter is unique, the eigenmeasure is unique as well. Finally, since $v(z_j) = \infty$ for $z_j\leq r_\mathrm{hc}$, the eigenmeasure $\nu = \frac{1}{\lambda^k}{\mathcal{L}_\beta^*}^k \nu$ must satisfy $\nu(\exists j\in \{1,\ldots,k\}:\, z_j\leq r_\mathrm{hc}) = 0$. This holds for all $k\in \N$, hence $\nu( (r_\mathrm{hc},\infty)^\N) =1$. 
\end{proof}	
 
Let $\nu_\beta^-$ be the probability measure on $\R_+^{\{\ldots, -1,0\}}$ obtained by flipping $\nu_\beta^+= \nu_\beta$, i.e., $\nu_\beta^-$ is the image of $\nu_\beta^+ =\nu_\beta$ under the map $(z_k)_{k\in \N} \mapsto (z_{1-\ell})_{\ell \leq 0}$. The measures $\nu_\beta^\pm$ represent equilibrium measures for the left and  right half-infinite chains. Let 
\bes 
  \mathcal{W}_0 = \mathcal{W}( \cdots z_{-1} z_0 \mid z_1 z_2 \ldots ) := \sum_{\heap{j \leq 0, k \geq 1}{|k-j|\leq m-1}} v(z_j+\cdots + z_k)
\ees
be the total interaction between left and right half-infinite chains,  cf.\ Proposition \ref{prop:lim-bulk}(d). We abbreviate the shifted versions as  $\mathcal{W}_\ell = \mathcal{W}(\cdots z_\ell \mid z_{\ell+1} \cdots)$.
Define $\varphi_\beta(z_1,z_2,\ldots)$ by 
\be \label{eq:phidef}
  \varphi_\beta(z_1,z_2,\ldots) = \frac{\nu_\beta^-(\exp( - \beta \mathcal{W}_0))}{\nu_\beta^- \otimes \nu_\beta^+( \exp( - \beta \mathcal{W}_0))}.
\ee
Thus $\varphi_\beta(z_1,z_2,\ldots)$ represents an averaged contribution to the Boltzmann weight from the left half-infinite chain. 

\begin{lemma} \label{lem:eigenphi}
	We have  $\mathcal{L}_\beta \varphi_\beta = \lambda_0(\beta) \varphi_\beta$ and $ \nu_\beta(\varphi_\beta) = 1$.
\end{lemma}

\begin{proof}  
The normalization is obvious, for the eigenvalue equation let $c_\beta= \nu_\beta^- \otimes \nu_\beta^+( \exp( - \beta \mathcal{W}_0))$ and use the eigenvalue equation for $\nu_\beta^\pm$
\begin{align*} 
   &\varphi_\beta(z_1,z_2,\ldots) \\ 
			&\qquad = \frac{1}{c_\beta} \int \e^{- \beta\mathcal{W}(\cdots z_0 \mid z_1\cdots)} \dd \nu_\beta^-\bigl( (z_j)_{j\leq0} \bigr) \\
      &\qquad = \frac{1}{c_\beta \lambda_0(\beta)} \int \e^{- \beta\mathcal{W}(\cdots z_0 \mid z_1\cdots)} \e^{- \beta (p z_0 + v(z_{0})+ v(z_0+z_{-1}) + \cdots )}\dd z_0 \dd \nu_\beta^-\bigl( (z_j)_{j\leq -1} \bigr) \\
      &\qquad = \frac{1}{c_\beta \lambda_0(\beta)} \int \e^{- \beta\mathcal{W}(\cdots z_{-1}\mid z_0  z_1\cdots)} \e^{-\beta (pz_0 + v(z_0)+ v(z_0+z_1)+\cdots)} \dd z_0  \dd \nu_\beta^-\bigl( (z_j)_{j\leq -1} \bigr) \\
      &\qquad = \frac{1}{\lambda_0(\beta)} \int \e^{-\beta h_0} \varphi_\beta(z_0,z_1,\ldots) \dd z_0 \\
      &\qquad = \frac{1}{\lambda_0(\beta)}(\mathcal{L}_\beta \varphi_\beta)(z_1,z_2,\ldots).
\end{align*}
See also~\cite[Section 5.12]{ruelle-book78}.
\end{proof} 
\noindent 
Define the operator   
\bes 
	\mathcal{S}_\beta f:= \frac{1}{\lambda_0(\beta) \varphi_\beta} \mathcal{L}_\beta(\varphi_\beta f)
\ees
so that $\mathcal{S}_\beta\mathbf{1}= \mathbf{1}$ and $\mathcal{S}_\beta^*(\varphi_\beta \nu_\beta^+) = \varphi_\beta \nu_\beta^+$. Let $\mu_\beta$ be the probability measure on $\R_+^\Z$  given by 
\be \label{eq:mudef}
  \frac{\dd \mu_\beta}{\dd \nu_\beta^-\otimes \nu_\beta^+} = \frac{1}{c_\beta} \e^{-\beta \mathcal{W}_0}, \quad c_\beta=\nu_\beta^-\otimes\nu_\beta^+(\e^{- \beta \mathcal{W}_0}).
\ee 
We will show in Proposition \ref{prop:thermolim} that $\mu_\beta$ is the measure satisfying \eqref{eq:mutherm}. Notice that for every bounded measurable function $f$ that depends on right-chain variables $z_1,z_2,\ldots$ only,
\be\label{eq:numu}
  \mu_\beta(f) = \nu_\beta^+(f \varphi_\beta),\quad \nu_\beta^+(f) = \frac{\mu_\beta(\e^{\beta \mathcal{W}_0} f)}{\mu_\beta(\e^{\beta \mathcal{W}_0})}.
\ee
Let $\tau:\R_+^\Z\to \R_+^\Z$ be the shift $(\tau z)_j= z_{j+1}$.

\begin{lemma}\label{lem:mu-S-shift} \hfill
	\begin{itemize}
		\item [(a)] $\mu_\beta$ is shift-invariant.
		\item [(b)] For all $f,g:\R_+^\N\to \R_+$ and all $n\in \N$, we have $\mu_\beta(f (g\circ \tau^n)) = \mu_\beta((\mathcal{S}_\beta^n f)g)$. 
	\end{itemize}
\end{lemma} 

The proof is standard~\cite{ruelle-book78} and therefore omitted. The lemma can be rephrased as follows: let $(Z_n)_{n\in \Z}$ be a stochastic process with law $\mu_\beta$, defined on some probability space $(\Omega,\mathcal{F},\P)$.  Then $(Z_n)_{n\in \Z}$ is stationary, and  
\bes
	\bigl( \mathcal{S}_\beta^n f\bigr)(Z_{n+1},Z_{n+2},\ldots) =  \mathbb{E}\Bigl[ f(Z_1,Z_2,\ldots)\,  \Big|\, Z_{n+1}, Z_{n+2},\ldots \Bigr]\quad \text{a.s.}
\ees
Our next task is to show that the  process is not only stationary but in fact ergodic and to estimate the decay of correlations. 

\subsection{Ergodicity} \label{sec:ergodicity}

Bounds on correlations are most conveniently expressed with the help of \emph{variations}, semi-norms that quantify how much a function depends on faraway variables. Notice that $\nu_{\beta}((r_{\rm hc}, \infty)^{\N}) = \mu_{\beta}((r_{\rm hc}, \infty)^{\Z}) = 1$. Let $f:\R_+^\N \to \R$ be a function and $n \in \mathbb N$.  The $n$th variation of $f$  on $(r_{\rm hc}, \infty)^{\N}$ is 
\bes
	 \var_n(f):= \sup\{|f(z) - f(z')|\, :\, z, z' \in (r_{\rm hc}, \infty)^{\N} \text{ such that }  z_1 = z'_1,\ldots,   z_n= z'_n\}. 
\ees
When $n=0$ the constraint on initial values is empty, $\var_0(f)$ is sometimes called the \emph{oscillation} of $f$~\cite[Eq. (8.2)]{georgii-book}. The oscillation vanishes if and only $f$ is constant.  Notice that $\var_k(h)$ decays algebraically: for $k\in \N$, as $v(r) = O(r^{-s})$,
\bes
	\var_k(h)  \leq 2 \sup_{z} \Bigl|\sum_{j=k+1}^\infty v(z_1+\cdots + z_j)\Bigr| = O\Bigl(\frac{1}{k^{s-1}}\Bigr). 
\ees
It follows that the variation is summable, $\sum_{k=1}^\infty \var_k(h) <\infty$. 
Set 
\bes
   C_ q:= \sum_{k=q+1}^\infty \var_k (h) = O\Bigl( \frac{1}{q^{s-2}}\Bigr).
\ees
Notice that for all $q\in \N_0$, $C_q$ is independent of $\beta$ and $p$. In fact the pressure only enters the oscillation $\var_0 (h)$. By a slight abuse of notation we identify a function $f:\R_+^\N\to \R$ with the function $f_1: \R_+^\Z\to \R_+$, $(z_{j})_{j\in \Z}\mapsto f( (z_j)_{j\in \N})$ and write $\mu_\beta(f)$ instead of $\mu_\beta(f_1)$. The results of this subsection hold for all $p>0$. 

\begin{theorem} \label{thm:ergodicity}
		Let $m\in\N \cup \{\infty\}$ and $p>0$. 
	The measure $\mu_\beta$ is mixing with respect to shifts, i.e., $\mu_\beta(f (g\circ \tau^n)) \to \mu_\beta(f) \mu_\beta(g)$ as $n\to \infty$, for all $f,g\in L^1(\R_+^\Z,\mu_\beta)$. Moreover for $\gamma(\beta) = \exp( - 3 \beta C_0)$ and all bounded $f,g:\R_+^\N \to \R$, $q,n\in \N$, $N\geq qn$, 
	\begin{multline*}
		\bigl|\mu_\beta\bigl( f (g\circ \tau^N)\bigr) - \mu_\beta(f) \mu_\beta(g)\bigr| 
			\leq  \Bigl( (1-\gamma(\beta))^q + \frac{1}{\gamma(\beta)}(\e^{3\beta C_n} - 1)\Bigr) ||g||_\infty ||f||_\infty \\
		+ \frac{1}{\gamma(\beta)} ||g||_\infty \var_n(f).
	\end{multline*}
\end{theorem}

We prove Theorem~\ref{thm:ergodicity} with 
Pollicott's \emph{method of conditional expectations}~\cite{pollicott00}. For alternative approaches, see~\cite{sarig02} and the references therein. The principal idea is the following: for $n\in \N$,  $f \in L^1(\R_+^\N,\varphi_\beta \nu_\beta)$ let $\Pi_n f$ be the projection 
\[ \bigl(\Pi_n f\bigr)(z_1,\ldots,z_n) 
   = \frac{\int_{\R_+^{\N}} \varphi_{\beta}(z_1, \ldots) f(z_1, \ldots) \e^{-\beta(h_1+\ldots+h_n)} \nu_{\beta}(\dd z_{n+1} \ldots)}{\int_{\R_+^{\N}} \varphi_{\beta}(z_1, \ldots) \e^{-\beta(h_1+\ldots+h_n)} \nu_{\beta}(\dd z_{n+1} \ldots)} \] 
onto the subspace of functions that depend on the first $n$ coordinates only, i.e., $\var_{n}(f) =0$. In terms of the stationary process $(Z_n)_{n\in \Z}$ with law $\mu_\beta$, 
\bes
	\bigl(\Pi_n f\bigr)(Z_1,\ldots,Z_n)  = \mathbb{E}\bigl[ f( (Z_j)_{j\geq 1})\, \big|\,  Z_1,\ldots,Z_n\bigr]\quad \text{a.s.}
\ees
Notice that 
\be \label{eq:pinvarn}
	||\Pi_n f- f||_1\leq ||\Pi_n f- f||_\infty \leq \var_n(f) 
\ee
where $||\cdot||_1$ is the  $L^1(\R_+^\N, \varphi_\beta \nu_\beta)$ norm. Let $q,n\in \N$. Then 
\bes \label{eq:smn}
	\mathcal{S}_\beta^{qn} = \Bigl(  \mathcal{S}_\beta^{qn} - (\mathcal{S}_\beta^n\Pi_n)^q\Bigr) + (\mathcal{S}_\beta^n\Pi_n)^q.
\ees
The difference enclosed in parentheses represents a truncation error; it is made small by choosing $n$ large. On the subspace of mean-zero functions, the truncated operator $\mathcal{S}^n\Pi_n$ satisfies a contraction property uniformly in $n$ (Lemma~\ref{lem:contraction}),  and  $(\mathcal{S}_\beta^n\Pi_n)^q$ goes to zero exponentially fast as $q\to \infty$. 

\begin{lemma} 
	  We have $\var_q(\log \varphi_\beta) \leq \beta C_q$ for all $q\in \N_0$ and $\beta,p>0$.
\end{lemma} 

\begin{proof}
	 Let $q\in \N_0$,  $(z_j)_{j\in \Z}, (z'_j)_{j\in \Z} \in (r_{\rm hc}, \infty)^{\Z}$ such that $z_j = z'_j$ for all $j \leq q$. Then 
      \bes 
	|\mathcal{W}_0(z) - \mathcal{W}_0(z')|  =| \sum_{j=0}^{\infty} \bigl(h_{-j}(z) - h_{-j}(z')\bigr)|  \leq \sum_{j=0}^\infty \var_{q+1+j}(h) = C_q 
      \ees
      and $\nu_{\beta}^-(\exp(- \beta \mathcal{W}_0)) \leq \exp(\beta C_q)  \nu_{\beta}^-(\exp(-\beta \mathcal{W}_0'))$. The claim then follows from the definition~\eqref{eq:phidef} of the invariant function.  
\end{proof} 

\begin{lemma} \label{lem:smoothen}
	Let $f:\R_+^\N\to \R$ be a bounded function. Then $n,k\in \mathbb N_0$, 
	\bes
		\var_{k} (\mathcal{S}_\beta^n f) \leq  \var_{n+k} (f) + ||f||_\infty (\e^{3 \beta C_k}-1).
	\ees
\end{lemma} 

\begin{proof} 
	Let 
$g =\sum_{j=1}^n h_j - \beta^{-1}\log [\lambda_0^n(\beta) \varphi_\beta] + \beta^{-1} \log \varphi_\beta\circ \tau^n$ on $(r_{\rm hc}, \infty)^{\N}$ and $g \equiv \infty$ on $\R_+^{\N} \setminus (r_{\rm hc}, \infty)^{\N}$ so that 
\bes 
    \mathcal{S}_\beta^n f (z_{n+1},z_{n+2},\ldots) = \int_{\R_+^n}  \e^{ -\beta g (z_1,z_2,\ldots)} f(z_1,z_2,\ldots)  \dd z_1\ldots \dd z_n. 
\ees	
Pick $z, z'  \in (r_{\rm hc}, \infty)^{\N}$ so that $z_j = z'_j$ for $j=1,\ldots,n+k$. Then 
\begin{align*}  
	\bigl| \e^{-\beta g(z)} f(z) -\e^{-\beta g(z')} f(z')  \bigr|
	& \leq \e^{-\beta g(z)} \bigl| f(z) - f(z')\bigr| +  \bigl| f(z') \bigr| \bigl| \e^{-\beta g(z)}- \e^{-\beta g(z')}\bigr|   \\
	& \leq \e^{-\beta g(z)}\Bigl( \var_{n+k}(f) + ||f||_\infty \bigl( \e^{\beta \var_{n+k}(g)} -1 \bigr) \Bigr). 
\end{align*} 
We integrate out $z_1,\ldots,z_n$, observe 
$\int \exp(-\beta g) \dd z_1\cdots \dd z_n = \mathcal{S}_\beta^n \mathbf{1} = \mathbf{1}$, and deduce 
\bes
	\var_{k}(\mathcal{S}^n f) \leq \var_{n+k}(f) +  ||f||_\infty \bigl(\e^{\beta \var_{n+k}(g)} -1 \bigr).
\ees
To conclude, we note 
\begin{align}
	\var_{k+n}(g)& \leq \sum_{j=0}^{n-1} \var_{n+k-j}(h) +\frac{1}{\beta} 
		\bigl( \var_{n+k}(\log \varphi)+ \var_{k}(\log \varphi) \bigr) \notag \\
		& \leq C_k + C_{n+k} + C_k\leq 3 C_k.\label{eq:vargbound}
\end{align}

\end{proof} 

\begin{lemma} \label{lem:contraction} 
	Let $f\in L^1(\R_+^\N,\varphi_\beta \nu_\beta)$ such that $\nu_\beta (f\varphi_\beta) =0$. Then for all $n\geq 1$ and $\gamma(\beta) = \exp( - 3 \beta C_0)$  
	\bes
		||\mathcal{S}_\beta^n \Pi_n f||_1 \leq (1- \gamma(\beta) \bigr) ||f||_1.  
	\ees
\end{lemma} 
 
\begin{proof} 
	We adapt \cite[Proposition 3]{ruelle68}. Consider first a non-negative function $f$ that depends on $z_1,\ldots,z_n$ only, i.e., $\var_n(f) =0$. 
	Let $k\geq 0$ $z,z'$ such that $z_j=z'_j$ for $j=1,\ldots,n$ and $g(z_1,z_2,\ldots)$ as in the proof of Lemma~\ref{lem:smoothen}. Then
\begin{align*}
	(\mathcal{S}_\beta^n f)(z_{n+1},z_{n+2},\ldots ) 
	& = \int \e^{-\beta g(z_1,\ldots)} f(z_1,\ldots,z_n) \dd z_1\cdots \dd z_n \\
	&\leq \e^{\beta \var_n(g)}\int \e^{-\beta g(z'_1,\ldots)} f(z'_1,\ldots,z'_n) \dd z'_1\cdots \dd z'_n \\
	& = \e^{\beta \var_{n} (g)} (\mathcal{S}_\beta^n f)(z'_{n+1},z'_{n+2},\ldots). 
\end{align*}
	By Inequality~\eqref{eq:vargbound} with $k = 0$  we have $\var_{n}(g) \leq 3 C_0$, uniformly in $n$. Thus $\mathcal{S}_\beta^n f(z) \leq \exp(- 3\beta C_0) (\mathcal{S}_\beta^n f)(z')$ for all $z,z' \in (r_{\rm hc}, \infty)^{\N}$. For non-negative $f$ with $f =\Pi_n f$ we have  by Lemma \ref{lem:mu-S-shift}
	\bes
	\inf \mathcal{S}_\beta^n f 
	\geq \gamma(\beta)  \sup \mathcal{S}_\beta^n f 
	 \geq \gamma(\beta) \mu_\beta(\mathcal{S}_\beta^n f)
	= \gamma(\beta) \mu_\beta(|f|).
	\ees
	Next let $f$ with $\var_n(f) =0$ and $\mu_\beta(f) =0$. Then $\mu_\beta(f_+)=\mu_\beta(f_-)$ and
	\begin{align*}
		|\mathcal{S}_\beta^n f| & \leq \bigl( \mathcal{S}_\beta^n f_+ - \gamma(\beta) \mu_\beta(f_+)\bigr) + \bigl( \mathcal{S}_\beta^n f_- - \gamma(\beta) \mu_\beta(f_-)\bigr) \\
			& = \mathcal{S}_\beta^n (f_++f_-) - \gamma(\beta)\mu_\beta(f_++ f_-) 
			= \mathcal{S}_\beta^n|f| - \gamma(\beta) \mu_\beta(|f|). 
	\end{align*}
	We integrate against $\mu_\beta$, use $\mu_\beta(\mathcal{S}_\beta^n|f|) = \mu_\beta(|f|) = ||f||_1$, and find $||\mathcal{S}_\beta^nf||_1 \leq (1-\gamma(\beta))||f||_1$. This holds for every local function $\var_n(f) =0$ with $\mu_\beta(f)=0$. For general $f$, we may apply the bound to $\Pi_n f$ and use $\mu_\beta(\Pi_n f) = \mu_\beta(f) =0$ and $\mu_\beta(|\Pi_n f|) \leq \mu_\beta(|f|)$, and we are done. 
\end{proof} 

\begin{lemma} \label{lem:truncerr}
	Let $f \in L^1(\R_+^\N,\varphi_\beta \nu_\beta)$ be a bounded map with $\nu_\beta(f\varphi_\beta) =0$. Then for all $q,n\in \N$,
	\bes
		 ||\mathcal{S}_\beta^{nq} f - (\mathcal{S}_\beta^n\Pi_n)^q f||_1 
		\leq \frac{1}{\gamma(\beta)} (\e^{3\beta C_n} - 1)||f||_\infty 
			+  \frac{1}{\gamma(\beta)} \var_n(f).
	\ees
\end{lemma} 

\begin{proof} 
	A telescope summation, the triangle inequality, and Lemma~\ref{lem:contraction} yield
	\begin{align*}
		||\mathcal{S}_\beta^{nq} f - (\mathcal{S}_\beta^n\Pi_n)^q f||_1 
		& \leq \sum_{k=0}^{q-1} || (\mathcal{S}_\beta^n\Pi_n)^k \bigl(\mathcal{S}_\beta^n\Pi_n - \mathcal{S}_\beta^n\bigr) (\mathcal{S}_\beta^n)^{q-k-1} f||_1 \\	
		& \leq \sum_{k=0}^{q-1} (1-\gamma(\beta))^k 
			|| \bigl(\mathcal{S}_\beta^n\Pi_n - \mathcal{S}_\beta^n\bigr) (\mathcal{S}_\beta^n)^{q-k-1} f||_1 \\
		& \leq \sum_{k=0}^{q-1} (1-\gamma(\beta))^k 
					|| \bigl(\Pi_n - \mathrm{id}\bigr) (\mathcal{S}_\beta^n)^{q-k-1} f||_1,
	\end{align*}
where in the second step we use that $\nu_{\beta}((\mathcal{S}_\beta^n\Pi_n)^i \bigl(\mathcal{S}_\beta^n\Pi_n - \mathcal{S}_\beta^n\bigr) (\mathcal{S}_\beta^n)^{q-k-1} f \varphi_\beta) = \nu_{\beta}(f \varphi_\beta) = 0$ for $i = 1, \ldots, k$ by Lemma \ref{lem:mu-S-shift} and the third step follows from $|\mathcal{S}_\beta^n\bigl(\Pi_n - \mathrm{id}\bigr) (\mathcal{S}_\beta^n)^{q-k-1} f| \le \mathcal{S}_\beta^n|\bigl(\Pi_n - \mathrm{id}\bigr) (\mathcal{S}_\beta^n)^{q-k-1} f|$ and Lemma \ref{lem:mu-S-shift}.
By Eq.~\eqref{eq:pinvarn} and Lemma~\ref{lem:smoothen}, this can be further estimated as 
	\begin{align*}
	  & \sum_{k=0}^{q-1}(1-\gamma(\beta))^k \var_n( \mathcal{S}_\beta^{n(q-k-1)} f) \\
	   & \quad \leq \sum_{k=0}^{q-1}(1-\gamma(\beta))^k \Bigl( (\e^{3\beta C_n} - 1)||f||_\infty + \var_{n(q-k)}(f) \Bigr) \\
		& \quad \leq \frac{1}{\gamma(\beta)} (\e^{3\beta C_n} - 1)||f||_\infty 
			+  \frac{1}{\gamma(\beta)} \var_n(f). \qedhere
	\end{align*}
\end{proof} 

\begin{proof} [Proof of Theorem~\ref{thm:ergodicity}] 
	Let $f,g:\R_+^\N\to \R$ be bounded functions and $q,n\in \N$,  $N \ge qn$. Using Eq.~\eqref{eq:smn} and Lemmas~\ref{lem:contraction} and~\ref{lem:truncerr}, we get 
	\begin{align*}
	& \bigl| \mu_\beta\bigl( f (g\circ\tau^{N}) \bigr) -\mu_\beta(f)\mu_\beta(g)\bigr| 
	 = \bigl|\mu_\beta \bigl((\mathcal{S}_\beta^{N} f)g \bigr) -\mu_\beta(f) \mu_\beta(g)\bigr|  \\
	&\qquad  \leq \mu_\beta\bigl(  |g|  \bigl|\mathcal{S}_\beta^{N}(f-\mu_\beta( f) \mathbf{1} )\bigr|\bigr) 
	 \leq ||g||_\infty \, ||\mathcal{S}_\beta^{{N}} (f-\mu_\beta( f) \mathbf{1} )||_1 \\ 
	&\qquad\leq ||g||_\infty \, ||\mathcal{S}_\beta^{qn} (f-\mu_\beta( f) \mathbf{1} )||_1 \\
	&\qquad \leq \Bigl( (1-\gamma(\beta))^q + \frac{1}{\gamma(\beta)}(\e^{3\beta C_n} - 1)\Bigr) ||g||_\infty ||f-\mu_\beta( f)||_\infty + \frac{1}{\gamma(\beta)} ||g||_\infty \var_n(f) 
	\end{align*}
since $||\mathcal{S}_\beta||_1\leq 1$. The explicit estimate on the decay of correlations follows.  That $\mu_\beta$ is mixing then follows from standard approximation arguments. 
\end{proof} 
 
\begin{proof} [Proof of Theorem~\ref{thm:corrgen}]
	The estimate for infinite $m$ is an immediate consequence of Theorem~\ref{thm:corrgen}. For finite $m$ and $n=m-1$, the truncation error in Lemma~\ref{lem:truncerr}  for a function $f : \R_+^n \to \R$ actually vanishes since $\var_n(f) = 0$ and $C_n = 0$. The bound simplifies accordingly. 
\end{proof}

\subsection{Thermodynamic limit} 

\begin{prop} \label{prop:thermolim}
	Let $m\in\N \cup \{\infty\}$ and $p>0$. 
	\begin{enumerate} 
		\item[(a)] The Gibbs free energy and its surface correction defined by the limits~\eqref{eq:gdef} exist and are given by 
		$$ g(\beta) = - \frac{1}{\beta}\log \lambda_0(\beta),\quad  g_{\rm surf}(\beta)= - g(\beta)-\frac{ 1}{\beta}\log \mu_\beta(\e^{\beta \mathcal{W}_0}).$$
		\item [(b)] Eqs.~\eqref{eq:nutherm} and~\eqref{eq:mutherm} hold true.
	\end{enumerate} 
\end{prop} 

\begin{proof} 
	 We compute
	\be \label{eq:embedded}
    \begin{aligned}
		&\nu_\beta\bigl(\e^{\beta \mathcal{W}(z_1\cdots z_n\mid z_{n+1}\cdots)} \bigr)  = \bigl(\frac{1}{\lambda_0(\beta)^n} {\mathcal{L}_\beta^*}^n\nu_\beta\bigr) \bigl(\e^{\beta \mathcal{W}(z_1\cdots z_n\mid z_{n+1}\cdots) } \bigr)  \\
		& \qquad= \frac{1}{\lambda_0(\beta)^n} \int \e^{\beta \mathcal{W}(z_1\cdots z_n\mid z_{n+1}\cdots)} \e^{-\beta \sum_{j=1}^n h_j } \dd z_1\cdots \dd z_n \dd\nu_\beta(z_{n+1} z_{n+2}\ldots) \\
		& \qquad = 	\frac{1}{\lambda_0(\beta)^n}\int \e^{-\beta \mathcal{E}_{n+1}(z_1,\ldots,z_n)} \dd z_1\cdots \dd z_n \dd\nu_\beta(z_{n+1} z_{n+2}\ldots) \\
		& \qquad = \frac{1}{\lambda_0(\beta)^n} Q_{n+1}(\beta).
	\end{aligned}
	\ee
	Let $\mathcal{W}_{0n} = \sum_{j\leq 0}\sum_{k\geq n+1} v(z_j+\cdots+z_k)$. We note 
	\bes
		\mathcal{W}(z_1\cdots z_n\mid z_{n+1}\cdots) = \mathcal{W}_n- \mathcal{W}_{0n}.  
	\ees
	and with \eqref{eq:numu} deduce 
	\bes
  \frac{1}{\lambda_0(\beta)^n}\, Q_{n+1}(\beta) = \nu_\beta(\e^{\beta \mathcal{W}(z_1\cdots z_n\mid z_{n+1}\cdots)}) = \frac{\mu_\beta(\exp(\beta [\mathcal{W}_0+\mathcal{W}_n- \mathcal{W}_{0n}]) )}{\mu_\beta(\exp( \beta \mathcal{W}_0))}. 
\ees
Now $\mathcal{W}_{0n} = O(n^{-(s-2)})\to 0$ uniformly  on $(r_{\rm hc}, \infty)^\Z$. By Theorem~\ref{thm:ergodicity}, $\mu_\beta (\exp(\beta [\mathcal{W}_0+ \mathcal{W}_n])) = \mu_\beta (f (f\circ \tau^n)) \to \mu_\beta(f)^2$ where $f= \exp(\beta \mathcal{W}_0)$. Consequently as $n\to \infty$ 
	\bes
		\log Q_{n+1}(\beta) = (n+1) \log \lambda_0(\beta) - \log \lambda_0(\beta) +   \log \mu_\beta(\e^{\beta \mathcal{W}_0}) + o(1),
	\ees
	from which part (a) of the lemma follows. A computation analogous to Eq.~\eqref{eq:embedded} shows that for every local test function $f\in C_b(\R_+^k)$, 
	\bes
		\Q_{n+1}^\ssup{\beta}(f) = \frac{\mu_\beta( f \exp( \beta [\mathcal{W}_0+ \mathcal{W}_n - \mathcal{W}_{0n}] )}{\mu_\beta(\exp( \beta [\mathcal{W}_0+ \mathcal{W}_n - \mathcal{W}_{0n}] )}.
	\ees
	Part (b) of the lemma then follows from Theorem~\ref{thm:ergodicity}. 
\end{proof}

\section{Large deviations as $\beta \to \infty$} 

Here we analyze the behavior of the bulk and surface Gibbs measures $\mu_\beta$ and $\nu_\beta$ and of the energies $g(\beta)$ and $g_\mathrm{surf}(\beta)$. The large deviations result for the surface measure $\nu_\beta$ is a consequence of the eigenvalue equation from Lemma~\ref{lem:eigenmeas}, exponential tightness, and the uniqueness of the solution to the fixed point equation in Proposition~\ref{prop:bellman-ersatz}. Since the bulk measure is absolutely continuous with respect to the product measure of two independent half-infinite chains (Eq.~\eqref{eq:mudef} and Proposition~\ref{prop:thermolim}(b)), we may go from the surface to the bulk measure with the help of Varadhan's integral lemma~\cite[Chapter 4.3]{dembo-zeitouni}. The asymptotic behavior of $e_\mathrm{surf}(\beta)$ is based on the representation from Proposition~\ref{prop:thermolim}(a).  

\subsection{A tightness estimate} 

The following estimate will help us prove that the infinite-volume measure $\nu_\beta$ is exponentially tight (see the proof of Lemma~\ref{lem:subseqldp}) which enters the proof of Theorem~\ref{thm:ldp}.

\begin{lemma} \label{lem:ti}
	For all $\beta,p>0$, $N\in \N$, $k\in \{1,\ldots,N-1\}$, and $r \ge 0$, we have  
	$$ \Q_N^\ssup{\beta}( \{z\in \R_+^{N-1} \mid z_k \geq z_{\max} +  r\}) \leq \exp( - \beta p r ). $$ 
\end{lemma} 

\begin{proof} 
	Fix $k\in \N$ and $r \ge 0$.  For $z= (z_1,\ldots,z_{N-1})\in \R_+^{N-1}$ with $z_k \geq z_{\max}+r$ we define a new configuration $z'$ by setting $z'_k = z_k - r$ and leaving all other spacings unchanged. This decreases the Gibbs energy by an amount at least 
\bes
	\mathcal{E}_N(z) - \mathcal{E}_N(z') \geq p z'_k - p z_k = p r. 
\ees
A change of variables thus yields
\begin{align*}
	\Q_N^\ssup{\beta} ( \{z \mid z_k \geq z_{\max} + r \} ) 
	& = \frac{1}{Q_N(\beta)} \int_{\R_+^{N-1}} \e^{-\beta \mathcal{E}_N(z)} \mathbf{1}_{ [z_{\max}+r, \infty) } (z_k)  \dd z \\
	&  \leq \frac{1}{Q_N(\beta)} \int_{\R_+^{N-1}}\e^{-\beta p r}  \e^{-\beta \mathcal{E}_N(z')} \mathbf{1}_{[z_{\max}, \infty)} (z'_k)  \dd z' \\
			& \leq \e^{-\beta p r},
	\end{align*}
	and the proof of the lemma is easily concluded.
\end{proof}

\subsection{Gibbs free energy in the bulk} 

\begin{lemma}\label{lemma:lambda0-limit}
Let $\beta \to \infty$ at fixed $p$. Then 
  \bes
       g(\beta) = -\frac{1}{\beta}\log \lambda_0(\beta) =  e_0+ O(\beta^{-1} \log \beta).
  \ees
\end{lemma}

\begin{proof}[Proof of Lemma~\ref{lemma:lambda0-limit}]
 The relation between $g(\beta)$ and $\lambda_0(\beta)$ has been proven in Proposition~\ref{prop:thermolim}. 
We proceed with an upper bound for $Q_N(\beta)$ and $\lambda_0(\beta)$. 
	For $z=(z_1,\ldots,z_{N-1})$, define $z'$ by $z'_j= \min (z_{\max},z_j)$.  Revisiting the proof of Lemma~\ref{lem:est-on-min}, we see that 
	\bes
		\mathcal{E}_N(z)\geq \mathcal{E}_N(z') + \sum_{j=1}^{N-1} \min( p(z_j - z_{\max}),0) \geq E_N +   \sum_{j=1}^{N-1} p \min((z_j - z_{\max}),0).
	\ees
	It follows that 
	\bes
		Q_N(\beta) \leq \e^{-\beta E_N} \prod_{j=1}^{N-1} \bigl(z_{\max} + \int_{z_{\max}}^\infty \e^{- \beta p (z_j - z_{\max}) }\dd z_j \bigr)
	\ees
	and 
	\bes
		\log \lambda_0(\beta) \leq  - \beta e_0 + \log \Bigl(z_{\max} + \frac{1}{\beta p} \Bigr),
	\ees
	whence $\beta^{-1}\log \lambda_0(\beta) \leq - e_0+ O(\beta^{-1})$. 
	For a lower bound, we let $\bar{z} \in [z_{\min}, z_{\max}]^{N-1}$ be the minimizer of $\mathcal{E}_N$ and choose $0 < \eps < a - z_{\min}$ so small that by Lemma \ref{lem:hessian} 
\[ \mathcal{E}_N(z) 
   \leq E_N + C \sum_{j=1}^{N-1}(z_j - \bar{z}_j)^2. \]
for every $z \in \times_{j=1}^{N-1} [\bar{z}_j - \eps, \bar{z}_j + \eps]$. We get 
\begin{align*}
  Q_N(\beta) 
  &\geq \e^{-\beta E_N} \prod_{j=1}^{N-1} \int_{\bar{z}_j-\eps}^{\bar{z}_j+\eps} \e^{- C \beta (z_j - \bar{z}_j)^2} \dd z_j \bigr) 
  = \e^{-\beta E_N} \Bigr( \int_{-\eps}^\eps \e^{-C \beta s^2} \dd s \Bigl)^{N-1}. 
\end{align*}
This yields 
\begin{align*}
  \log \lambda_0(\beta) 
  &\geq -\beta  e_0  + \log \Bigl(\int_{-\eps}^\eps \e^{- C \beta s^2}\dd s\Bigr) \\
  &= - \beta e_0 - \log \sqrt{ \frac{C \beta}{\pi}}+ \log \Bigl( 1- \sqrt{\frac{2}{\pi}} \int_{\eps\sqrt{2 C \beta}}^\infty \e^{-x^2 /2} \dd x\Bigr).
	\end{align*}
and $\beta^{-1}\log\lambda_0 (\beta) \geq - e_0 + O(\beta^{-1}\log\beta)$. 
\end{proof}

\subsection{Large deviations principles for $\nu_\beta$ and $\mu_\beta$} \label{sec:ldp}

Here we prove Theorem~\ref{thm:ldp}. 

\begin{lemma} \label{lem:subseqldp}
	Every sequence $\beta_j\to \infty$ has a subsequence along which $(\nu_{\beta_j})_{j\in \N}$ satisfies a large deviations principle with speed $\beta_j$ and some good rate function.
\end{lemma}

\begin{remark}
	If $p= p_\beta \to 0$, we lose exponential tightness and only know that every sequence $(\nu_{\beta_j})$ has a subsequence along which it satisfies a weak large deviations principle~\cite[Lemma 4.1.23]{dembo-zeitouni}, which means that the upper bound in~\eqref{eq:ldp} is required to hold for compact sets rather than closed sets.
\end{remark} 

\begin{proof}
	The lemma is a consequence of exponential tightness. 
		Let $n\in \N_0$. Define $K_n= \times_{j=1}^\infty [0,z_{\max}+n+j]$. $K_n$ is compact in the product topology. Passing to the limit $N\to \infty$ in Lemma~\ref{lem:ti}, we find 
\bes
	\nu_\beta( \{z\in \R_+^\N\mid z_k\geq z_{\max}+ r\} ) \leq \e^{-\beta p r} 
\ees
	for all $k\in \N$ and $r \ge 0$. Therefore 
	\begin{align*}
		\nu_\beta(K_n^\mathrm{c}) 
		& \leq \sum_{k=1}^\infty \nu_\beta( \{ z\in \R_+^\N\mid z_k >  z_{\max}+ k+ n \} ) \\
		& \leq \sum_{k=1}^\infty \e^{-\beta p(k+n)} 
		= \frac{\exp(-\beta p(n+1))}{1- \exp( -\beta p)}. 
	\end{align*}
	It follows that the family of measures $(\nu_\beta)_{\beta\geq 1}$ is exponentially tight, i.e.,  for every $M>0$, we can find a compact subset $K\subset \R_+^\N$ such that 
	$\limsup_{\beta\to\infty} \frac{1}{\beta}\log \nu_\beta( K^\mathrm{c}) \leq - M$.
	$\R_+^\N$ endowed with the product topology is separable and metrizable and therefore has a countable base. Lemma~4.1.23 in~\cite{dembo-zeitouni} applies and yields the claim. 
\end{proof}

\begin{lemma}  \label{lem:fpe}
	Suppose that Assumption \ref{assu:hcvtoinfty} holds true and assume that along some subsequence $(\beta_j)$ the measure $\nu_{\beta_j}$ satisfies a large deviations principle with good rate function $I(z_1,z_2,\ldots)$. Then $I$ satisfies 
	\begin{equation*}
		I(z_1,z_2,\ldots) =  \bigl(h(z_1,z_2,\ldots)- e_0\bigr) + I(z_2,z_3,\ldots).
	\end{equation*} 
	on $\R_+^\N$. In particular, $I((z_j)_{j\in \N}) = \infty$ if $z_j \leq r_\mathrm{hc}$ for some $j\in \N$. 
\end{lemma}

\begin{proof} 
 Write $\beta$ instead of $\beta_j$. We will see that the fixed point equation for $I$ follows from the eigenvalue equation in Lemma~\ref{lem:eigenmeas} and the asymptotics of the principal eigenvalue provided in Lemma~\ref{lemma:lambda0-limit}. According to these, 
	\be\label{eq:iteratedeigen}
		\dd \nu_{\beta}  (z_1 z_2\ldots) = \e^{-\beta [h_1 + \ldots + h_n - ne_0 + o(1)]} \dd z_1 \ldots \\ z_n \dd \nu_{\beta}(z_{n+1}\ldots)
	\ee
	for any $n \in \N$ where the $o(1)$-term comes from $\log \lambda_0^{n}(\beta) = - \beta [ne_0+o(1)]$ and is independent of $(z_j)_{j\in \N}$. 

We first show that $I$ can only be finite on $(r_{\rm hc}, \infty)^\N$.  Fix $n \in \N$ and for $\eps > 0$ consider the open set $O_{\eps} = \{z \in \R^\N \mid 0 < z_n < r_{\rm hc}+\eps\}$. A repeated application of Lemma~\ref{lem:eigenmeas} and Lemma~\ref{lemma:lambda0-limit} give 
\begin{align*} 
  \nu_\beta(O_{\eps}) 
  = \int_{O_{\eps} \cap (r_{\rm hc}, \infty)^{\N}} \e^{-\beta [h_1 + \ldots + h_n - n e_0 + o(1)]} \dd z_1 \ldots \dd z_n \dd \nu_\beta(z_{n+1}\ldots).
\end{align*}
Let $-C$ be a lower bound for $-e_0 + v(z_{\max}) + \sum_{k=2}^\infty v(z_1 + \cdots + z_k )$ on $(r_\mathrm{hc},\infty)^\N$. Then 
\begin{align*}
  \nu_\beta(O_{\eps}) 
  &\le \int_{(r_{\rm hc}, \infty)^{n-1}} \e^{-\beta [p(z_1 + \ldots + z_{n-1}) - C(n-1) + o(1)]} \dd z_1 \ldots \dd z_{n-1} \\ 
  &\qquad \times \int_{(r_{\rm hc}, r_{\rm hc}+\eps)} \e^{-\beta [p z_n + v(z_n) - C]} \dd z_n 
\end{align*}
and 
\begin{align*}
  \log \nu_\beta(O_{\eps}) 
  &\le \beta (C + o(1)) (n-1) + \log \eps - \beta \inf_{s \in (r_\mathrm{hc},r_\mathrm{hc} + \eps ]} (p s + v(s)). 
\end{align*}
Hence 
\[ - \inf_{O_\eps} I 
   \leq C(n-1) - \inf_{s \in (r_\mathrm{hc},r_\mathrm{hc} + \eps ]} (p s + v(s)) =:- f(\eps) \]
It follows that 
\[ \inf\{ I(z) \mid z_n \le r_{\rm hc} \} 
   \geq \lim_{\eps \to 0} f(\eps) = \infty. \] 
Since $n$ was arbitrary we have shown that $I \equiv \infty$ on $\R_+^{\N} \setminus (r_{\rm hc}, \infty)^{\N}$. In particular, as $\nu_\beta$ satisfies a large deviations principle on $\R_+^\N$ with rate function $I$,  the same large deviations principle holds on $(r_\mathrm{hc},\infty)^\N$. 

We now establish another (weak) large deviations principle on $(r_\mathrm{hc},\infty)^\N$. Let $K\subset (r_ \mathrm{hc},\infty)^\N$ be a  (relatively) closed set and  $[\alpha,b]\subset (r_ \mathrm{hc},\infty)$ a compact interval. Then \eqref{eq:iteratedeigen} with $n = 1$ yields 
	\bes
		\nu_\beta([\alpha,b] \times K) = \int_\alpha^b \Bigl( \int_K \e^{-\beta [h(z_1,z_2,\ldots) - e_0+o(1)]} \dd \nu_\beta(z_2,z_3,\ldots) \Bigr)\dd z_1.  
	\ees
	Write $f_\beta(z_1;K)$ for the inner integral. As $h$ is bounded from below and for every fixed $z_1>r_\mathrm{hc}$,  $(z_2,z_3,\ldots) \mapsto h(z_1,z_2,\ldots)$ is continuous n $(r_{\rm hc}, \infty)^{\N}$ with respect to the product topology, we deduce from Varadhan's lemma \cite[Chapter 4.3]{dembo-zeitouni} that 
	\be \label{eq:fbc}
		\limsup_{\beta\to \infty} \frac{1}{\beta}\log f_\beta(z_1;K) \leq - \inf_{(z_j)_{j\geq 2}\in K} \bigl(h(z_1,z_2,\ldots) - e_0+ I(z_2,z_3,\ldots)\bigr).
	\ee
	for all $z_1\in [\alpha,b]$. Next we note that for all $(z_j)_{j\in \N}\in (r_\mathrm{hc},\infty)^\N$, $z'_1>r_\mathrm{hc}$, and suitable $C>0$, 
	\bes
		|h(z_1,z_2,\ldots) - h(z'_1,z_2,\ldots)| \leq |v(z_1) - v(z'_1)| + C|z_1- z'_1|.
	\ees
	For $z_1,z'_1$ bounded away from $r_{\mathrm{hc}}$ we may exploit that the derivative of $v$ is bounded and drop the first term, making $C$ larger if need be. Plugging these estimates into the definition of $f_\beta(z_1,K)$, we find that for some $C_\alpha>0$ and all $\beta>0$, 
	\bes 
		\Bigl|\frac{1}{\beta}\log f_\beta(z_1;K) -\frac{1}{\beta}\log f_\beta(z_1';K)\Bigr|\leq C_{ \alpha} |z_1-z'_1|\quad (z_1,z'_1>\alpha>r_\mathrm{hc}).  
	\ees
	It follows that the upper bound~\eqref{eq:fbc} is uniform on compact subsets of $(r_\mathrm{hc},\infty)$ and 
	\be\label{eq:ubwldp}
		\limsup_{\beta\to \infty}\frac{1}{\beta}\log \nu_\beta([ \alpha,b]\times K) \leq - \inf_{z\in [\alpha,b]\times K} \bigl( h(z_1,z_2,\ldots) -e_0 + I(z_2,z_3,\ldots)\bigr).
	\ee
	A similar argument shows that for all $b>\alpha>r_\mathrm{hc}$ and all (relatively) open subsets $O\subset(r_{\mathrm{hc}}, \infty)^{\N}$, 
	\be\label{eq:lbwldp}
		\liminf_{\beta\to \infty}\frac{1}{\beta}\log \nu_\beta((\alpha,b)\times O) \geq - \inf_{z\in (\alpha,b) \times O} \bigl( h(z_1,z_2,\ldots) -e_0 + I(z_2,z_3,\ldots)\bigr).
	\ee
	Taking monotone limits, the latter inequality is seen to extend to $\alpha=r_\mathrm{hc}$ and $b= \infty$. It follows that $(\nu_\beta)$, as a family of probability measures on $(r_\mathrm{hc},\infty)^\N$, satisfies a weak large deviations principle with rate function $J= h_1 - e_0+I(z_2,\ldots)$.  (It is indeed sufficient to consider product sets. This is easy to see for the lower bound: If $U \subset (r_\mathrm{hc},\infty)^\N$ is open, then for any $\eps > 0$ one finds $\bar{z}\in (\alpha,b) \times O \subset U$ with $h(\bar{z}_1,\bar{z}_2,\ldots) -e_0 + I(\bar{z}_2,\bar{z}_3,\ldots)
- \eps \le \inf_{z\in U} \bigl( h(z_1,z_2,\ldots) -e_0 + I(z_2,z_3,\ldots)\bigr)$, from which it follows that \eqref{eq:lbwldp} holds for $U$ instead of $(\alpha,b) \times O$. The upper bound for a general compact $V \subset (r_\mathrm{hc},\infty)^\N$ is obtained by covering, for given $\eps > 0$, $V \subset \bigcup_{i = 1}^{N_{\eps}} (\alpha_{x_i},b_{x_i}) \times B_{\delta(x_i)}(x_i)$, where for each $x \in V$, $b_x > \alpha_x > r_\mathrm{hc}$ and $\delta(x) > 0$ are chosen such that  
$h(x_1,x_2,\ldots) -e_0 + I(x_2,x_3,\ldots) - \eps \le \inf_{z\in (\alpha_{x},b_{x}) \times B_{\delta(x)}(x)} \bigl( h(z_1,z_2,\ldots) -e_0 + I(z_2,z_3,\ldots)\bigr)$. This is possible since $I$ is lower semicontinuous. With the help of \eqref{eq:ubwldp} we can now deduce that \eqref{eq:ubwldp} holds for $V$ instead of $[\alpha,b] \times K$.)

Since  $(r_\mathrm{hc},\infty)^\N$ is  a Polish space, the rate function in a weak large deviations principle is uniquely defined~\cite[Chapter 4.1]{dembo-zeitouni}, hence $J=I$  on $(r_\mathrm{hc},\infty)^\N$. To finish the proof it remains to observe that also $J=I$ on $\R_+^\N \setminus (r_\mathrm{hc},\infty)^\N$ because both $I$ and $h$ are equal to $\infty$ on that set.  
\end{proof}

\begin{proof}[Proof of Theorem~\ref{thm:ldp}] 
	 The large deviations principle for $\nu_\beta$  with good rate function $\overline{\mathcal{E}}_{\mathrm{surf}} - \min \mathcal{E}_\mathrm{surf}$ is an immediate consequence of Lemmas~\ref{lem:subseqldp} and~\ref{lem:fpe} and Proposition~\ref{prop:bellman-ersatz}. 
	As a consequence, $\nu^-_{\beta} \otimes \nu^+_{\beta}$ satisfies a deviations principle with good rate function $(z_j)_{j \in \Z} \mapsto \overline{\mathcal{E}}_\mathrm{surf}(z_1,z_2,\ldots) + \overline{\mathcal{E}}_\mathrm{surf}(z_0,z_{-1},\ldots)  - 2 \min \mathcal{E}_\mathrm{surf}$ on $\R_+^{\Z}$ and on $[r_{\rm hc}, \infty)^{\Z}$, The large deviations principle for $\mu_\beta$ thus follows from Eq.~\eqref{eq:mudef},  Lemmas~4.3.4 and~4.3.6 in~\cite{dembo-zeitouni}, $\min \overline{\mathcal{E}}_\mathrm{bulk} =0$ and 
	\bes
		\overline{\mathcal{E}}_\mathrm{bulk}(z_1,z_2,\ldots) = \overline{\mathcal{E}}_\mathrm{surf}(z_1,z_2,\ldots) + \overline{\mathcal{E}}_\mathrm{surf}(z_0,z_{-1},\ldots) + \mathcal{W}_0(\cdots z_0\mid z_1\cdots) 
	\ees
by Proposition \ref{prop:lim-bulk}, and the observation that $\mathcal{W}_0$ is continuous on $[r_{\rm hc}, \infty)^{\Z}$.
\end{proof} 

\subsection{Surface corrections to the Gibbs free energy} 

\begin{proof} [Proof of Theorem~\ref{thm:geld}] 
The statements about $g(\beta)$ have already been proven in Lemma~\ref{lemma:lambda0-limit}. For $g_\mathrm{surf}(\beta)$, we start from the formula in Proposition~\ref{prop:thermolim}(a), to which we  apply Lemma~\ref{lemma:lambda0-limit}, Theorem~\ref{thm:ldp} and Varadhan's lemma. This yields 
\bes 
	\lim_{\beta\to \infty} g_{\rm surf}(\beta) 
		 = - e_0 + \inf \bigl( \mathcal{E}_{\rm bulk} - \mathcal{W}_0\bigr). 
\ees
But now for $(z_j)$ with $\sum_{j\in \Z}(z_j-a)^2<\infty$ 
	\begin{align*}
		\mathcal{E}_{\rm bulk}- \mathcal{W}_0 
		& = \sum_{j\in \Z} \sum_{k=1}^{m} \bigl( v(z_j + \cdots + z_{j+k-1}) - v(ka) + \delta_{1k} p(z_j - a)\bigr)  \notag \\
		&\quad -  \sum_{j \le 0, \ell \ge 1 \atop |\ell - j| \le m-1} \bigl( v(z_j+\cdots + z_\ell) - v((\ell -j +1) a)\bigr) - \sum_{k=1}^{m} (k-1)\, v(ka) \notag \\
		& =  \mathcal{E}_{\rm surf}(z_1,z_2,\ldots) +  \mathcal{E}_{\rm surf}(z_0,z_{-1},\ldots)
+ e_\mathrm{clamp}+e_0
	\end{align*}
with $e_\mathrm{clamp} = - pa - \sum\limits_{k=1}^\infty k \,v(ka)$. So
	\bes
		\inf (\mathcal{E}_{\rm bulk}-\mathcal{W})-e_0 = 2 \inf \mathcal{E}_{\rm surf}+e_\mathrm{clamp} = e_{\rm surf}. \qedhere
	\ees
\end{proof}


\section{Gaussian approximation} \label{sec:gaussian}

Here we prove Theorems~\ref{thm:gaussian2} and~\ref{thm:gaussian3} on the Gaussian approximation to the bulk measure $\mu_\beta$ when $m$ is finite. We start from a standard idea, namely perturbation theory for transfer operators~\cite{helffer-book}, however we  need  to put some work into a good choice of transfer operator as the standard symmetrized choice~\eqref{eq:Tbeta} does not work well. This aspect is explained in more detail in Section~\ref{sec:choices}. Throughout this section $m$ satisfies $2\leq m <\infty$. Remember $d=m-1$. 

\subsection{Decomposition of the energy. Choice of transfer operator} \label{sec:choices}

For finite $m$, the treatment with transfer operators from Section~\ref{sec:transferoperator-infinite} can be considerably simplified: instead of an operator that acts on functions of infinitely many variables, the transfer operator becomes an integral operator in $L^2(\R^d)$ ($L^2$ space with respect to Lebesgue measure).  There are several possible choices, corresponding each to an additive decomposition of the energy. 
Let
$V(z_1,\ldots,z_d):= \mathcal E_m(z_1,\ldots,z_d)$ and 
$$
	W(z_1,\ldots,z_d;z_{d+1},\ldots, z_{2d}) =  \sum_{\substack{1\leq i \leq d<j \leq 2d\\ |i-j|\leq d}} v(z_i+\cdots + z_j).
$$
Let us block variables as $x_j = (z_{dj+1},\ldots, z_{dj + d})$. 
Then for $(z_j)_{j \in \Z} \in \mathcal{D}_0^+$ we have  
\be  \label{eq:bulk1}
	\mathcal{E}_\mathrm{bulk}((z_j)_{j \in \Z}) 
   = \sum_{j \in \Z} \big( V(x_j) +  W(x_j, x_{j+1})- d e_0 \big)
\ee
with only finitely many  non-zero summands. By Proposition~\ref{prop:lim-bulk} the sum extends to $\mathcal D^+$ by continuity. The transfer operator associated with the representation~\eqref{eq:bulk1} is the integral operator with kernel $\exp( - \beta[V(x) + W(x;y)])$; it is clearly related to the $d$-th power of the transfer operator $\mathcal L_\beta$ from Section~\ref{sec:transferoperator-infinite}. 
The analysis is simpler for a symmetrized operator with kernel 
\be \label{eq:Tbeta}
	T_\beta(x,y) = \1_{(r_\mathrm{hc},\infty)^d}(x) \exp\Bigl( - \beta\Bigl[\tfrac12 V(x) + W(x;y) + \tfrac12 V(y)\Bigr] \Bigr)\1_{(r_\mathrm{hc},\infty)^d}(y).
\ee
which has the advantage of being Hilbert-Schmidt: 
The pressure term present in $V(x)$ and $V(y)$ ensures that $T_\beta(x,y)$ decays exponentially fast when $|x|+|y|\to \infty$ so that $\int_{\R^{2d}} T_\beta(x,y)^2 \dd x \dd y <\infty$. The  transfer operator $T_\beta$ corresponds to a rewriting of ~\eqref{eq:bulk1}, 
$$
	\mathcal{E}_\mathrm{bulk}((z_j)_{j \in \Z}) 
   = \sum_{j \in \Z} \big( \tfrac12 V(x_j) +  W(x_j, x_{j+1}) +  \tfrac12 V(x_{j+1}) - d e_0 \big).
$$
For the analysis of the limit $\beta\to \infty$, we would like to have a transfer operator that concentrates in some sense around the optimal spacings so that we may approximate it with a Gaussian operator. When $m\geq 3$, unfortunately, the function $(x,y)\mapsto \tfrac12 V(x) + W(x;y) + \tfrac12 V(y)$ need not have its minimum at $(x,y) = (\vect a, \vect a)$, with $\vect a = (a,\ldots,a)\in \R^d$. Therefore we introduce yet another variant of the transfer operator: we look for a function $\widehat H(x,y)$ such that 
$$
	\mathcal{E}_\mathrm{bulk}((z_j)_{j \in \Z}) 
   = \sum_{j\in \Z} \widehat H(x_j, x_{j+1}) 
$$
and $\widehat H(x,y) \geq \widehat H(\vect a, \vect a) =0$, and work with the kernel 
$$
	K_\beta(x,y) := \1_{(r_\mathrm{hc},\infty)^{d}}(x) \exp\Bigl( - \beta\widehat H(x,y) \Bigr)  \1_{(r_\mathrm{hc},\infty)^{d}}(y). 
$$
By a slight abuse of notation we use the same letter for the integral operator 
$$
	(K_\beta f)(x) = \int_{\R^d} K_\beta(x,y) f(y) \dd y. 
$$
in $L^2(\R^d)$. 
The function $\widehat H$ is defined as follows. 
Set
\begin{align*}
	H(x,y)& : = \inf \left \{ \mathcal{E}_{\mathrm{bulk}}\bigl( (z_j)_{j\in \Z}\bigr)\mid (z_j)_{j\in \Z} \in (r_\mathrm {hc},\infty)^\Z:\,  (z_1,\ldots,z_{2d}) = (x,y)\right\}, \\
	  w(x) & : = \inf \left \{ \mathcal{E}_{\mathrm{bulk}}\bigl( (z_j)_{j\in \Z}\bigr)\mid  (z_j)_{j\in \Z} \in (r_\mathrm {hc},\infty)^\Z:\, (z_1,\ldots,z_{d}) = x\right\}.
\end{align*} 
and 
\bes
	\widehat H(x,y) := H(x,y) - \tfrac12 w(x) - \tfrac12 w(y). 
\ees
Remember 
\bes
	 u(x) = \inf\{ \mathcal E_{\mathrm{surf}}\bigl((z_j)_{j\in \N}\bigr)\mid  (z_j)_{j\in \Z} \in (r_\mathrm {hc},\infty)^\N:\, (z_1,\ldots,z_d) = x\}.
\ees
\begin{lemma} \label{lem:hhat}
	Assume $2\leq m<\infty$, $p\in [0,p^*)$, and $r_\mathrm{hc}>0$.   Then:
	\begin{enumerate} 
	\item[(a)] For all $x,y\in (r_\mathrm{hc},\infty)^d$, we have $\widehat H(x,y) \geq \widehat H(\vect a, \vect a) =0$.
	\item[(b)]  The function $g(x):= \frac12[u(x) - u(\sigma x)]$ is bounded, and we have 
	$$
		\widehat H(x,y) = - g(x) + \Bigl( \tfrac12 V(x) + W(x,y) + \tfrac12 V(y)- d e_0\Bigr) + g(y).
	$$
	\item [(c)] $\widehat H( x,  y) = \widehat H(\sigma y,\sigma x)$ for all $x,y\in (r_\mathrm{hc},\infty)^d$.
	\end{enumerate} 
\end{lemma} 

\begin{proof}
	One easily checks 
	\bes 
		w(x) = \inf_{y \in (r_\mathrm{hc},\infty)^d} H(x,y),\quad w(y) = \inf_{x \in (r_\mathrm{hc},\infty)^d}  H(x,y)
	\ees
	which yields 
	\be \label{eq:positivehat}
		H(x,y)  -\tfrac12 w(x) - \tfrac12 w(y) = \tfrac12[ H(x,y) - w(x)] + \tfrac12 [H(x,y) - w(y)] \geq 0. 
	\ee
	For $x = y = \vect a$, we have $H(\vect a,\vect a)= w(\vect a)$ hence $\widehat H(\vect a,\vect a) =0$. This proves part (a) of the lemma. The symmetry in part (c) is immediate from the reversal symmetry of $\mathcal E_\mathrm{bulk}$.  
For (b), we note that  
$$
	H(x,y) = u(\sigma x) + W(x,y) + u(y),\quad w(x) = u(\sigma x)  + u(x) - V(x) + d e_0, 
$$	
the formula for $\widehat H$ follows. Because of 
$$
	u(x) = \inf_y \bigl( V(x) + W(x,y) - d e_0 + u(y)),
$$
and $V(\sigma x) = V(x)$, $C:= \sup_{(x,y)\in (r_\mathrm{hc},\infty)^{2d}} |W(x,y)-W(\sigma x, y)| <\infty$, we have 
$$
	u(x) \leq \inf_y \bigl( V(\sigma  x) + W(\sigma x, y) +C - d e_0 + u(y)\bigr) = u(\sigma x) + C.
$$
The roles of $x$ and $\sigma x$ can be exchanged, hence $u(x) - u(\sigma x)$ is bounded. 
\end{proof} 

\subsection{Some properties of the transfer operator} \label{sec:transferoperator-finite} 
 
\begin{lemma}\label{lem:kelementary}
	Assume $2\leq m<\infty$, $p\in (0,p^*)$, and and $r_\mathrm{hc}>0$.  Then:
	\begin{enumerate} 
	\item[(a)] 
	The kernels $K_\beta$ and $T_\beta$ are related as follows: 
	$$
		K_\beta(x,y) = \e^{\beta d e_0+ \tfrac12 \beta [u(x) - u(\sigma x) ]} T_\beta(x,y) \e^{ -\tfrac12  \beta[ u(y) - u(\sigma y)]}. 
	$$
	\item[(b)] The operator $K_\beta$ is a Hilbert-Schmidt operator in  $L^2(\R^d)$, and the kernel has the symmetry $K_\beta(x,y) = K_\beta(\sigma y, \sigma x)$. 
	\end{enumerate} 
\end{lemma} 
\noindent The lemma follows from Lemma~\ref{lem:hhat}, the elementary proofs are omitted. 

By the Krein-Rutman theorem~\cite{krein-rutman48}, \cite[Chapter 6]{deimling-book}, the operator norm $||K_\beta||=:\Lambda_0(\beta)$ is a simple eigenvalue of $K_\beta$, the associated eigenfunction $\phi_\beta$ can be chosen strictly positive on $(r_\mathrm{hc},\infty)^d$, and the other eigenvalues of $K_\beta$ have absolute value strictly smaller than $\Lambda_0(\beta)$, i.e., 
$$
	\Lambda_1(\beta)  = \sup\{|\lambda|\, : \, \lambda\, \text{eigenvalue of } K_\beta,\, \lambda \neq \Lambda_0(\beta)\}<\Lambda_0(\beta).
$$
By Lemma~\ref{lem:kelementary}(b), the function $\phi_\beta \circ \sigma$ is a left eigenfunction of $K_\beta$:
$$
	\int_{\R^d} \phi_\beta(\sigma x) K_\beta(x,y) \dd x 
	 =	\Lambda_0(\beta) \phi_\beta(\sigma y).
$$
Let $\Pi_\beta$  be the rank-one projection in $L^2(\R^d)$ given by 
$$
	\Pi_\beta f :=\frac{ \la f, \phi_\beta \circ \sigma \ra}{\la \phi_\beta, \phi_\beta \circ \sigma \ra} \phi_\beta.
$$
Then $K_\beta\Pi_\beta = \Lambda_0(\beta) \Pi_\beta = \Pi_\beta K_\beta$ and an induction over $n\in \N$ shows
\be \label{eq:knp}
	\frac{1}{\Lambda_0(\beta)^n} K_\beta ^n - \Pi_\beta = \Bigl( \frac{1}{\Lambda_0(\beta)} K_\beta - \Pi_\beta \Bigr)^n.
\ee
Since $\Lambda_1(\beta)$ is nothing else but the spectral radius of $ K_\beta -\Lambda_0(\beta) \Pi_\beta$, it follows that 
\be \label{eq:knp2}
	\limsup_{n\to \infty} || \Lambda_0(\beta)^{-n} K_\beta^n - \Pi_\beta||^{1/n} =\frac{\Lambda_1(\beta)}{\Lambda_0(\beta)}<1.
\ee
The spectral properties of $K_\beta$ are related to the Gibbs free energy and the Gibbs measure as follows.

\begin{lemma} \label{lem:standardtransfer} 
	Assume $2 \leq m< \infty$, $p\in (0,p^*)$, and $r_\mathrm{hc}>0$. Then:
	\begin{enumerate}
		\item [(a)] 	The Gibbs free energy is given by $g(\beta)=e_0-\frac{1}{\beta d}\log \Lambda_0(\beta)$. 
		\item [(b)]  The $nd$-dimensional marginals of the bulk Gibbs measure $\mu_\beta$ have probability density function 
		$$
			\frac1c \phi_\beta(\sigma x_1) \Biggl(\prod_{i=1}^{n-1} \frac{1}{\Lambda_0(\beta)} K_\beta(x_i,x_{i+1})\Biggr) \phi_\beta(x_n)
		$$
		with $c=\la\phi_\beta, \phi_\beta \circ \sigma\ra$.
		\item [(c)] For all $\eps>0$ and all bounded $f,g:\R^d\to \R$, writing $f_0\bigl((z_j)_{j\in \Z}\bigr):=f(z_{0},\ldots, z_{d-1})$ and $g_n\bigl((z_j)_{j\in \Z}\bigr):= g(z_{nj},\ldots, z_{nj+d-1})$, we have
		$$
			\bigl|\mu_\beta(f_0g_n) - \mu_\beta(f_0)\mu_\beta(g_n)\bigr|
				\leq C_\eps(\beta) \Bigl(\frac{\Lambda_1(\beta)}{\Lambda_0(\beta)}\Bigr)^{(1-\eps)n} ||f||_\infty ||g||_\infty
		$$
		with some constant $C_\eps(\beta)$ that does not depend on $f$, $g$, or $n$. If $m=2$, we can pick $\eps = 0$ and $C_0=1$. 
	\end{enumerate}	
\end{lemma}

\begin{proof}[Proof of Lemma~\ref{lem:standardtransfer}]
	For $N= nd +1$, the partition function $Q_N(\beta)$ is given by 
	\begin{align*}
		Q_{nd+1}(\beta)&= \la \e^{-\beta V/2}, T_\beta^{n-1} \e^{-\beta V/2}\ra = \e^{-(n-1)\beta d e_0}\la \e^{- \beta V/2 -\beta[u- u\circ \sigma]/2 },K_\beta^{n-1} \e^{- \beta V/2 + \beta[u- u\circ \sigma] /2}\ra.
	\end{align*}
	For the second identity we have used Lemma~\ref{lem:kelementary}(a). The function $u-u\circ \sigma$ is bounded by Lemma~\ref{lem:hhat}(b) and $\exp(- \beta V)$ is integrable because $V(z_1,\ldots,z_d)=\mathcal E_m(z_1,\ldots,z_d)$ grows linearly when $|z_j|\to \infty$. Therefore $F_\beta:= \exp(- \beta V/2-  \beta[u- u\circ \sigma] /2)$ and $F_\beta \circ \sigma$ are in $L^2(\R^d)$, and as $n \to \infty$,
	$$
		\la F_\beta , K_\beta^{n-1} F_\beta \circ \sigma\ra = \Lambda_0(\beta)^{n-1}\la F_\beta, \phi_\beta \ra^2 + O(\Lambda_1(\beta)^{n-1}). 
	$$
	It follows that 
	$$
		g(\beta) = - \lim_{n\to\infty}\frac{1}{\beta (nd+1)}\log Q_{nd+1}(\beta) = e_0- \frac{1}{\beta d}\log	 \Lambda_0(\beta),
	$$
	which proves part (a) of the lemma. The standard proof of part (b) is omitted (compare~\cite[Chapter 4]{helffer-book}). For (c), we use the formula for the $(n+1)d$- dimensional marginal provided by (b). Let us choose multiplicative constants in such a way that $c=\la \phi_\beta,\phi_\beta \circ \sigma\ra=1$. Then
	\begin{align*}
		\mu_\beta(f_0g_n) - \mu_\beta(f_0)\mu_\beta(g_n) &= \la  f(\phi_\beta \circ \sigma),\frac{1}{\Lambda_0(\beta)^n}K_\beta ^n (g \phi_\beta)\ra - \la f ( \phi_\beta\circ \sigma), \phi_\beta \ra\la \phi_\beta \circ \sigma, g \phi_\beta\ra \\
		&=\la f(\phi_\beta \circ \sigma),\Bigl( \frac{1}{\Lambda_0(\beta)^n}K_\beta ^n- \Pi_\beta\Bigr) (g \phi_\beta)\ra.
	\end{align*}
	Eq.~\eqref{eq:knp} yields 
	$$
		\bigl| \mu_\beta(f_0g_n) - \mu_\beta(f_0)\mu_\beta(g_n)\bigr|
			\leq || \Bigl( \frac{1}{\Lambda_0(\beta)} K_\beta - \Pi_\beta\Bigr)^n ||\, ||f(\phi_\beta \circ \sigma)||\, ||g \phi_\beta||
	$$
	where $||\cdot||$ refers to the $L^2$-norm for functions and the operator norm for the operator. We further bound $|| g\phi_\beta||\leq ||g||_\infty ||\phi_\beta||$ and $||f (\phi_\beta \circ \sigma)||\leq ||f||_\infty ||\phi_\beta||$ and conclude with~\eqref{eq:knp2}. If $m=2$, the operators are symmetric, hence the operator norm is the same as the spectral radius and the estimates simplify accordingly. 
\end{proof}

\begin{remark}[Associated Markov chain] 
Define the kernel 
\be \label{eq:pkernel}
	P_\beta(x,\dd y):= \frac{1}{\Lambda_0(\beta) \phi_\beta(x)} K_\beta (x,y) \phi_\beta (y) \dd y
\ee
on $(r_{\mathrm{ hc}},\infty)^d$. Then $P_\beta$ is a Markov kernel with invariant measure $\rho_\beta(x)\dd x$ where 
$$
	\rho_\beta (x) =\frac1c \phi_\beta(\sigma x) \phi_\beta(x).
$$
If in the bulk Gibbs measure $\mu_\beta$ we group spacing in blocks as $x_n = (z_{dn},\ldots,z_{dn+d-1})$, we obtain a probability measure on $(r_\mathrm{hc},\infty)^d$. This measure is exactly the distribution of 
the two-sided stationary Markov chain $(X_j)_{j\in \Z}$ with state space $\R^d$, transition kernel $P_\beta$, and initial law $\mathcal L(X_0) = \rho_\beta(x)\dd x$.
\end{remark}

\subsection{Gaussian transfer operator}
Here we introduce the Gaussian counterpart to the transfer operator $K_\beta$ and study its spectral properties. We start from the quadratic approximation to the bulk energy $\mathcal E_\mathrm{bulk}$. The differentiability of $\mathcal E_\mathrm{bulk}$ in a neighborhood of the constant sequence $z_j\equiv a$ is checked in Lemma~\ref{lem:ebulk-C2} below, for the definition of the Gaussian transfer operator we only need the infinite matrix of partial derivatives at $(\ldots, a,a,\ldots)$. 

In the following we block variables as $x_j = (z_{dj},\ldots, z_{dj + d-1})$ for $z = (z_j)_{j \in \Z}$ and $\xi_j = (\zeta_{dj},\ldots, \zeta_{dj + d-1})$ for $\zeta = (\zeta_j)_{j \in \Z}$. Remember the decomposition~\eqref{eq:bulk1}. Set $\vect a=(a,\ldots,a)\in \R^d$ and define the $d\times d$ matrices 
\be \label{eq:abdef}
	A := W_{yy}(\vect a,\vect a) + V_{xx}(\vect a) + W_{xx}(\vect a,\vect a),\quad B:= - W_{xy}(\vect a,\vect a).
\ee
	We note the following relations:
	\be \label{eq:sygauss}
		W_{yy}(\vect a) =\sigma W_{xx}(\vect a)\sigma,\quad B^T= \sigma B \sigma, \quad \sigma A \sigma = A.
	\ee
The Hessian $\mathrm D^2 \mathcal E_{\mathrm{bulk}}$ at $(\ldots,a,a,\ldots)$ is a doubly infinite, band-diagonal matrix with block form 
\be \label{eq:bulkbanddiag} 
	\begin{pmatrix} 
			\ddots &\ddots &\ddots &&&  \\
			 & - B^T& A   & - B& &\\
			    &      &  - B^T & A & - B & \\
			   &&& \ddots &\ddots & \ddots 
	\end{pmatrix}. 
\ee
Note that Lemma \ref{lem:hessian} implies that $\mathrm D^2 \mathcal E_{\mathrm{bulk}}(\ldots,a,a,\ldots)$ is positive definite. We look for a quadratic form $\mathcal Q(x,y)$ on $\R^{2d}$ that is positive-definite and satisfies 
\bes
	\mathcal E_\mathrm{bulk}\bigl( (z_j)_{j\in \Z}\bigr) = \tfrac12 \sum_{j\in \Z} \mathcal Q(x_j-\vect a, x_{j+1}-\vect a) +  o\Bigl( \sum_{j \in \Z} |x_j-\vect a|^2\Bigr). 
\ees
One candidate choice could be 
\bes 
	\mathcal Q(x,y) := \tfrac12 \la x, Ax\ra - 2 \la x, By'\ra + \tfrac12 \la y, A y\ra \quad (x',y'\in \R^d),
\ees
but it is not easily related to $\widehat H(x,y)$. We make a different choice which mimicks the definition of $\widehat H(x,y)$ and show later that this amounts to picking the Hessian of $\widehat H(x,y)$ (see Lemma~\ref{lem:Hmin} below). 

We introduce the quadratic counterparts to the functions $H(x,y)$, $w(x)$, and $\widehat H(x,y)$ from Section~\ref{sec:transferoperator-finite}.  Remember the bulk Hessian from~\eqref{eq:bulkbanddiag}.  Since it is positive-definite, there exist uniquely defined positive-definite matrices
	$M\in \R^{2d\times 2d}$ and $N\in \R^{d\times d}$ such that 
	\begin{align}
			\la \begin{pmatrix} x\\ y\end{pmatrix}, M \begin{pmatrix} x\\ y\end{pmatrix} \ra & = \inf\{ \la z, \mathrm D^2\mathcal E_\mathrm{bulk}(a,a,\ldots) z\ra \mid z\in \ell^2(\Z),\, (z_1,\ldots, z_{2d}) =( x,y) \}  \label{eq:Mdef} \\
			\la x, N x\ra  &= \inf\{ \la z, 		\mathrm D^2\mathcal E_\mathrm{bulk}(a,a,\ldots) z\ra \mid z\in \ell^2(\Z),\, (z_1,\ldots, z_d) = x\} \label{eq:Ndef} 
	\end{align} 
	for all $x,y\in \R^d$.  The quadratic forms associated with $M$ and $N$ are the Gaussian counterparts to the functions $H(x,y)$ and $w(x)$, respectively. Finally set 
	\be \label{eq:hatMdef}
		\widehat M:= M- \begin{pmatrix} 
								\frac12 N& 0 \\ 0 & \frac12 N
							\end{pmatrix}.  
	\ee
	and 
	\bes 
		\widehat{ \mathcal Q}(x,y):=\bigl \la \begin{pmatrix} x\\y \end{pmatrix}, \widehat M  \begin{pmatrix} x\\y \end{pmatrix}\bigr \ra.  
	\ees
	We will see in the proof of Lemma~\ref{lem:Hmin} that $M$, $N$ and $\widehat{M}$ are the Hessians of $H$ at $(\vect a, \vect a)$, $w$ at $\vect a$ and $\widehat{H}$ at $(\vect a, \vect a)$, respectively. The relation between $\mathcal Q$ and $\widehat {\mathcal Q}(x,y)$ is clarified in Lemma~\ref{lem:qq} below. 	
	We are going to work with the kernel
\bes
	G_\beta(x,y) := \exp\Bigl( - \tfrac12 \beta \widehat{\mathcal Q}(x- \vect a,y- \vect a)\Bigr)\qquad (x,y\in \R^d)
\ees
and the associated integral operator $(G_\beta f)(x)=\int_{\R^d} G_\beta(x,y) f(y)\dd y$. In Section~\ref{sec:perturbation} we show that $G_\beta$ is a good approximation for $K_\beta$, here we study the operator $G_\beta$ on its own. Clearly it is enough to understand the integral operator $G$ with kernel 
$$
	G(x,y):= \exp(- \tfrac12 \widehat {\mathcal Q}(x,y)),
$$
since $G$ and $G_\beta$ are related by the change of variables $x\mapsto \sqrt \beta (x- \vect a)$, see Eq.~\eqref{eq:gutrafo} below. 

\begin{lemma}\label{lem:Q} 
	Assume $2 \leq m< \infty$, $p\in [0, p^*)$. 
	Then the quadratic form $\widehat{\mathcal Q}$ is positive-definite: $\widehat{\mathcal Q}(x,y)\geq \eps(|x|^2+ |y|^2)$ for some $\eps>0$ and all $(x,y)\in \R^{2d}$.
\end{lemma} 

\begin{proof}
	 First we show that $\widehat M$ is positive semi-definite, by an argument similar to Lemma~\ref{lem:kelementary}(a). Define 
	$$
		F(x,y):= \la \begin{pmatrix} x\\ y\end{pmatrix},  M \begin{pmatrix} x\\ y\end{pmatrix}\ra.  
	$$
	Clearly 
	$$
		\la x, N x\ra = \inf_{y\in \R^d} F(x,y) \quad \la y, N y\ra = \inf_{x\in \R^d} F(x,y), 
	$$ 
	hence 
	\be \label{eq:mhatdeco}
		\la \begin{pmatrix} x\\ y\end{pmatrix}, \widehat M \begin{pmatrix} x\\ y\end{pmatrix}\ra 
			= \frac12 \Bigl( F(x,y)  - \la x, N x\ra \Bigr) 
		+ \frac12 \Bigl( 	F(x,y)  - \la y, N y\ra \Bigr) \geq 0
	\ee
	for all $(x,y)\in \R^d\times \R^d$ and $\widehat M$ is positive semi-definite. Next let $(x_0,y_0)\in \R^d\times \R^d$ be a zero of the quadratic form associated with $\widehat M$. Then by~\eqref{eq:mhatdeco}, the function  $y\mapsto F(x_0,y)$ must be minimal at $y=y_0$, hence $\nabla_y F(x_0,y) =0$. Similarly, the function $y\mapsto F(x,y_0)$ must be minimal at $x=x_0$, hence $\nabla_x F(x_0,y_0) =0$. Thus $(x_0,y_0)$ is a critical point of $F$. But $F$ is strictly convex because $M$ is positive-definite, therefore the critical point $(x_0,y_0)$ is a global minimizer of $F$ which yields $(x_0,y_0)=0$. It follows that $\widehat M$ is positive-definite.
\end{proof} 

It follows from Lemma~\ref{lem:Q} that $\int_{\R^{2d}} G(x,y)^2 \dd x\dd y <\infty$, hence $G$ is Hilbert-Schmidt with strictly positive integral kernel and Krein-Rutman theorem is applicable. So we may ask for its principal eigenvalue and eigenvector and its spectral gap. It is natural to look for a Gaussian eigenfunction.

\begin{lemma} \label{lem:gaussianansatz}
	Let $F$ be a positive-definite, symmetric $d\times d$ matrix. Then the following two statements are equivalent: 
	\begin{enumerate}
		\item[(i)] $\phi(x) :=\exp(- \tfrac12 \la x, 
F x\ra)$ is an eigenfunction of $G$. 
		\item [(ii)] The function $x\mapsto \la x, F x\ra$ satisfies the quadratic Bellman equation
\be \label{eq:bellmanhat}
	\la x,F x\ra = \inf_{y\in\R^d}\bigl( \widehat{\mathcal Q}(x,y) + \la y , Fy\ra\bigr).
\ee
	\end{enumerate}
\end{lemma}

\begin{proof}
	The proof is by a straightforward completion of squares: write 
	$$
		\widehat M= \begin{pmatrix} \widehat M_1 & \widehat M_2\\ \widehat M_2^ T & \widehat M_3\end{pmatrix} 
	$$
	with $d\times d$ -matrices $\widehat M_j$. The diagonal blocks $\widehat M_1$ and $\widehat M_3$ are positive-definite because $\widehat M$ is positive-definite, therefore $\widehat M_3+ F$ is positive-definite as well. Then 
	\begin{align*} 
	\widehat{\mathcal Q}(x,y) + \la y, Fy\ra 
	& =\la x, \widehat M_1 x\ra + 2 \la x, \widehat M_2 y \ra +\la y,( \widehat M_3+F) y\ra\\
	&= \la x, \widehat M_1  x\ra + \la y + (\widehat M_3+F)^{-1} \widehat M_2^T x,( \widehat M_3+F) (y + (\widehat M_3+F)^{-1} \widehat M_2^T x)\ra \\
			&\qquad \qquad - \la x, \widehat M_2( \widehat M_3+F)^{-1} \widehat M_2^T x\ra.
	\end{align*} 
	It follows that 
	$$
		\inf_{y\in \R^d} \bigl( 	\widehat{\mathcal Q}(x,y) + \la y, Fy\ra \bigr) = \la x, (\widehat M_1 - \widehat M_2 ( \widehat M_3+ F)^{-1} \widehat M_2^T) x\ra
	$$
	and 
	\be \label{eq:ga-eval}
		(G\phi)(x) = \sqrt{\frac{(2\pi)^d}{\det (\widehat M_3+F)}}\, \exp\Bigl( - \frac12 \la x, (\widehat M_1 - \widehat M_2 (\widehat M_3+F)^{-1} \widehat M_2^T) x\ra\Bigr).
	\ee
	Therefore (i) and (ii) hold true if and only if $F$ solves
	$$
		F = \widehat M_1 - \widehat M_2 (\widehat M_3+F)^{-1} \widehat M_2^T.
	$$
	In particular, (i) and (ii) are equivalent.
\end{proof} 

In Lemma~\ref{lem:qq} below we check that $M$ is of the form 
\be \label{eq:Mblocks}
	M = \begin{pmatrix} \sigma C\sigma & - B \\ - B^T & C \end{pmatrix} 
\ee
for some positive-definite $d\times d$ matrix $C$.

\begin{lemma} \label{lem:gauprincipal}
	The principal eigenvalue of $G$ is $\sqrt{(2\pi)^d /\det C}$ and the principal eigenfunction is $\exp( - \tfrac12 \la x, \tfrac12 N x\ra )$ (up to scalar multiples). 
\end{lemma} 

\begin{proof} 
 A close look at our definitions shows that $F:= \frac12 N$ solves~\eqref{eq:bellmanhat} (it is positive-definite because $N$ is). Indeed, by  the definition of $\widehat {\mathcal Q}$, $\widehat M$, we have 
\begin{align*}
	\inf_{y\in \R^d}\bigl(\widehat {\mathcal Q}(x,y) +\la y, \tfrac12N y\ra\bigr)
		& = - \la x, \tfrac12 N x\ra +	\inf_{y\in \R^d} \la \begin{pmatrix} x\\y \end{pmatrix}, M \begin{pmatrix} x\\y \end{pmatrix}\ra = \la x, \tfrac12 N x\ra.
\end{align*} 
Therefore, by Lemma~\ref{lem:gaussianansatz}, the function $\phi(x) =\exp( - \frac14 \la x, N x\ra)$ is an eigenfunction of $G$. 
 The matrix $\widehat M_3 +F$ in~\eqref{eq:ga-eval} is  equal to $(C- \tfrac12 N) + F=C$, and we find that the principal eigenvalue of $G$ is $\sqrt{(2\pi)^d /\det C}$.
\end{proof} 

In order to  identify the block $C$ in~\eqref{eq:Mblocks}, we  introduce the quadratic analogue to the function $u(x)$. 
	Let $A$ and $B$ be the $d\times d$ matrices from~\eqref{eq:abdef}
	and  $A_1:= V_{xx}(\vect a) + W_{xx}(\vect a, \vect a)$. 	
	The infinite matrix $(\partial_i \partial_j \mathcal E_\mathrm{surf}(a,a,\ldots))_{i,j\in \N}$ is  band-diagonal with block structure 
	$$
		\mathrm D^2\mathcal E_\mathrm{surf}(a,a,\ldots) 
			= \begin{pmatrix} 
					A_1 & - B & 0 & \cdots  && \\
					- B^T & A & - B & 0  & \cdots&   \\
					0 & - B^T & A & - B & 0 & \\
					\vdots & \ddots& \ddots & \ddots &\ddots & \ddots 
			\end{pmatrix}. 
	$$
	The matrix differs from the bulk Hessian~\eqref{eq:bulkbanddiag} by the upper left corner $A_1$: we have 
	\begin{equation} \label{eq:A1A}
		A= A_1 + W_{yy}(\vect a, \vect a). 
	\end{equation}
	By a  reasoning similar to Lemma~\ref{lem:hessian},  the Hessian of $\mathcal E_\mathrm{surf}$ is  positive-definite. Therefore there is a uniquely defined positive-definite $d\times d$-matrix $D$ such that 
	$$
		\la x, D x\ra  = \inf\{ \la z, 		\mathrm D^2\mathcal E_\mathrm{surf}(a,a,\ldots) z\ra \mid z\in \ell^2(\N),\, (z_1,\ldots, z_d) = x\}
	$$
	for all $x\in \R^d$. (Analogous arguments as in the proof of Lemma~\ref{lem:Hmin} show that $D$ is the Hessian of $u$ at $\vect a$.) Set
	\be \label{eq:Cdef}
		C:=D+ W_{yy}(\vect a, \vect a)
	\ee
	and 
\bes 
		J:=  D+ W_{yy}(\vect a, \vect a) - \sigma D \sigma- W_{xx}(\vect a, \vect a) = C - \sigma C \sigma 
	\ees
	(remember the symmetries~\eqref{eq:sygauss}).

\begin{lemma}\label{lem:qq}
	The matrix $C$ solves 
	\bes 
	C = A - B C^{-1} B^T
	\ees
	 and Eq.~\eqref{eq:Mblocks} holds true. Moreover 
	\begin{equation*}
		\widehat{\mathcal Q}(x,y) = 	- \la x, J x\ra + {\mathcal Q}(x,y) + \la y, J y\ra. 
	\end{equation*} 
\end{lemma}	
	
\begin{proof}
 Clearly
	$$
		\la x, D x\ra = \inf_{y\in \R^d} \bigl( \la x, A_1 x\ra - \la x, B y\ra - \la B^T x, y\ra + \la y, (W_{yy}(\vect a, \vect a) + D) y\ra \bigr)
	$$
	hence 
	\begin{equation} \label{eq:bequ}
		D = A_1 - B (W_{yy}(\vect a, \vect a)+ D)^{-1} B^T.
	\end{equation} 
	by a completion of squares similar to the proof of Lemma~\ref{lem:gaussianansatz}. We add $W_{yy}(\vect a, \vect a)$ to both sides, remember~\eqref{eq:A1A}, and obtain the equation for $C$. 	
	It is easy to see that 
	$$
		M = \begin{pmatrix} 
				\sigma D\sigma + W_{xx}(\vect a, \vect a) & -  B \\
				- B^T& W_{yy}(\vect a, \vect a) + D
			\end{pmatrix} = \begin{pmatrix} \sigma C\sigma  & - B \\ - B^T & C \end{pmatrix}
	$$
			which proves~\eqref{eq:Mblocks}. 
	Furthermore, 
	$$
		\la x, N x\ra = \inf_{y\in \R^d} \la \begin{pmatrix} x\\ y\end{pmatrix}, M \begin{pmatrix} x\\ y\end{pmatrix}\ra, \quad \la y, N y\ra = \inf_{x\in \R^d} \la \begin{pmatrix} x\\ y\end{pmatrix}, M \begin{pmatrix} x\\ y\end{pmatrix}\ra,
	$$
	hence, 
	$$
		N =\sigma C\sigma  - B C ^{-1} B^T, \quad N = C - B^T (\sigma C \sigma) ^{-1} B.
	$$
	Let us check that the two expressions for $N$ are indeed identical, and that $\sigma N\sigma = N$. 
	Combining with~\eqref{eq:A1A} and~\eqref{eq:bequ}, the two expressions for $N$ become 
	$$
		N= \sigma D\sigma + W_{xx}(\vect a, \vect a) - \bigl(A - W_{yy}(\vect a, \vect a) - D\bigr) = D + \sigma D \sigma + W_{xx}(\vect a, \vect a) +  W_{yy}(\vect a, \vect a)   - A 
	$$
	and 
	$$
		N = D + W_{yy}(\vect a, \vect a) - \sigma \bigl( A - W_{yy}(\vect a, \vect a) - D\bigr) \sigma = D + \sigma D \sigma + W_{xx}(\vect a, \vect a) + W_{yy}(\vect a, \vect a)   - A.
	$$
	The two expressions are indeed equal, and from the end formula and~\eqref{eq:sygauss} we read off that $\sigma N \sigma = N$. Actually 
	$$
		N = D + \sigma D \sigma - V_{xx}(\vect a),
	$$
	which is the analogue of $w(x) = u(x) + u(\sigma x) - V(x)$. 
	
		Now we compute $\widehat M$. The off-diagonal blocks of $\widehat M$ are the same as those of $M$. The upper left diagonal block is 
	\begin{align*}
		M_1 - \tfrac12 N & = \sigma D \sigma + W_{xx}(\vect a, \vect a) - \tfrac12 \bigl( D + \sigma D \sigma + W_{xx}(\vect a, \vect a) + W_{yy}(\vect a, \vect a)   - A\bigr) \\
		& = \tfrac12 A + \tfrac12 \bigl(\sigma D \sigma + W_{xx}(\vect a, \vect a)\bigr) - \tfrac12 \bigl( D + W_{yy}(\vect a, \vect a)\bigr). 
	\end{align*}
	A similar computation yields the lower right block.
	 Altogether we find 
	\begin{equation*} 
		\widehat M =
			\begin{pmatrix}
				\frac12 (A -J)  & - B \\ - B^T & \frac12 (A+J)
			\end{pmatrix}
	\end{equation*} 
	and the lemma follows. 
\end{proof}

\noindent	Finally we come back to the $\beta$-dependent operator $G_\beta$. 

\begin{prop}\label{prop:gbetagauss}	Assume $2\leq m <\infty$ and $p\in [0,p^*)$. 
		The principal eigenvalue of $G_\beta$ is 
	$$
		\Lambda_0^\mathrm{Gauss}(\beta) = \sqrt{\frac{(2\pi)^d}{\beta^d\, \det C}} 
	$$
	and the normalized, positive principal eigenfunction is 
	$$
		\phi_\beta^\mathrm{Gauss}(x) = \Bigl(\frac {\beta^d \det(\frac12 N)}{\pi^d}\Bigr)^{1/4} \exp\Bigl( - \tfrac12 \beta\la x - \vect a, \tfrac12 N\, (x-\vect a)\ra \Bigr). 
	$$
\end{prop} 

\begin{proof} 
	Let $U_\beta: L^2(\R^d)\to L^2(\R^d)$ be the unitary operator given by 
	\be \label{eq:ubeta}
		(U_\beta f)(x') = \beta^{-d/4} f(\vect a + \beta^{-1/2} x'). 
	\ee
	We have 
	\begin{align*}
	\bigl( U_\beta G_\beta f\bigr) (x') & = \beta^{-d/4} (G_\beta f)(\vect a+ \beta^{-1/2} \vect x')\\
		& = \beta^{-d/4} \int_{\R^d} G_\beta(\vect a+ \beta^{-1/2} x', \vect a+ \beta^{-1/2} y') f(\vect a+ \beta^{-1/2} y') \beta^{-d/2} \dd y' \\
	& = \beta^{- d/2} \int_{\R^d} G(x',y') (U_\beta f)(y') \dd y'
	\end{align*} 
	hence 
	\be \label{eq:gutrafo}
		G_\beta = \beta^{-d/2} U_\beta^*  G U_\beta
	\ee
	and the principal eigenvalue and eigenfunction of $G_\beta$ are obtained from those of $G$ in Lemma~\ref{lem:gauprincipal} by straightforward transformations. 
\end{proof} 

\begin{remark}
	When $m=2$, all eigenvalues and eigenfunctions of $G$ (hence $G_\beta$) can be computed explicitly, and the eigenfunctions are expressed with Hermite polynomials. See~\cite[Section 5.2]{helffer-book} on the harmonic Kac operator. 
\end{remark}

\subsection{Perturbation theory} \label{sec:perturbation}

 Remember the unitary operator $U_\beta$ from~\eqref{eq:ubeta} and the relation $G_\beta = \beta^{-d/2} U_\beta^* G U_\beta$. The main technical result of this section is the following.

\begin{prop} \label{prop:perturbation-operatornorms}
Assume $2\leq m <\infty$, $p\in (0,p^*)$, and $r_\mathrm{hc}>0$. 
We have 
	$||\beta^{d/2} (K_\beta - G_\beta)|| = ||G - \beta^{d/2} U_\beta K_\beta U_\beta ^* || \to 0$ as $\beta \to \infty$. 
\end{prop} 

\noindent Before we come to the proof of the proposition, we state a corollary on the principal eigenvalue and eigenfunction. Remember the quantities  $\Lambda_0(\beta)$, $\Lambda_1(\beta)$, $\phi_\beta$ defined before Lemma~\ref{lem:standardtransfer}.  We choose multiplicative constants so that $||\phi_\beta|| =1$. Let $\lambda_j^\mathrm{Gauss}$, $j\in \N_0$, be an enumeration of the eigenvalues of $G$ with $\lambda_0^\mathrm{Gauss} = ||G||$ and 
$$
	\gamma^\mathrm{Gauss} = \max_{j\neq 0}\frac{|\lambda_j^\mathrm{Gauss}|}{\lambda_0^{\mathrm{Gauss} }}. 
$$
\begin{cor} \label{cor:perturbation-spectral} 
	Under the assumptions of Proposition~\ref{prop:perturbation-operatornorms}: 
	Let $\Lambda_0^\mathrm{Gauss}(\beta)$ and $\phi^\mathrm{Gauss}_\beta(x)$ be as in Proposition~\ref{prop:gbetagauss}. Then as  $\beta\to \infty$, 
	$$
		\Lambda_0(\beta)  =  \bigl(1+ o(1)\bigr)\Lambda_0^\mathrm{Gauss}(\beta),\qquad \int_{\R^d} |\phi_\beta(x) - \phi^\mathrm{Gauss}_\beta(x)|^2\dd x\to 0, 
	$$ 
 and 
	$$
		\lim_{\beta \to \infty} \frac{\Lambda_1(\beta)}{\Lambda_0(\beta)}  = \gamma^\mathrm{Gauss} < 1. 
	$$
\end{cor}

\noindent The corollary follows from Proposition~\ref{prop:perturbation-operatornorms} and standard perturbation theory for compact operators~\cite{reed-simon-vol4}. 
The proof of Proposition~\ref{prop:perturbation-operatornorms} builds on several lemmas. First we show that $\mathcal E_\mathrm{bulk}$ is  $C^2$ in a neighborhood of its global minimizer. 

\begin{lemma} \label{lem:ebulk-C2}
	The mapping $\mathcal E_\mathrm{bulk}$ is $C^2$ in some open  neighborhood in $\mathcal D^+$ of the constant sequence $(\ldots,a,a,\ldots)$. 
\end{lemma} 

\begin{proof} 
Note that 
$$ V(z_1, \ldots, z_d) +  W(z_1, \ldots, z_d, z_{d+1}, \ldots, z_{2d})- d e_0 
	= \sum_{i=1}^{d} h(z_{i}, \ldots, z_{d+i}) 
$$
defines a $C^2$ function in a neighborhood of $(a, \ldots, a) \in \R^{d} \times \R^d$ which vanishes for $(z_{1}, \ldots, z_{2d}) = (a, \ldots, a)$. Moreover, using that $(\ldots, a, a, \ldots)$ minimizes $\mathcal{E}_\mathrm{bulk}$ on $\mathcal{D}_0^+$ and so $\partial_{x_j} \mathcal{E}_\mathrm{bulk}(\ldots, a, a, \ldots) = 0$, we see that also 
$$ V_x(a, \ldots, a) + W_x(a, \ldots, a) + W_y(a, \ldots, a) = 0. $$ 
For all $z \in \mathcal{D}_0^+$ the derivative of $\mathcal{E}_\mathrm{bulk}$ at $z$ is given by 
\begin{align*}
  \mathrm D \mathcal{E}_\mathrm{bulk}(z) \zeta 
	= \sum_{j \in\Z} \big( V_x(x_j) + W_x(x_j, x_{j+1}) + W_y(x_{j-1}, x_j) \big) \xi_j,
\end{align*}
for all $\zeta \in \ell^2(\Z)$ with $\zeta_j = 0$ for all but finitely many $j$. So 
\begin{align}\label{eq:D-Ebulk}
  \mathrm D \mathcal{E}_\mathrm{bulk}(z) 
	= \big( V_x(x_j) + W_x(x_j, x_{j+1}) + W_y(x_{j-1}, x_j) \big)_{j \in\Z}.  
\end{align}
Since  
\begin{align*}
  &\sum_{j \in \Z} | V_x(x_j) + W_x(x_j, x_{j+1}) + W_y(x_{j-1}, x_j) 
   - V_x(x'_j) - W_x(x'_j, x'_{j+1}) - W_y(x'_{j-1}, x'_j) |^2 \\  
  &~~\le C \sum_{j \in  \Z} | (x_{j-1}, x_j, x_{j+1}) - (x'_{j-1}, x'_j, x'_{j+1}) |^2 
     \le C || z - z' ||_{\ell^2} 
\end{align*}
for $z, z' \in \mathcal{D}^+$ in a neighborhood of $(\ldots, a, a, \ldots)$ with a uniform constant $C$, the right hand side of \eqref{eq:D-Ebulk} extends to a uniformly continuous function there. Writing 
$$ \mathcal{E}_\mathrm{bulk}(z + \zeta) 
   = \mathcal{E}_\mathrm{bulk}(z) 
     + \int_0^1 \mathrm D \mathcal{E}_\mathrm{bulk} (z + t \zeta) \zeta \, \dd t $$ 
for $z, z' \in \mathcal{D}_0^+$, a standard approximation argument shows that indeed $\mathcal{E}_\mathrm{bulk}$ is $C^1$ in a neighborhood of $(\ldots, a, a, \ldots)$ also in $\mathcal{D}^+$ with $\mathrm D \mathcal{E}_\mathrm{bulk}$ given by \eqref{eq:D-Ebulk}. In fact, $\mathcal{E}_\mathrm{bulk}$ is even $C^2$ on a neighborhood of $(\ldots, a, a, \ldots)$ in $\mathcal{D}^+$ and
\begin{align}\label{eq:DD-Ebulk}
\begin{split}
  \mathrm D^2 \mathcal{E}_\mathrm{bulk}(z) \zeta 
  &= \big( ( V_{xx}(x_j) + W_{xx}(x_j, x_{j+1}) + W_{yy}(x_{j-1}, x_j) ) \xi_j \\ 
  &\quad\quad\quad + W_{xy}(x_j, x_{j+1}) \xi_{j+1} + W_{xy}(x_{j-1}, x_j) \xi_{j-1} \big)_{j \in\Z}.  
\end{split}
\end{align}
This follows similarly as above by extending the derivative of $\mathrm D \mathcal{E}_\mathrm{bulk}$, where we now use that the mappings $\R^d \times \R^d \times \R^d \to \R$, $(x,x',x'') \mapsto V_{xx}(x') + W_{xx}(x', x'') + W_{yy}(x, x')$ and $\R^d \times \R^d \to \R$, $(x,x') \mapsto W_{xy}(x, x')$ are uniformly continuous in a neighborhood of $x = x' = x'' = (a, \ldots, a)$ and so $\mathrm D^2 \mathcal{E}_\mathrm{bulk}$ extends to a continuous mapping from a neighborhood of $(\ldots, a, a, \ldots)$ to $L(\ell^2(\Z))$ (the space of bounded linear operators on $\ell^2(\Z)$) given by \eqref{eq:DD-Ebulk}. 
\end{proof} 

\noindent Next we show that $\widehat M$ is in fact the Hessian of $\widehat H$. 

\begin{lemma} \label{lem:Hmin}
	Assume $2\leq m <\infty$, $p\in [0,p^*)$, and $r_\mathrm{hc}>0$. 
	We have $\widehat H(x,y) \geq \widehat H(\vect a,\vect a) =0$ for all $x,y \in \R_+^d$, moreover as $x,y\to \vect a$, 
	$$
		\widehat H(x,y) = \tfrac12 \widehat{\mathcal Q}(x-\vect a, y-\vect a) + o(|x-\vect a|^2 + |y- \vect a|^2).
	$$
\end{lemma} 

\noindent The lemma leaves open whether $(\vect a,\vect a)$ is the unique global minimizer of $\widehat H$.

\begin{proof} 
The first part of the lemma has already been proven in Lemma~\ref{lem:kelementary}(a). With $M \in \R^{2d \times 2d}$, $N \in \R^{d \times d}$ as in \eqref{eq:Mdef} and \eqref{eq:Ndef} we let $\widehat{M}$ as in \eqref{eq:hatMdef}. It remains to show that $\mathrm D^2 \widehat H(\vect a, \vect a) = \widehat{M}$. 
Since, for a suitable $\eps > 0$, $\mathcal E_{\mathrm{bulk}}$ is convex on $\mathcal{D}^+ \cap [z_{\min}, z_{\max} + \eps]^{\Z}$, see (the proof of) Proposition~\ref{prop:lim-bulk}, Lemma~\ref{lem:minEbulk-quantest} shows that there is a unique function on a neighborhood of $(\vect a, \vect a)$ in $\R^d \times \R^d$ with values in $\R^{-\N} \times \R^{\N}$, $(x,y) \mapsto \tilde{z} = (z_-, z_+) = (z_-(x,y), z_+(x,y))$ such that 
$$ H(x,y) 
   = \mathcal E_{\mathrm{bulk}} (z_-(x,y), x, y, z_+(x,y)). $$ 
As $\mathrm D^2 \mathcal E_{\mathrm{bulk}} (\ldots,a,a,\ldots)$ is positive definite, the implicit function theorem shows that this mapping is $C^1$ and satisfies 
\begin{align*}
  \mathrm D_{\tilde{z}}\mathcal E_{\mathrm{bulk}} (z_-, \cdot, \cdot, z_+) = 0 
\end{align*}
as well as 
\begin{align*}
  \mathrm D_{(x,y)} \tilde{z}
	 = \big( \mathrm D_{\tilde{z}}^2 \mathcal E_{\mathrm{bulk}} (z_-, \cdot, \cdot, z_+) \big)^{-1} \mathrm D_{(x,y)} \mathrm D_{\tilde{z}}\mathcal E_{\mathrm{bulk}} (z_-, \cdot, \cdot, z_+). 
\end{align*}
The latter identity implies 
$$ \mathrm D_{(x,y)} H 
   = \mathrm D_{(x,y)} \mathcal E_{\mathrm{bulk}} (z_-, \cdot, \cdot, z_+), $$ 
so that $H$ is indeed $C^2$ near $(\ldots, a, a, \ldots)$ and 
$$ \mathrm D_{(x,y)}^2 H 
   = \big[ \mathrm D_{(x,y)}^2 \mathcal E_{\mathrm{bulk}} 
	   - \mathrm D_{(x,y)\tilde{z}} \mathcal E_{\mathrm{bulk}} \big( \mathrm D_{\tilde{z}}^2 \mathcal E_{\mathrm{bulk}} \big)^{-1} \mathrm D_{(x,y)\tilde{z}}\mathcal E_{\mathrm{bulk}} \big] (z_-, \cdot, \cdot, z_+). $$ 
In particular, since $\tilde{z}(\vect a, \vect a) = (\ldots, a, a, \ldots)$, 
$$ \mathrm D^2 H(\vect a, \vect a) 
   = \big[ \mathrm D_{(x,y)}^2 \mathcal E_{\mathrm{bulk}} 
	   - \mathrm D_{(x,y)\tilde{z}} \mathcal E_{\mathrm{bulk}} \big( \mathrm D_{\tilde{z}}^2 \mathcal E_{\mathrm{bulk}} \big)^{-1} \mathrm D_{(x,y)\tilde{z}}\mathcal E_{\mathrm{bulk}} \big] (\ldots, a, a, \ldots). $$ 
The same analysis applied to the quadratic approximation $\ell^2(\Z) \to \R$, $z \mapsto \frac{1}{2} \langle z, \mathrm D^2 \mathcal E_{\mathrm{bulk}}(\ldots,a,a,\ldots) z \rangle$ leads to 
$$ M
   = \big[ \mathrm D_{(x,y)}^2 \mathcal E_{\mathrm{bulk}} 
	   - \mathrm D_{(x,y)\tilde{z}} \mathcal E_{\mathrm{bulk}} \big( \mathrm D_{\tilde{z}}^2 \mathcal E_{\mathrm{bulk}} \big)^{-1} \mathrm D_{(x,y)\tilde{z}}\mathcal E_{\mathrm{bulk}} \big] (\ldots, a, a, \ldots), $$ 
too. So we have $\mathrm D^2 H(\vect a, \vect a) = M$. 
A completely analogous reasoning gives $\mathrm D^2 w(a, \ldots, a) = N$ and it follows that $\mathrm D^2 \widehat{H}(\vect a, \vect a) = \widehat{M}$. 
\end{proof} 	
	
\begin{lemma} \label{lem:hhatgrowth}
	Assume $2\leq m <\infty$. 
	For some $c_2>0$ and all $(z_1,\ldots,z_{2d})\in (r_{\mathrm{hc}},\infty)^{2d}$, 
	$$\widehat H\bigl((z_1,\ldots,z_d),(z_{d+1},\ldots,z_{2d})\bigr) \geq \tfrac12 p \sum_{i=1}^{2d} z_i- c_2.$$
\end{lemma} 

\begin{proof}
	Since the pair potential $v$ is bounded from below, we have for some constant $c>0$ 
	$$
		V(z_1,\ldots, z_d) = p \sum_{i=1}^d z_j  - c, \quad \inf_{\R^{2d}} W(x;y) \geq - c. 
	$$
	In combination with Lemma~\ref{lem:hhat} this yields the claim. 
\end{proof} 

\noindent In order to estimate $||K_\beta - G_\beta||$, we split the configuration space into a neighborhood $\mathcal A\supset B_\delta(\vect a)$ of $\vect a$ and its complement $\mathcal B= \R^d\setminus \mathcal A$ and treat blocks separately. For $U\subset \R^d$, we write $\mathbf1_U$ for the multiplication operator with the indicator function $\1_U$. 

\begin{lemma} \label{lem:perturbation1}
	Suppose that $\mathcal A\subset \R^d$ is compact, contains an open neighborhood of $\vect a$, and is such that $\widehat H(x,y)>0$ for all $(x,y) \in \mathcal A\times \mathcal A\setminus \{(\vect a,\vect a)\}$. Then 
	$$
	\lim_{\beta \to \infty} || \mathbf1_{\mathcal A}\,  \beta^{d/2} (K_\beta - G_\beta) \mathbf1_{\mathcal A}|| =0.
	$$
\end{lemma} 

\begin{proof} 
By Lemma~\ref{lem:Hmin}, for every $\eps>0$, there is a $\delta>0$ such that for all $s,t\in \R^d$ with $|s|\leq \delta$ and $|t|\leq \delta$, we have 
\bes
	\tfrac12 (1-\eps) \widehat{\mathcal Q}(s,t)	\leq \widehat H(\vect a+ s, \vect a + t) \leq \tfrac12 (1+\eps) \widehat{\mathcal Q}(s,t)
\ees
Choosing $\delta>0$ small enough we may assume without loss of generality that $B_\delta(\vect a) \subset \mathcal A$. 
We estimate 
\begin{align*}
	\int_{B_\delta(\vect a)^2} \beta^d |K_\beta(x,y) - G_\beta(x,y)|^2 \dd x \dd y &\leq \int_{B_\delta(0)^2} \beta^d \bigl(\e^{\beta \eps \widehat{\mathcal Q}(s,t)} - 1\bigr)^2\e^{-\beta \widehat{\mathcal Q}(s,t)}  \dd s \dd t \\
	&\leq \int_{\R^d}\beta^d \bigl( \e^{- \beta(1-2\eps) \widehat{\mathcal Q}(s,t)} - 2 \e^{-\beta (1-\eps) \widehat{\mathcal Q}(s,t)} + \e^{-\beta \widehat{\mathcal Q}(s,t)} \bigr)\dd s\dd t\\
	& = \Bigl(\frac{1}{(1-2\eps)^d} - \frac{2}{(1-\eps)^d}+1\Bigr) \frac{(2\pi)^d}{\sqrt{\det \widehat M}} \leq k \eps
\end{align*}
for some $k>0$. On $\mathcal A^2\setminus B_\delta(\vect a)^2$, the function $\widehat H$ stays bounded away from $0$, therefore 
$$
	\int_{\mathcal A^2\setminus B_\delta(\vect a)^2} \beta^d |K_\beta(x,y)|^2 \dd x \dd y \leq \e^{- c_\eps \beta}.  
$$
A similar estimate clearly holds true for $G_\beta$ as well. 
Hence 
$$
	\limsup_{\beta\to \infty} \int_{\mathcal A^2} \beta^d |K_\beta(x,y) - G_\beta(x,y)|^2 \dd x \dd y \leq k \eps.
$$
This holds true for every $\eps>0$, so the left-hand side converges to zero. Since operator norms are bounded by Hilbert-Schmidt norms, the lemma follows. 
\end{proof} 

\begin{lemma} \label{lem:perturbation2}
	Assume that $\mathcal B\subset \R^d$ is such that $\mathrm{dist}(\vect a, \mathcal B)>0$ and $\mathcal B$ is invariant under reversals, $\sigma(\mathcal B)=\mathcal B$. Then 
	$||\mathbf1_{\mathcal B} K_\beta \mathbf1_{\mathcal B}|| = O(\e^{-\beta \delta})\to 0$.
\end{lemma} 

\begin{proof}
	We may view $K_\beta^\mathcal B = \mathbf 1_{\mathcal B} K_\beta \mathbf 1_{\mathcal B}$ as an operator in $L^2(\mathcal B,\dd x)$. The Krein-Rutman theorem is applicable and shows that $\lambda = ||K_\beta^\mathcal B||$ is a simple eigenvalue and there exists an eigenfunction $\psi$ that is strictly positive on $\mathcal B\cap (r_{\mathrm{hc}},\infty)^d$.  Because of the symmetry $\widehat H(\sigma y,\sigma x) = \widehat H(x,y)$, the function $\psi\circ \sigma$ is a left eigenfunction. 	
	Moreover for all $f,g\in L^2(\mathcal B,\dd x)$, we have 
	$$
		\lim_{n\to \infty} \frac{1}{\lambda^n} \la f, (K_\beta^\mathcal B)^n g\ra = \la f, \psi\ra \la \psi\circ \sigma, g\ra
	$$
	so for all strictly positive functions $f,g\in L^2(\mathcal B,\dd x)$, 
	$$
		\lambda = \lim_{n\to \infty}\Bigl( \la f, (K_\beta^{\mathcal B})^n g\ra\Bigr)^{1/n}.
	$$
	We choose $f(y)= \exp(- \beta \widehat H(\vect a, y))$ and $g(x) =\exp(- \beta  \widehat H(x,\vect a))$. 
	The scalar product becomes 
	$$
		\la f, (K_\beta^\mathcal B)^n g\ra = \int_{{\mathcal B}	^n}  \e^{ - \beta \sum_{i=0}^{n+1} \widehat H(x_i,x_{i+1}) } \dd x_1\cdots \dd x_{n+1}
	$$
	with $x_0 =x_{n+2} = \vect a$. 
	By Lemma~\ref{lem:hhat}(b) , remembering $u(\vect a) =0$, we have 
	$$
		 \sum_{i=0}^{n+1} \widehat H(x_{i-1},x_i) 
		= - (n+2) d e_0 - V(\vect a) + \sum_{i=0}^{n+1} V(x_i) + \sum_{i=1}^n W(x_i,x_{i+1}). 
	$$
	Define $(z_1,\ldots,z_{(n+1)d} ) = (x_1,\ldots,x_{n+1})$ and for $j\in \Z\setminus \{1,\ldots, (n+1)d\}$, $z_j = a$. Then we recognize 
	$$
		\sum_{i=0}^{n+1} \widehat H(x_{i-1},x_i) = \mathcal E_{\mathrm{bulk}}\bigl( (z_j)_{j\in \Z}\bigr) + \mathrm{const}
	$$
	where the constant depends on $e_0$, $d$, and $V(a)$ alone. As $z_1,\ldots, z_{(n+1)d}$ stay bounded away from $a$, we obtain 
	$$
		\sum_{i=0}^{n+1} \widehat H(x_{i-1},x_i) \geq \delta(n+1) d- c
	$$
	for some $\delta,c>0$ and all $n\in \N$ and $x_1,\ldots,x_{n+1}\in \mathcal B$. 
	It follows that $	||K_\beta^{\mathcal B}|| = \lambda \leq \e^{- \beta \delta}$. 
\end{proof} 

\begin{lemma} \label{lem:perturbation3}
	Suppose that  $\mathcal A\subset \R^d$ and $\mathcal B = \R^d\setminus \mathcal A$ are such that 
	\be \label{eq:quantbell}
		V(x) + W(x,y) - d e_0 + u(y)  \geq u(x)+ \delta
	\ee
	for some $\delta>0$ and all $x\in \mathcal A$, $y\in \mathcal B$. Assume also that $\mathcal A$ is invariant under reversals, $\sigma(\mathcal A) = \mathcal A$. 
	Then 
	$$\lim_{\beta \to \infty} \beta^{d/2} \bigl( ||\mathbf 1_{\mathcal A}  K_\beta \mathbf 1_{\mathcal B}|| + ||\mathbf 1_{\mathcal B} K_\beta \mathbf 1_{\mathcal A}||\bigr) =0. $$
\end{lemma} 

\begin{proof}
		Revisiting the proof of Lemma~\ref{lem:hhat}, we see that 
\be \label{eq:hwdiff}
	H(x,y) - w(x) = V(x) + W(x,y)  - d e_0 + u(y) - u(x).
\ee
	Eqs.~\eqref{eq:hwdiff},~\eqref{eq:positivehat} and~\eqref{eq:quantbell} show that $\widehat H(x,y) \geq \delta/2$ for all $x\in \mathcal A$ and $y\in \mathcal B$. This estimate together with the growth estimate from Lemma~\ref{lem:hhatgrowth} shows 
	\bes
		\limsup_{\beta\to \infty}\frac{1}{\beta}\log \Bigl( \int_{\mathcal A\times \mathcal B} |K_\beta(x,y)|^2\dd x\dd y\Bigr) \leq - \tfrac12 \delta < 0 
	\ees
	hence $||\mathbf 1_\mathcal A K_\beta \mathbf1_\mathcal B||\to 0$. 
	The estimate on $||\mathbf 1_\mathcal BK_\beta \mathbf1_\mathcal A||$ follows from the symmetry $K_\beta(\sigma y,\sigma x) = K_\beta(x,y)$.  
\end{proof}

\begin{proof} [Proof of Proposition~\ref{prop:perturbation-operatornorms}]
	Let $\eps >0$, $\mathcal A_\eps:= [z_{\min}, z_{\max}+\eps]^d$, and 
	$\mathcal B= \R^d\setminus \mathcal A$. The sets $\mathcal A$ and $\mathcal B$ are clearly invariant under reversals, moreover $z_{\min}< a \leq z_{\max}$ by Theorem~\ref{thm:periodic}(b), so $a$ is in the interior of $\mathcal A$ and bounded away from $\mathcal B$. Thus $\mathcal A$ and $\mathcal B$ satisfy the assumptions of  Lemmas~\ref{lem:perturbation1} and~\ref{lem:perturbation2}. By Lemma~\ref{lem:unibell}, they also satisfy the condition~\eqref{eq:quantbell} from Lemma~\ref{lem:perturbation3}. 
	By the triangle inequality, 
	$$
		||K_\beta - G_\beta||\leq ||\mathbf1_\mathcal A (K_\beta - G_\beta)\mathbf1_\mathcal A|| + ||K_\beta - 1_\mathcal A K_\beta \mathbf1_\mathcal A|| + ||G_\beta - \mathbf 1_\mathcal A G_\beta\mathbf1_\mathcal A||. 
	$$
	The first term on the right-hand side, multiplied by $\beta^{d/2}$, goes to zero by Lemma~\ref{lem:perturbation1}. For the second term, we estimate
	$$
		||K_\beta - \mathbf 1_\mathcal A K_\beta \mathbf1_\mathcal A|| \leq || \mathbf 1_\mathcal B K_\beta \mathbf 1_\mathcal B|| + \bigl(|| \mathbf 1_\mathcal A K_\beta \mathbf 1_\mathcal B||+  || \mathbf 1_\mathcal B K_\beta \mathbf 1_\mathcal A||\bigr)
	$$
	and conclude from Lemmas~\ref{lem:perturbation2} and~\ref{lem:perturbation3} that d $\beta^{d/2}||K_\beta - 1_\mathcal A K_\beta \mathbf1_\mathcal A||\to 0$. Bounding Hilbert-Schmidt norms, it is straightforward to check that $||\beta ^{d/2} (G_\beta - \mathbf1_\mathcal AG_\beta \mathbf 1_\mathcal A)||\to 0$ as well, and the proof is complete. 
\end{proof} 

\subsection{Proof of Theorems~\ref{thm:gaussian2}, \ref{thm:gaussian3} and ~\ref{thm:corr-finitem}} 

\begin{proof} [Proof of Theorem \ref{thm:gaussian3}]
Combining Lemma~\ref{lem:standardtransfer}(a) and Corollary~\ref{cor:perturbation-spectral}, we obtain  
\begin{equation*} 
	g(\beta,p) = e_0 - \frac{1}{\beta }\log \sqrt{\frac{2\pi}{\beta (\det C)^{1/d}}} + o(\beta^{-1}). \qedhere
\end{equation*}
\end{proof}

\begin{proof}[Proof of Theorem~\ref{thm:corr-finitem}]
	The  theorem is an immediate consequence of Lemma~\ref{lem:standardtransfer}(c) and Corollary~\ref{cor:perturbation-spectral}. 
\end{proof} 

For the proof of Theorem~\ref{thm:gaussian2}, we first express the marginals of $\mu^\mathrm{Gauss}$ in terms of the matrices $A$ and $B$ from Eq.~\eqref{eq:abdef} and the matrix $C$ from~\eqref{eq:Cdef}.  We group variables in blocks $x_j \in \R^d$ as usual and view $\mu^\mathrm{Gauss}$ as a measure on $(\R^d)^\Z$. 

\begin{prop} \label{prop:gaussmarginals}
	Under the assumptions of Theorem~\ref{thm:gaussian2}, the distributions of $x_0 = (z_0,\ldots, z_{d-1})$, $(x_0,x_1)$, and $(x_0,\ldots, x_n)$ ($n\geq 2$) under $\mu^\mathrm{Gauss}$ have probability density functions proportional to 
	\begin{enumerate} 
		\item [(a)] $\exp( - \frac12 \beta\la x_0, (\sigma C \sigma - BC^{-1} B^T) x_0\ra )$, 
		\item [(b)] $\exp( - \tfrac12 \beta  [\la \sigma x_0, C \sigma x_0\ra - 2 \la x_0, B x_1\ra + \la x_1, C x_1\ra])$, 
		\item [(c)] $\exp(- \tfrac12 ( \la \sigma x_0, (C - \tfrac12 A) \sigma x_0\ra + \sum_{i=0}^{n-1} \mathcal Q(x_i,x_{i+1}) + \la x_{n-1}, (C - \tfrac12 A) x_{n-1}\ra ))$
	\end{enumerate} 
	respectively. 
\end{prop} 

\begin{proof} 
	We recall a standard fact on marginals of multivariate Gaussians and Schur complements. Suppose we are given a positive-definite $(n+k)\times (n+k)$-matrix in block form 
	$$ \mathcal H = \begin{pmatrix}
			\mathcal H_1  & \mathcal H_2 \\ \mathcal H_2^T &  \mathcal H_3 
		\end{pmatrix}
	$$
	where $\mathcal H_1,\mathcal H_2,\mathcal H_3$ are $n\times n$, $n\times k$ and $k\times k$ matrices, respectively.  Think of $\mathcal H$ as the Hessian of the energy. Consider the Gaussian measure on $\R^{n+k}$ with covariance matrix $\mathcal H^{-1}$ and probability density function
	$$
		\rho (x,y) = \sqrt{ \frac{\det \mathcal H} {(2\pi)^{(n+k) } }}\, \exp\Bigl( - \frac12 \la \begin{pmatrix} x\\ y\end{pmatrix} , \mathcal H \begin{pmatrix} x \\y \end{pmatrix}\ra\Bigr) \qquad (x\in \R^{n}, y\in \R^k). 
	$$
	Then for all $x\in \R^n$, 
	\be \label{eq:gaussmarginal}
		\int_{\R^{k}} \rho(x, y) \dd y = \sqrt{ \frac{\det \mathcal M} {(2\pi)^{n } }}\, \exp\Bigl( - \frac12 \la x, \mathcal M x\ra\Bigr) 
	\ee
	with $\mathcal M = \mathcal H_1 - \mathcal H_2 \mathcal H_3^{-1} \mathcal H_2^T$ the Schur complement of $\mathcal H_3$ in $\mathcal H$. The inverse $\mathcal M^{-1}$ is equal to the upper left block of $\mathcal H^{-1}$. Another characterization is provided by a completion of squares, similar to the proof of Lemma~\ref{lem:gaussianansatz}: we have 
	$$  \la x, \mathcal M x\ra = \inf_{y\in \R^k} \la \begin{pmatrix} x \\ y\end{pmatrix} , \mathcal H \begin{pmatrix} x \\ y\end{pmatrix}\ra.
	$$ 
	Now let $\mathcal H = (\mathcal H_{ij})_{i,j\in \Z}$ be the Hessian of $\mathcal E_\mathrm{bulk}$ at $(\ldots, a,a,\ldots)$. By definition of $\mu^\mathrm{Gauss}$, the distribution of $(z_1,\ldots,z_n)$ is Gaussian with mean zero and covariance matrix $(\mathcal H^{-1})_{i,j=1,\ldots, n}$. Let $\mathcal M= (\mathcal M_{ij})_{0\leq i,j\leq n-1}$ be the $n\times n$-matrix defined by $\mathcal M ^{-1}= (\mathcal H^{-1})_{0\leq i,j\leq n-1}$. 	It is not difficult to check that the considerations above generalize to the infinite matrices at hand, hence for all $z_0,\ldots, z_{n-1} \in \R$, 
	\be \label{eq:dvar}
		\sum_{i,j=0}^{n-1} \mathcal M_{ij} z_i z_j = \inf \Bigl \{ \sum_{i,j\in \Z} \mathcal H_{ij} z'_i z'_j \, \Big|\, (z'_j)_{j\in \Z}\in \ell^2(\Z):\ z'_0 = z_0,\ldots, z'_{n-1} = z_{n-1} \Bigr \}. 
	\ee
	Eq.~\eqref{eq:dvar} provides a variational description of the covariance matrix $\mathcal M^{-1}$ of the $n$-dimensional marginal of $\mu^\mathrm{Gauss}$. 
	For $n= 2d = 2(m-1)$, with $x_0=(z_0,\ldots,z_{d-1})$ and $x_1 = (z_d,\ldots, z_{2d-1})$,  Eq.~\eqref{eq:dvar} shows $\mathcal M = M$, by the definition~\eqref{eq:Mdef} of $M$. Combining with~\eqref{eq:Mblocks} we get 
	$$
		\mathcal M = \begin{pmatrix} \sigma C \sigma & - B \\ - B^T & C \end{pmatrix} = M.
	$$
	This proves part (b) of the lemma. The proof of (c)  is similar. Part (a) follows from (b) and a relation similar to~\eqref{eq:gaussmarginal}. 	
\end{proof}

\begin{proof}[Proof of Theorem~\ref{thm:gaussian2}]
	It is enough to treat the $nd$-dimensional marginals with $n\geq 2$. 	
	Let $\phi_\beta$ be the principal eigenfunction of $K_\beta$, with multiplicative constant chosen so that $\la \phi_\beta \circ \sigma, \phi_\beta \ra=1$. Set  $\tilde \phi_\beta (x):= U_\beta \phi_\beta(x) =\beta^{-1/4} \phi_\beta(\vect a+\beta^{-1/2}x)$ and
	$$
		\tilde K_\beta(x,y):= \frac{1}{\Lambda_0(\beta)} \bigl(U_\beta K_\beta U_\beta^*\bigr)(x,y) = \frac{1}{\Lambda_0(\beta)} K_\beta(\vect a +\beta^{-1/2}x,\vect a +\beta^{-1/2}y)
	$$
	By Lemma~\ref{lem:standardtransfer}, the probability density $\rho_{nd}^\ssup{\beta}$ for $(x_1,\ldots,x_n)\in \R^{nd}$ satisfies 
	$$
		\tilde \rho_{nd}^{(\beta)} (x_1,\ldots,x_n)= \beta^{-nd/2}	\rho_{nd}^\ssup{\beta} (\vect a +\beta^{-1/2} x_1,\ldots, \vect a +\beta^{-1/2} x_n) \\
		= \tilde \phi_\beta(\sigma x_1) \Biggl( \prod_{i=1}^{n-1} \tilde K_\beta (x_i,x_{i+1}) \Biggr)\tilde \phi_\beta(x_n). 
	$$
	By Proposition~\ref{prop:gaussmarginals}, the analogous representation for the Gaussian density $\rho_{nd}^\mathrm{Gauss}$ is 
	$$
		\rho_{nd}^\mathrm{Gauss}(x_1,\ldots,x_n) = \phi^\mathrm{Gauss}(\sigma x_1) \Biggl(\prod_{i=1}^{n-1} \tilde G (x_i,x_{i+1})  \Biggr)\phi^\mathrm{Gauss}(x_n)
	$$
	with $\tilde G(x,y) = (\lambda_0^\mathrm{Gauss}) G(x,y)$ and 	
	$\phi^\mathrm{Gauss}(x) \propto \exp(-\frac12 \la x, \frac12 N x \ra)$ the principal eigenfunction of $G$, normalized so that $\la \phi^\mathrm{Gauss}\circ \sigma, \phi^\mathrm{Gauss}\ra =1$. It follows that 
	\begin{align*}
			&\int_{\R^{nd}}\bigl|\tilde \rho_{nd}^{(\beta)} (x_1,\ldots,x_n) -\rho_{nd}^\mathrm{Gauss}(x_1,\ldots,x_n)	\bigr|\dd x_1\ldots \dd x_n  \\
			&\quad \leq \bigl| \la \tilde \phi_\beta \circ \sigma- \phi^\mathrm{Gauss}\circ \sigma, \tilde K_\beta^{n-1} \tilde \phi_\beta \ra \bigr|
			 + \sum_{i=1}^{n-1} \bigl| \la \phi^\mathrm{Gauss}\circ \sigma, \tilde  G^i (\tilde K_\beta - \tilde G)	{\tilde K_\beta}^{n-i-2} \tilde \phi_\beta \ra\bigl| \\
			 &\qquad \qquad + \bigl|\la \phi^\mathrm{Gauss}\circ \sigma,  \tilde G^{n-1} (\tilde \phi_\beta - \phi^\mathrm{Gauss}\ra\bigr|.
	\end{align*}
	Using $\tilde K_\beta \tilde \phi_\beta = \tilde \phi_\beta$ and $\tilde G^* (\phi^\mathrm{Gauss}\circ \sigma) = \phi^\mathrm{Gauss}\circ \sigma$, we get 
	$$
		||\rho_{(n+1)d}^{(\beta)}- \rho_{(n+1)d}^\mathrm{Gauss}||_{L^1} \leq \bigl( ||\tilde \phi_\beta||_{L^2} + ||\phi^\mathrm{Gauss}||_{L^2}\bigr) ||\tilde \phi_\beta - \phi^\mathrm{Gauss}||_{L^2}  + ||\tilde K_\beta - \tilde G||
	$$
	which goes to zero by Proposition~\ref{prop:perturbation-operatornorms} (see also Corollary~\ref{cor:perturbation-spectral}).
\end{proof}


\section{A Brascamp-Lieb type covariance estimate for $m=\infty$}
\label{sec:brascamp}

Here we prove Proposition~\ref{prop:restrict-alg}.  Key to the proof is a matrix lower bound $A$ for the Hessian of $\mathcal{E}_N$. For Gaussian measures with probability density proportional to $\exp(- \frac{\beta}{2} \la z, A z\ra)$ and test functions $f_i = z_i$, $g_j = z_j$, we end up estimating the covariance $C_{ij} = ([\beta A]^{-1})_{ij}$. We follow~\cite{menz14}, see also~\cite{otto-reznikoff07}.

\begin{proof} [Proof of Proposition~\ref{prop:restrict-alg}]
Revisiting the proof of Lemma~\ref{lem:hessian}, we obtain bounds on matrix elements of the Hessian. Let $N\in \N$, $z\in [z_{\mathrm{min}}, z_\mathrm{max}]^{N-1}$. For $1 \leq i< j \leq N-1$ we have 
\begin{align*}
	 0 \geq \partial_i \partial_j \mathcal{E}_N(z)   & =   \sum_{L\supset \{i,j\}} v''(\sum_{k\in L} z_k) 
	 	  \geq \sum_{n=j-i+1}^{N-1} v''(n z_\mathrm{min}) \#\{L \mid \#L =n,\, L \supset \{i,j\}\} \\
	 	 & \geq \sum_{n=j-i+1}^\infty (n-j+i) v''(nz_\mathrm{min}) =: - \kappa_{j-i}
\end{align*}
with 
\be\label{eq:kappaj}
	0  \leq \kappa_{j-i} \leq \sum_{n=j-i+1}^\infty \frac{\alpha_2 n}{ (nz_\mathrm{min})^{s+2}} 
	     \leq 
	      \frac{\alpha_2}{s z_\mathrm{min}^{s+2} (j-i)^{s}}
\ee
For $1\leq i \leq N-1$ we also have 
\bes
	\partial_i^2 \mathcal{E}_N(z)  = \sum_{L \ni i} v''(\sum_{k\in L} z_k) \geq v''(z_\mathrm{max}) - \sum_{n=2}^{\infty} n \bigl| v''(n z_\mathrm{min}) \bigr| =:\rho >0
\ees
by Assumption~\ref{assu:v}(iv). Moreover 
\bes
	\eta := \rho - 2 \sum_{\ell=1}^\infty \kappa_\ell = v''(z_{\max}) - \sum_{n=2}^\infty n^2 |v''(n z_{\min})|>0
\ees
again by Assumption~\ref{assu:v}(iv). Let $A_N$ be the $(N-1)\times (N-1)$-matrix with diagonal $\rho$ and off-diagonal entries $-\kappa_{|j-i|}$; notice that $\eta,\kappa_{j-i}, \rho$ do not depend on $N$. $A_N$ is symmetric and positive-definite. 
  
The previous estimates together with \cite[Remark 2.6]{menz14} show that the energy $\mathcal{E}_N$ satisfies the assumptions of~\cite[Theorem 2.3 and Proposition 3.5]{menz14}.   
It follows that for all smooth $f,g:\R_+ \to \R$, 
\bes
	\Bigl|\tilde \mu_\beta^\ssup{N} (f_i g_j ) 
 - \tilde \mu_\beta^\ssup{N} (f_i) 	\tilde \mu_\beta^\ssup{N} (g_j) \Bigr|
  	\leq  \frac{1}{\beta} (A_N^{-1})_{ij} \Bigl( \tilde \mu_\beta^\ssup{N} \bigl({f'_i}^2 \bigr) \tilde \mu_\beta^\ssup{N} \bigl({g'_j}^2 \bigr) \Bigr)^{1/2}.
\ees
Let $X_1,X_2,\ldots$ be i.i.d. random variables with law 
\bes
	\P(X_i  = \ell) = \frac{\kappa_{|\ell|}}{\rho - \eta},\quad \ell \in \Z\setminus \{0\}, \quad \P(X_i  = \ell) =0
\ees
and $S_n = X_1 + \cdots + X_n$. 
We may decompose $A_N$ as $\rho \mathrm{Id}$ plus an off-diagonal matrix, write a Neumann series for the inverse, and find that for $i<j$
\be \label{eq:aninv}
	(A_N^{-1})_{ij} \leq \frac{1}{\rho} \sum_{k=1}^\infty \bigl( 1- \frac{\eta}{\rho}\bigr)^{k} \P(S_k = j-i).
\ee
Clearly 
\be \label{eq:psk}
	\P(S_k = j-i) \leq \sum_{r=1}^k \P(X_r \geq (j-i)/k,\, S_k = j-i ).
\ee 
By~\eqref{eq:kappaj}, we have $\P(X_r = \ell) \leq C/|\ell|^s$ for some constant $C>0$. Following~\cite[Proposition 3.5]{menz14} we may estimate, for each $m\in \N$, 
\begin{align*}
	\P(X_2 \geq m,\, S_k = j-i )  & \leq \sum_{\ell =m}^\infty \P(X_2 = \ell) \P(X_1 + X_3+\cdots + X_n = j-i-\ell) \\ 
	& \leq \sup_{\ell \geq m} \P(X_2 = \ell) \leq \frac{C}{m^s}. 
\end{align*}
Similar estimates apply to other $r$. Combining with~\eqref{eq:psk} we find 
\bes 
	\P(S_k = j-i) \leq \frac{C\, k^{s+1}}{|j-i|^{s}}.
\ees
It follows that 
\begin{align*}
	(A_N^{-1})_{ij} & \leq \frac{C}{\rho |i-j|^s} \sum_{k=1}^{\infty} k^{s+1} \bigl(1-\frac{\eta}{\rho}\bigr)^{k}
\end{align*}
Notice that the series is convergent. The bound is plugged into the estimate~\eqref{eq:aninv} and the proposition follows by passing to the limit $N\to \infty$.
\end{proof}

\subsubsection*{Acknowledgments} 

We thank Nils Berglund, Andrew Duncan,  Andr{\'e} Schlichting,  and Martin Slowik for helpful discussions.

\bibliographystyle{amsalpha}
\bibliography{lenjon}

\end{document}